\documentclass[12pt]{article}
\usepackage{bbm}
\usepackage{url}
\usepackage{times}
\usepackage[colorlinks]{hyperref}
\usepackage{framed}
\usepackage[normalem]{ulem}

\usepackage[utf8]{inputenc}
\usepackage{scalerel}
\usepackage{enumerate}
\usepackage{authblk}
\usepackage{wrapfig}
\usepackage{tikz-cd}
\usepackage[new]{old-arrows}
\usepackage{amsmath,amscd}
\usepackage{amsthm}
\usepackage{amsfonts}
\usepackage{color}
\usepackage{thmtools}
\usepackage{thm-restate}
\usepackage{amssymb}
\usepackage[margin=1in]{geometry}

\newtheorem{theorem}{Theorem}
\newtheorem{proposition}{Proposition}[section]
\newtheorem{corollary}{Corollary}[section]
\newtheorem{lemma}{Lemma}[section]
\newtheorem{Remark}{Remark}[section]
\newtheorem{claim}{Claim}[section]

\newtheorem{conjecture}{Conjecture}

\newtheorem{innercustomthm}{Question}
\newenvironment{customquestion}[1]
{\renewcommand\theinnercustomthm{#1}\innercustomthm}
{\endinnercustomthm}

\theoremstyle{definition}
\newtheorem{definition}{Definition}
\newtheorem{example}{Example}[section]
\newcommand{\Z}{\mathbb{Z}}

\newcommand{\R}{\mathbb{R}}

\newcommand{\Sp}{\mathbb{S}}
\newcommand{\diam}{\mathrm{diam}}
\newcommand{\dgh}{d_\mathrm{GH}}

\newcommand{\filrad}{\mathrm{FillRad}}
\newcommand{\sfilrad}{\mathrm{sFillRad}}
\newcommand{\hyp}{\mathrm{hyp}}
\newcommand{\rad}{\mathrm{rad}}
\newcommand{\spread}{\mathrm{spread}}

\newcommand{\dgm}{\mathrm{barc}}
\newcommand{\dgmR}{\dgm^\mathrm{VR}}

\newcommand{\Hom}{\mathrm{H}}
\newcommand{\PH}{\mathrm{PH}}
\newcommand{\vr}{\mathrm{VR}}
\newcommand{\per}{\mathrm{B_*}}
\newcommand{\met}{\mathrm{Met}}
\newcommand{\pmet}{\mathrm{PMet}}

\newcommand{\htop}{\mathrm{hTop}_*}

\newcommand{\dhi}{d_\mathrm{HI}}
\newcommand{\di}{d_\mathrm{I}}

\definecolor{darkblue}{rgb}{0.0, 0.0, 0.8}
\definecolor{darkred}{rgb}{0.8, 0.0, 0.0}
\definecolor{darkgreen}{rgb}{0.0, 0.8, 0.0}

\begin{document}
	\title{Vietoris-Rips Persistent Homology, Injective Metric Spaces, and The Filling Radius}

	\author[1]{Sunhyuk Lim}
	\author[2]{Facundo M\'emoli}
	\author[3]{Osman Berat Okutan}
	
	\affil[1]{Department of Mathematics,
		The Ohio State University\\
		\texttt{lim.991@osu.edu}}
	\affil[2]{Department of Mathematics and Department of Computer Science and Engineering,
		The Ohio State University\\ 	\texttt{memoli@math.osu.edu}}

	\affil[3]{Department of Mathematics,
		Florida State University\\
		
		\texttt{okutan@math.fsu.edu}}
	
	\date{\today}
	\maketitle

	\begin{abstract}
		In the applied algebraic topology community, the persistent homology induced by the Vietoris-Rips simplicial filtration is a standard method for capturing topological information from metric spaces. In this paper, we consider a different, more geometric way of generating persistent homology of metric spaces which arises by first embedding a given metric space into a larger space and then considering thickenings of the original space inside this ambient metric space. In the course of doing this, we  construct an appropriate category for studying this notion of persistent homology and  show that, in a category theoretic sense, the standard persistent homology of the Vietoris-Rips filtration is isomorphic to our geometric persistent homology provided that the ambient metric space satisfies a property called injectivity.  
		
		As an application of this isomorphism result we are able to precisely characterize the type of intervals that appear in the persistence barcodes of the Vietoris-Rips filtration  of any compact metric space and also 
		to give succinct proofs of the characterization of the persistent homology of products and metric gluings of metric spaces. Our results also permit proving several bounds on the length of intervals in the Vietoris-Rips barcode by other metric invariants, for example the notion of spread introduced by M. Katz. 
		
		As another application, we  connect this geometric persistent homology to the notion of filling radius of manifolds introduced by Gromov \cite{G07} and show some consequences related to (1) the homotopy type of the Vietoris-Rips complexes of spheres which follow from work of M.~Katz and (2) characterization (rigidity) results for spheres in terms of their Vietoris-Rips persistence barcodes which follow from work of F.~Wilhelm. 
		
		Finally, we establish a sharp version of Hausmann's theorem for spheres which may be of independent interest.
	\end{abstract}

	\newpage
	\tableofcontents
	
	\section{Introduction}
	
	The simplicial complex nowadays referred to as the Vietoris-Rips complex was originally introduced by Leopold Vietoris in the early 1900s in order to build a homology theory for metric spaces \cite{vietoris}. Later, Eliyahu Rips and Mikhail Gromov \cite{ghg} both utilized the Vietoris-Rips complex in their study of hyperbolic  groups.

	Given a metric space $(X,d_X)$ and $r>0$ the $r$-Vietoris-Rips complex $\mathrm{VR}_r(X)$  has $X$ as its vertex set, and simplices are all those nonempty finite subsets of $X$ whose diameter is strictly less than $r$. In \cite{h95}, Hausmann showed that the Vietoris-Rips complex can be used to recover the homotopy type of a Riemannian manifold $M$. More precisely, he introduced a quantity $r(M)$ (a certain variant of the injectivity radius), and proved that $\vr_r(M)$ is homotopy equivalent to $M$ for any $r\in(0,r(M))$.

	Since $\mathrm{VR}_r(X)\subseteq \mathrm{VR}_s(X)$ for all $0<r\leq s$, this construction then naturally induces the so called  Vietoris-Rips simplicial filtration of $X$, denoted by  $\mathrm{VR}_*(X) = \big(\mathrm{VR}_r(X)\big)_{r>0}$.  By applying the simplicial homology functor  (with coefficients in a given field) one obtains a \emph{persistence module}: a directed system $V_*=\big(V_r\xrightarrow{v_{rs}} V_s\big)_{r\leq s}$ of vector spaces and linear maps (induced by the simplicial inclusions). The persistent module obtained from $\vr_*(X)$ is referred to as the Vietoris-Rips persistent homology of $X$.  
	
	The notion of \emph{persistent homology} arose from work  by Frosini,  Ferri, and Landi \cite{f90,f93}, Robins \cite{V99}, and Edelsbrunner \cite{JE95,HLZ00} and  collaborators. After that, considering the persistent homology of the simplicial filtration induced from  Vietoris-Rips complexes was a natural next step. For example,  Carlsson and de Silva \cite{cds} applied Vietoris-Rips persistent homology to topological estimation from point cloud data, and Ghrist and de Silva applied  it to sensor networks \cite{DG07}. Its efficient computation has been addressed by Bauer in \cite{bauer}. A more detailed historical survey and review of general ideas related to persistent homology can be found in \cite{G09,HH08,HH10}.

	The persistent homology of the Vietoris-Rips filtration of a metric space provides a functorial way\footnote{Where for metric spaces $X$ and $Y$ morphisms are given by  $1$-Lipschitz maps $\phi:X\rightarrow Y$, and for persistence modules $V_*$ and $W_*$ morphisms are systems of linear maps $\nu_* = (\nu_r:V_r\rightarrow W_r)_{r>0}$ making all squares commute.} of assigning a persistence module to a metric space. Persistence modules are usually represented, up to isomorphism, as \emph{barcodes}: multisets of intervals each representing the lifetime of a homological feature. In  this paper,  barcodes are associated to Vietoris-Rips filtrations, and these barcodes will be denoted by $\dgmR_*(\cdot)$. In the areas of topological data analysis (TDA) and computational topology, this type of persistent homology is a widely used tool for capturing topological properties of a dataset  \cite{bauer,cds,DG07}.

	Despite its widespread use in applications, little is known in terms of relationships between Vietoris-Rips barcodes and other metric invariants. For instance, whereas it is obvious that the right endpoint of any interval $I$ in $\dgmR_\ast(X)$ must be bounded above by the diameter of $X$, there has been little progress in relating the length of bars to other invariants such as volume (or Hausdorff measure) or curvature (whenever defined).

	\paragraph{Contributions.}  One main contribution of this paper is establishing a precise relationship (i.e. a filtered homotopy equivalence) between the Vietoris-Rips simplicial filtration of a metric space and  a more geometric (or extrinsic) way of assigning a persistence module to a metric space, which consists of first isometrically embedding it into a larger space and then  considering the persistent homology of the filtration obtained by considering the resulting system of nested neighborhoods of the original space inside this ambient space. These neighborhoods, being also metric (and thus topological) spaces, permit giving a short proof of the K\"unneth formula for Vietoris-Rips persistent homology. 
	
	A particularly nice ambient space inside which one can isometrically embed any given compact metric space $(X,d_X)$ is $L^\infty(X)$; the Banach space consists of all the bounded real-valued functions on $X$, together with the $\ell^\infty$-norm. The embedding is given by $X\ni x\mapsto  d_X(x,\cdot)$: it is indeed immediate that this embedding is isometric since $\|d_X(x,\cdot)-d_X(x',\cdot)\|_{\infty} = d_X(x,x')$ for all $x,x'\in X$. This is usually called the \emph{Kuratowski} isometric embedding of $X$.
	
	That the Vietoris-Rips filtration of a \emph{finite} metric space produces persistence modules isomorphic to the sublevel set filtration of the distance function $$\mbox{$\delta_X:L^\infty(X)\rightarrow \mathbb{R}_{\geq 0}$, $L^\infty(X)\ni f \mapsto \inf_{x\in X}\|d_X(x,\cdot)-f\|_{\infty}$}$$ was already used in \cite{dgh-rips} in order to establish the Gromov-Hausdorff stability of Vietoris-Rips persistence of finite metric spaces.
	
	In this paper we significantly generalize this point of view  by proving an isomorphism theorem between the Vietoris-Rips filtration of \emph{any} compact metric space $X$  and its \emph{Kuratowski filtration}:  $$\big(\delta_{X}^{-1}([0,r))\big)_{r>0},$$
	a fact which immediately implies that their persistent homologies are isomorphic.
	
	We do so by constructing  a filtered homotopy equivalence between the Vietoris-Rips filtration and the sublevel set filtration induced by $\delta_X$. Furthermore, we prove that $L^\infty(X)$ above can be replaced with  \emph{any injective (or equivalently, hyperconvex) metric space}   \cite{dress-book,l13} admitting an isometric embedding of $X$:

	\begin{restatable*}[Isomorphism Theorem]{theorem}{isotheorem}\label{theorem:isom}
		Let $\eta: \met \to \pmet$ be a metric homotopy pairing (for example the Kuratowski functor). Then $\per \circ \eta: \met \to \htop$ is naturally isomorphic to $\vr_{2*}$.
	\end{restatable*}

	Above, $\met$ is the category of compact metric spaces with 1-Lipschitz maps, $\pmet$ is the category of metric pairs $(X,E)$ where $X\hookrightarrow E$ isometrically,  $E$ is an injective metric space,  a metric homotopy pairing is any right adjoint to the forgetful functor (e.g. the Kuratowski embedding), and $\mathrm{B}_\ast$ is the functor sending a pair $(X,E)$ to the filtration $\big(B_r(X,E)\big)_{r>0}$; see Sections \ref{sec:pairs} and  \ref{sec:isom}.
	One useful consequence of Theorem \ref{theorem:isom} is given by Theorem \ref{theorem:pbarcode} which establishes the existence of Vietoris-Rips barcodes for arbitrary totally bounded metric spaces (a fact that was seemingly overlooked in prior literature).
	
	A certain well known construction which involves the isometric embedding $X\hookrightarrow L^\infty(X)$ is that of the \emph{filling radius} of a Riemannian manifold \cite{gfilling} defined by Gromov in the early 1980s. In that construction, given an $n$-dimensional Riemannian manifold $M$ one studies for each $r>0$ the inclusion $$\iota_r: M\hookrightarrow \delta_M^{-1}([0,r)),$$ and seeks the infimal $r>0$ such that the map induced by $\iota_r$ at $n$-th homology level annihilates the fundamental class $[M]$ of $M$. This infimal value defines $\filrad(M)$, the filling radius of $M$. 
	
	Via our isomorphism theorem we are able prove that there always exists a bar in the barcode of a manifold whose length is exactly twice its filling radius:
	
	\begin{restatable*}[]{proposition}{propfrfdgm}\label{prop:filradpersistence}
		Let $M$ be a closed connected $n$-dimensional Riemannian manifold. Then, $$(0,2\,\filrad(M;\mathbb{F})]\in\dgm_n^\mathrm{VR}(M;\mathbb{F}),$$
		where $\mathbb{F}$ is an arbitrary field if $M$ is orientable, and $\mathbb{F}=\Z_2$ if $M$ is non-orientable. Moreover, this is the \emph{unique} interval in $\dgm_n^\mathrm{VR}(M;\mathbb{F})$ starting at $0$ and $\filrad(M;\mathbb{F})\leq\filrad(M)$ when $M$ is orientable.
	\end{restatable*}
	
	As a step in his proof of the celebrated systolic inequality, Gromov proved in \cite{gfilling} that  the filling radius satisfies $\filrad(M)\leq c_n \big(\mathrm{vol}(M)\big)^{1/n}$ for any $n$-dimensional complete manifold $M$ (where $c_n$ is a universal constant, and Nabutovsky recently proved that $c_n$ can be improved to $\frac{n}{2}$ \cite[Theorem 1.2]{nabutovsky2019linear}). This immediately yields a relationship between $\dgmR_*(M)$ and the volume of  $M$.  The fact that the filling radius has already been connected to a number of other metric  invariants also permits importing these results to the setting of Vietoris-Rips barcodes (see Section \ref{sec:bound}). This in turn permits relating  $\dgmR_\ast(M)$ with other metric invariants of $M$, a research thread which has remained mostly unexplored. See Proposition \ref{prop:genfilradpersistence} for a certain generalization of Proposition \ref{prop:filradpersistence} to ANR spaces.

	In a series of papers \cite{k83,katz-s2,katz-s1,katz-cpn} M. Katz studied both the problem of computing the filling radius of spheres (endowed with the geodesic distance) and complex projective spaces, and the problem of understanding the change in homotopy type of $\delta_X^{-1}([0,r))$ when $X\in\{\mathbb{S}^1,\mathbb{S}^2\}$ as $r$ increases.
	
	Of central interest in topological data analysis has been the question of providing a complete characterization of the Vietoris-Rips  persistence barcodes of spheres of different dimensions. Despite the existence of a complete answer to the question for the case of $\mathbb{S}^1$ \cite{adams} due to Adams and Adamaszek, relatively little is known for higher dimensional spheres. In \cite{AAF17} the authors consider a variant of the Vietoris-Rips filtration, which they call Vietoris-Rips metric thickening. The authors are able to obtain information about the succesive homotopy types of this filtration on spheres of different dimension (see Section 5 of \cite{AAF17}) for a certain range of values of the scale parameter. 
	
	The authors of \cite{AAF17} conjecture that the open Vietoris-Rips filtration (which is the one considered in the present paper) is filtered homotopy equivalent to their open Vietoris-Rips metric thickening filtration (as a consequence their persistent homologies  are isomorphic). This isomorphism was conjectured in Conjecture 6.12 of \cite{AAF17} which was recently settled in Corollary 5.10 of \cite{adams2021persistent}.

	Our isomorphism theorem (Theorem \ref{theorem:isom}) permits applying Katz's results in order to provide partial answers to the questions mentioned above and also to elucidate other properties of the standard open Vietoris-Rips filtration and its associated persistence barcodes $\dgmR_\ast(\cdot)$. In addition to these results derived from our isomorphism theorem, in Appendix \ref{sec:app:SnHausmann}, we refine certain key lemmas used in the original proof of Hausmann's theorem \cite{h95} and establish the homotopy equivalence between $\vr_r(\Sp^n)$ and $\Sp^n$ for any $r\in\left(0,\arccos{\left(-\frac{1}{n+1}\right)}\right]$: 
	
	\begin{restatable*}[]{theorem}{thmSnfirsttype}\label{cor:Snfirsttype}
		For any $n\in\mathbb{Z}_{>0}$, we have $\mathrm{VR}_{r}(\mathbb{S}^n)\simeq \mathbb{S}^n$ for any $r\in\left(0,\arccos\left(\frac{-1}{n+1}\right)\right]$.
	\end{restatable*}
	
	Note that this is indeed an improvement since, for spheres, Hausmann's quantity satisfies $r(\Sp^n)=\frac{\pi}{2}<\arccos{\left(-\frac{1}{n+1}\right)}$. This improvement is obtained with the aid of a refined version of Jung's theorem (cf. Theorem \ref{thm:Jung}) which we also establish. Theorem \ref{cor:Snfirsttype} also improves upon \cite[Remark p.508]{k83}, see discussion in Section \ref{sec:spheres-geod}.
	
	In the direction of characterizing the Vietoris-Rips barcodes of spheres, we are able to provide a complete characterization of the homotopy types of the Vietoris-Rips complexes of round spheres $\mathbb{S}^{n-1}\subset \R^n$ endowed with the (restriction of the) $\ell^\infty$-metric, which we denote by $\mathbb{S}^{n-1}_\infty$. Two critical observations are  that (1) the $r$-thickening of $\mathbb{S}^{n-1}_\infty$ inside of $\R^n_\infty$ ($\R^n$ equipped with the $\ell^\infty$-metric) is homotopy equivalent to the $r$-thickening of $\mathbb{S}^{n-1}_\infty$ inside of $\mathbb{D}^n_\infty$ ($n$-dimensional unit ball with $\ell^\infty$-metric), and (2) that it is easier to find the precise shape of the latter:
	
	\begin{restatable*}[]{theorem}{thmsphinf}\label{thm:ncircletightspan}
		For any $n\in\mathbb{Z}_{>0}$ and $r>0$,
		$$B_r(\mathbb{S}^{n-1}_\infty,\R^n_\infty)\simeq B_r(\mathbb{S}^{n-1}_\infty,\mathbb{D}^n_\infty)=\mathbb{D}^n_\infty\backslash V_{n,r}$$
		where
		$$V_{n,r}:=\bigcap_{(p_1,\dots,p_n)\in\{r,-r\}^n}\left\{(x_1,\dots,x_n)\in\mathbb{R}^n:\sum_{i=1}^n (x_i-p_i)^2\leq 1\right\}.$$
		
		In particular, for $r>\frac{1}{\sqrt{n}}$ we have $V_{n,r}=\emptyset$ so that $B_r(\mathbb{S}^{n-1}_\infty,\mathbb{D}^n_\infty)=\mathbb{D}^n_\infty$. As a result, $B_r(\mathbb{S}^{n-1}_\infty,\R^n_\infty)$ is homotopy equivalent to $\mathbb{S}^{n-1}$ for $r\in\left(0,\frac{1}{\sqrt{n}}\right]$ and contractible for $r>\frac{1}{\sqrt{n}}$ (See Figure \ref{fig:d2-s1-tight-span} for an illustration for the case when $n=2$).
	\end{restatable*}
	
	\begin{figure}
		\centering
		\includegraphics[width=0.5\linewidth]{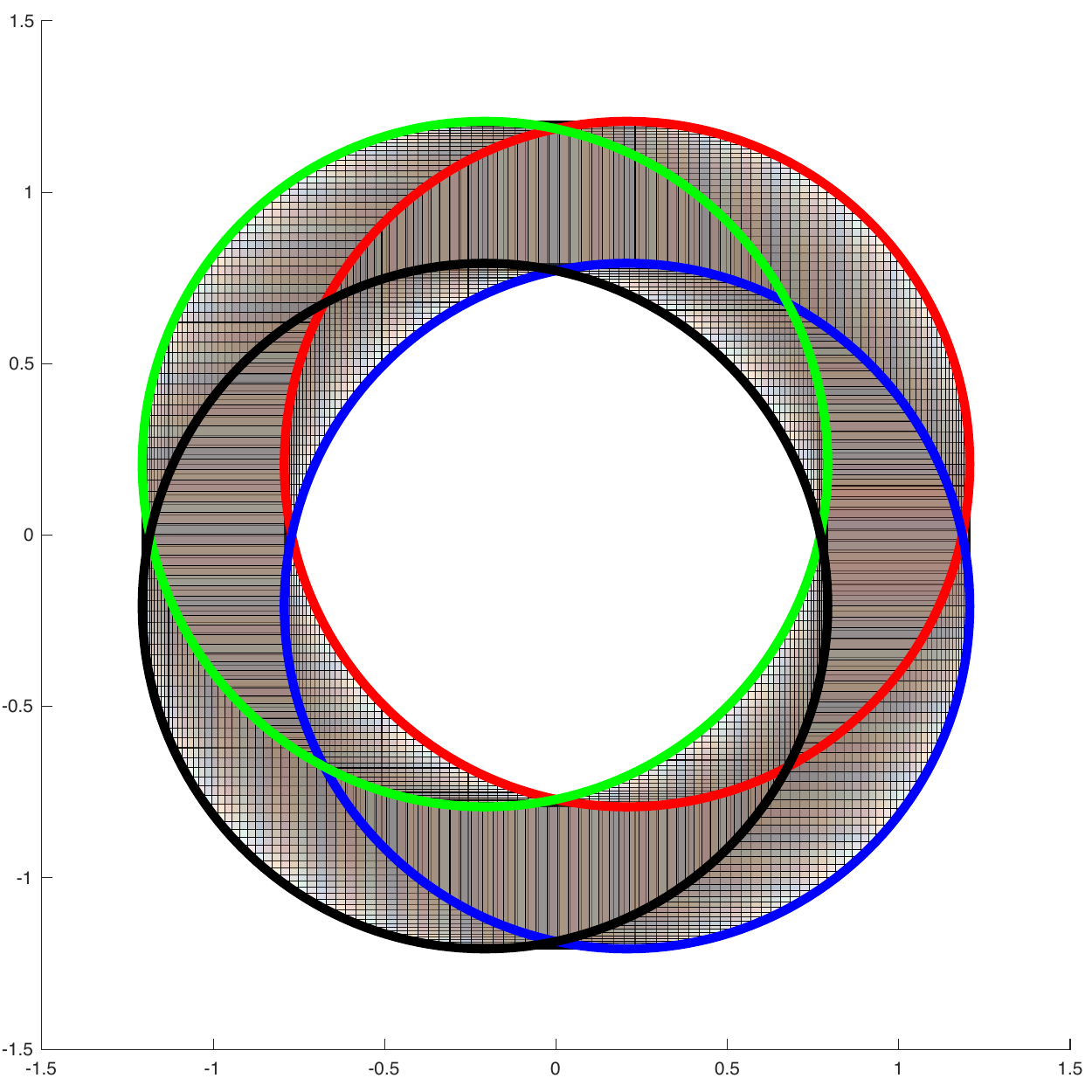}
		\caption{$B_r(\mathbb{S}^1_\infty,\mathbb{D}^2_\infty)=\mathbb{D}^2_\infty\backslash V_{2,r}$ in the plane $\R^2_\infty$. The set $V_{2,r}$ is given by the intersection of the 4 closed disks shown in the figure. See Theorem \ref{thm:ncircletightspan}.}
		\label{fig:d2-s1-tight-span}
	\end{figure}
	
	From a different perspective, by appealing to our isomorphism theorem, it is also possible to apply certain results from quantitative topology to the problem of characterization of metric spaces by their Vietoris-Rips persistence barcodes.  In applied algebraic topology, a general question of interest is:

	\begin{restatable*}[]{question}{qexact}\label{q:exact}
		Assume $X$ and $Y$ are compact metric spaces such that $\dgmR_k(X;\mathbb{F})=\dgmR_k(Y;\mathbb{F})$ for all $k\in\mathbb{Z}_{\geq 0}.$ Then, how similar are $X$ and $Y$ (in a suitable sense)?
	\end{restatable*}
	
	It follows from work by Wilhelm \cite{wilhem} and Yokota \cite{yokota} on rigidity properties of spheres via the filling radius, and the isomorphism theorem (Theorem \ref{theorem:isom}), that any $n$-dimensional Alexandrov space without boundary and sectional curvature bounded below by $1$ such that its Vietoris-Rips persistence barcode agrees with that of $\mathbb{S}^n$ must be \emph{isometric} to $\mathbb{S}^n$. This provides some new information about the inverse problem for persistent homology \cite{curry,gameiro}. More precisely, and for example, we obtain the corollary below, where for an $n$-dimensional manifold $M$, $I_{n,\mathbb{F}}^M$ denotes the persistence interval in $\dgmR_n(M)$ induced by the fundamental class of $M$ (cf. Proposition \ref{prop:filradpersistence}):
	
	\begin{restatable*}[$\dgm^\vr_*$ rigidity for spheres]{corollary}{cororigid}\label{Cor:Wilhelmrigidity}
		For any closed connected $n$-dimensional Riemannian manifold $M$ with sectional curvature $\mathrm{K}_M\geq 1$, we have: 
		\begin{enumerate}
			\item $I_{n,\mathbb{F}}^M\subseteq I_n^{\Sp^n}$.
			
			\item If $I_{n,\mathbb{F}}^M=I_n^{\Sp^n}$  then $M$ is isometric to $\Sp^n$.
			
			\item There exists $\epsilon_n>0$ such that if $\mathrm{length}(I_n^{\Sp^n})-\epsilon_n<\mathrm{length}(I_{n,\mathbb{F}}^M)$, then $M$ is diffeomorphic to $\Sp^n$.
			
			\item If $\mathrm{length}(I_{n,\mathbb{F}}^M)> \frac{\pi}{3}$, then $M$ is a twisted $n$-sphere (and in particular, homotopy equivalent to the $n$-sphere).
		\end{enumerate}
	\end{restatable*}

	\medskip

	The lower bound on sectional curvature is crucial -- in Example \ref{ex:one-parameter} we construct a one parameter family of deformations of the sphere $\Sp^2$ with constant filling radius (cf. Figure \ref{fig:oneparameter}).

	See Propositions \ref{prop:LiuWilhelm1} and \ref{prop:LiuWilhelm2} for additional related results, and see Question \ref{q:approx} for a relaxation of Question \ref{q:exact}.

	\begin{figure}
		\centering
		\includegraphics[width=15cm]{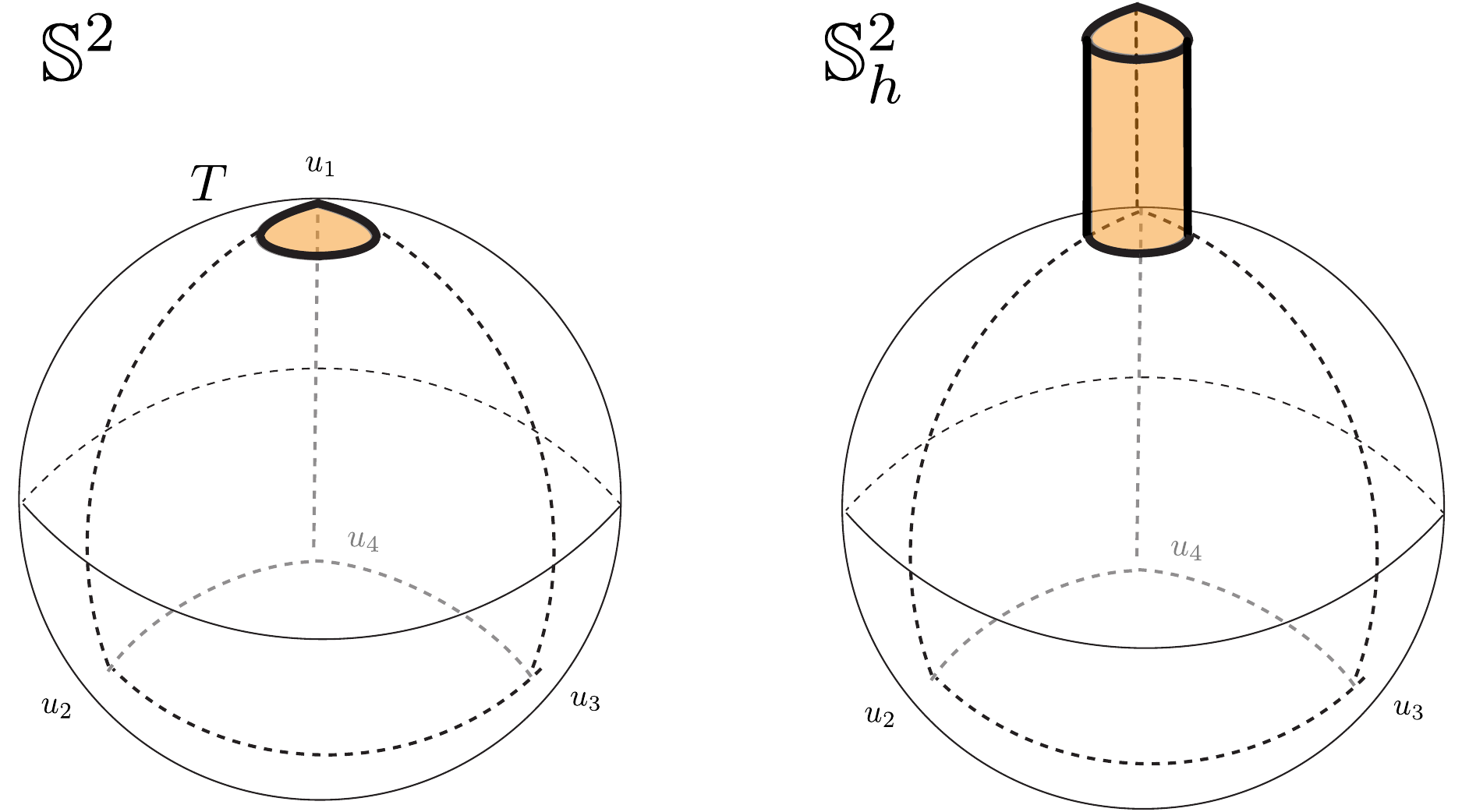}
		\caption{The construction of the one parameter family of surfaces $\Sp^2_h$ with the same filling radius as $\Sp^2$. $u_1$, $u_2$, $u_3$ and $u_4$ are vertices of a regular geodesic tetrahedron,  $T$ is a small  geodesic triangle, which is used to form a cylinder of height $h$ (left figure). See Example \ref{ex:one-parameter} for details.} \label{fig:oneparameter}
	\end{figure}

	Lastly, let us address a   variant of Question \ref{q:exact} concerning the case when $\dgmR_k(X;\mathbb{F})$ and $\dgmR_k(Y;\mathbb{F})$ are possibly different. Recall that there is the bottleneck distance $d_\mathrm{B}$  measuring the dissimilarity between two barcodes (cf. Definition \ref{def:bottledis}). One of the fundamental results of Topological Data Analysis is the following stability theorem (cf. Theorem \ref{thm:isometry} and Theorem \ref{thm:stab-barcodes}): for any field $\mathbb{F}$,
	\begin{equation}\label{eq:stabilityeq}
		\ell^\vr(X,Y):= \frac{1}{2}\sup_k d_\mathrm{B}(\dgmR_k(X;\mathbb{F}),\dgmR_k(Y;\mathbb{F}))\leq \dgh(X,Y).
	\end{equation}
	Therefore, in order to understand  how strong the Vietoris-Rips barcode is as a geometric invariant, it is natural to ask the following question.
	
	\begin{customquestion}{2 (i)}\label{customq:estimator}
		How good is $\ell^\vr(X,Y)$ as an estimator of $\dgh(X,Y)$?
	\end{customquestion}
	
	For an example, one might ask whether the inequality (\ref{eq:stabilityeq}) is tight or not. What we know is that this is indeed not tight for the cases when $X,Y$ are spheres of different dimension because of the following reason: in Corollary \ref{coro:dgh-spheres},  we show that $\ell^\vr(\Sp^m,\Sp^n)=\frac{1}{4}\arccos{\left(\frac{-1}{m+1}\right)}$ for any $0<m<n$. However, in \cite[Theorem B]{dgh-spheres} it is proved that $\frac{1}{2}\arccos{\left(\frac{-1}{m+1}\right)}$ (i.e., the twice  $\ell^\vr(\Sp^m,\Sp^n)$) lower bounds $\dgh(\Sp^m,\Sp^n)$ for any $0<m<n$ and that this bound is tight. 
	
	\medskip
	Now, let us ask the following question.
	
	\begin{customquestion}{2 (ii)}\label{customq:estimator-tight}
		For what type of spaces $X$ and $Y$ does  inequality (\ref{eq:stabilityeq}) become tight?
	\end{customquestion}
	
	Or, one might ask the following question, too.
	
	\begin{customquestion}{2 (iii)}\label{customq:estimator-reverse}
		For what type of spaces $X$ and $Y$ do we have the reverse stability inequality $\dgh(X,Y)\leq C\cdot\ell^\vr(X,Y)$ for some $C>0$?
	\end{customquestion}
	
	Note that the reverse stability inequality mentioned in Question \ref{customq:estimator-reverse} cannot hold in general. For  example, if we let $X=\Sp^1$ and $Y$ be  $\Sp^1$ attached with disjoint trees of arbitrary length (regarded as a geodesic metric space), then we can prove $\ell^\vr(X,Y)=0$ whereas $\dgh(X,Y)$ can be arbitrarily large (depending on the length of the attached trees). See Figure \ref{fig:s1-tree} and the beginning of Section \ref{sec:rigidity} for a more detailed explanation.

	\bigskip
	The authors hope that this paper can help bridge between the Applied Algebraic Topology and the Quantitative Topology communities.
	
	\paragraph{Organization.}
	In Section \ref{sec:background}, we provide some necessary definitions and results about Vietoris-Rips filtration, persistence, and injective metric spaces.

	In Section \ref{sec:pairs}, we construct a category of metric pairs. This category will be the natural setting for our extrinsic persistent homology. Although being functorial is trivial in the case of Vietoris-Rips persistence, the type of functoriality which one is supposed to expect in the case of metric embeddings is a priori not obvious. We address this question in Section \ref{sec:pairs} by introducing a suitable category structure.

	In Section \ref{sec:isom}, we show that the Vietoris-Rips filtration can be (categorically) seen as a special case of persistent homology obtained through metric embeddings via the isomorphism theorem (Theorem \ref{theorem:isom}). In this section, we also we also establish the stability of the filtration obtained via metric embeddings. 
	
	Sections \ref{sec:vrendpts}, \ref{sec:vrapp}, \ref{sec:hom-types}, \ref{sec:hyperbolicity}, and \ref{sec:fil} provide applications of our isomorphism theorem to different questions.
	
	In Section \ref{sec:vrendpts}, we prove that any interval in persistence barcode for open Vietoris-Rips filtration must have open left endpoint and closed right endpoint.

	In Section \ref{sec:vrapp}, we obtain new proofs of formulas about the Vietoris-Rips persistence of metric products and metric gluings of metric spaces.

	In Section \ref{sec:hom-types}, we prove a number of results concerning the homotopy types of Vietoris-Rips filtrations of spheres and complex projective spaces. Also, we fully compute the homotopy types of Vietoris-Rips filtration of spheres with $\ell^\infty$-norm.

	In Section \ref{sec:hyperbolicity}, we  reprove Rips and Gromov's result about the contractibility of the Vietoris-Rips complex of hyperbolic geodesic metric spaces, by using our method consisting of isometric embeddings into injective metric spaces. As a result, we will be able to bound the length of intervals in Vietoris-Rips persistence barcode by the hyperbolicity of the underlying space.

	In Section \ref{sec:fil}, we give some applications of our ideas to the filling radius of Riemannian manifolds and also study  consequences related to the characterization of spheres by their persistence barcodes and some generalizations and novel stability properties of the filling radius.
	
	The appendix contains relegated proofs and some background material.
	
	\subsubsection*{Acknowledgements.}
	
	We thank Prof. Henry Adams and Dr. Johnathan Bush for very useful feedback about a previous version of this article. We also thank Prof. Mikhail Katz and  Prof. Michael Lesnick for explaining to us some aspects of their work. We thank Dr. Qingsong Wang for bringing to our attention the paper \cite{schmahl2022} which was critical for the proof of Theorem \ref{theorem:pbarcode}. Finally, we thank Dr. Alexey Balitsky for pointing out an imprecision in the statement of Proposition \ref{prop:bcdmtpct}. 
	
	This research was supported by NSF under grants DMS-1723003, CCF-1740761, and CCF-1526513.

	\section{Background}\label{sec:background}
	In this section we cover the background needed for proving our main results. We alert readers that, in this paper, the same notation can mean either a simplicial complex itself or its geometric realization,  interchangeably. The precise meaning will be made clear in each context.
	
	\subsection{Vietoris-Rips filtration and persistence}\label{sec:vietoris}
	References for the definitions and results in this subsection are \cite{bl14,lesnick-iso}.  
	
	\begin{definition}[Vietoris-Rips filtration]
		Let $X$ be a metric space and $r>0$. The \emph{(open) Vietoris-Rips complex} $\vr_r(X)$ of $X$ is the simplicial complex whose vertices are the points of $X$ and whose simplices are those finite subsets of $X$ with diameter strictly less then $r$. Note that if $r\leq s$, then $\vr_r(X)$ is contained in $\vr_s(X)$. Hence, the family $\vr_*(X)$ is a filtration, called the \emph{open Vietoris-Rips filtration} of $X$.
	\end{definition}
	
	The (geometric realization of) a Vietoris-Rips filtration is a special case of the following more general notion.
	
	\begin{definition}[Persistence family] \label{def:pers-fams}
		A \emph{persistence family} is a collection $\big(U_r,f_{r,s} \big)_{r\leq s\in T}$ where $T$ is a nonempty subset of $\R$ such that, for each $r\leq s \leq t \in T$, $U_r$ is a topological space, $f_{r,s}: U_r \to U_s$ is a continuous map, $f_{r,r}=\mathrm{id}_{U_r}$ and $f_{s,t} \circ f_{r,s}=f_{r,t}$.
		
		Given two persistence families $(U_*,f_{*,*})$ and $(V_*,g_{*,*})$ indexed by the same $T\subseteq\R$, a morphism from the first one to the second is a collection $(\phi_r)_{r\in T}$ such that for each $r\leq s$, $\phi_r$ is a homotopy class of maps from $U_r \to V_r$ and $\phi_s \circ f_{r,s}$ is homotopy equivalent to $g_{r,s} \circ \phi_r$. 
	\end{definition}
	
	\begin{definition}[Persistence module]
		A persistence module $V_*=(V_r,v_{r,s})_{r\leq s\in T}$ over $T\subseteq\R$ is a family of $\mathbb{F}$-vector spaces $V_r$ for some field $\mathbb{F}$ with morphisms $v_{r,s}:V_r \to V_s$ for each $r \leq s$ such that 
		\begin{itemize}
			\item $v_{r,r}=\mathrm{id}_{V_r}$,
			\item $v_{s,t}\circ v_{r,s}=v_{r,t}$ for each $r\leq s \leq t$.
		\end{itemize}
		In other words, a persistence module is a functor from the poset $(T,\leq)$ to the category of vector spaces. The morphisms $v_{\ast,\ast}$ are referred to as the structure maps of $V_\ast$.
	\end{definition}
	
	By $0_*$ we will denote the zero persistence module.
	
	\medskip
	
	For any $k\geq 0$, applying the degree $k$ homology functor (with coefficients in a field $\mathbb{F}$) to a persistence family $\big(U_r,f_{r,s} \big)_{r\leq s\in T}$ produces the persistence module $\mathrm{H}_k(U_*;\mathbb{F})$ where the morphisms are those induced by $\big(f_{r,s}\big)_{r\leq s}$.
	
	Following the extant literature, we will use the term \emph{persistent homology} of a  persistence family (a.k.a. a filtration) to refer to the persistence module obtained  upon applying the  homology functor to this family.
	
	In particular, one can apply the homology functor to the Vietoris-Rips filtration of a metric space $X$. This induces a persistence module (with $T = \R_{>0}$) where the morphisms are those induced by inclusions. As a persistence module, it is denoted by $\PH_k(\vr_*(X);\mathbb{F})$ and referred to as the \emph{Vietoris-Rips persistent homology} of $X$.
	
	\begin{definition}[Interval persistence module, \cite{carlsson2010zigzag}]
		Given an interval $I$ in $T\subseteq\R$ (i.e. if $r\leq s\leq t$ and $r,t\in I$, then $s\in I$) and a field $\mathbb{F}$, the persistence module $\mathbb{F}_*[I]$ over $T$ is defined as follows: The vector space at $r$ is $\mathbb{F}$ if $r$ is in $I$ and zero otherwise. Given $r \leq s$, the morphism  corresponding to $(r,s)$ is the identity if $r,s$ are in $I$ and zero otherwise.  
	\end{definition}
	
	\begin{definition}[Barcode]
		For a given persistence module $V_*$, if there is a multiset of intervals $(I_\lambda)_{\lambda\in\Lambda}$ such that $V_*$ is isomorphic to $\bigoplus_{\lambda\in \Lambda} \mathbb{F}_*[I_\lambda]$, then that multiset is denoted by $\dgm(V_*)$ and referred to as a \emph{(persistence) barcode} associated to the persistence module $V_*$ (see below). Modules for which there exist such a multiset of intervals are said to be \emph{interval decomposable}.
	\end{definition}
	
	By Azumaya's theorem \cite{azumaya1950corrections}, persistence barcodes, whenever they exist, are unique: any two persistence barcodes associated to a given $V_*$ must agree  (up to reordering). The most important existence result for persistence barcodes is Crawley-Boevey's theorem \cite{crawley2015decomposition} which guarantees the existence of a persistence barcode associated to $V_*=(V_r,v_{r,s})$ if $V_*$ is pointwise finite-dimensional (i.e. $\dim(V_r)<\infty$ for all $r$). However, for many natural persistence modules (e.g., Vietoris-Rips persistent homology of a non-finite metric space $X$), it is not straightforward to verify the pointwise finite-dimensionality condition. Nevertheless, in Theorem \ref{theorem:pbarcode}, we were able to prove that if $X$ is totally bounded then its Vietoris-Rips persistent homology has a (unique) persistence barcode. This is achieved without invoking Crawley-Boevey's theorem and instead through combining our main (isomorphism) theorem (see Theorem \ref{theorem:isom}) with a recent result by Schmahl \cite[Theorem 1.2]{schmahl2022}. The totally boundedness condition is required in order to guarantee the following notion of regularity which will be useful in the proof of Theorem \ref{theorem:pbarcode}.
	
	\begin{definition}[q-tame persistence module]
		A persistence module $V_*=\big(V_r,v_{r,s} \big)_{r\leq s\in T}$ is said to be \emph{q-tame} if $\mathrm{rank}(v_{r,s})<\infty$ whenever $r < s$.
	\end{definition}
	
	\begin{Remark}
		The notions of interval decomposability and q-tameness are not equivalent. Indeed, \begin{enumerate}
			\item[(1)] there exist  q-tame modules which are not interval decomposable  \cite[Remark 2.9]{chazal2016structure} and 
			\item[(2)] there exist interval decomposable modules which are not q-tame 
			\cite[Example 3.30]{chazal2016structure}. 
		\end{enumerate} Interval decomposability and q-tameness are however related through a certain notion of weak isomorphism; see \cite{chazal2014observable}.  
	\end{Remark}
	
	\begin{Remark}\label{rmk:ttbddqtame}
		In \cite[Proposition 5.1]{cdo14}, the authors proved that if $X$ is a totally bounded metric space, then $\PH_k(\vr_*(X);\mathbb{F})$ is q-tame for any nonnegative integer $k\geq 0$ and any field $\mathbb{F}$.
	\end{Remark}
	
	\begin{restatable}{theorem}{vrintdcp}\label{theorem:pbarcode}
		If $X$ is a totally bounded metric space, then there is a (unique) persistence barcode associated to $\PH_k(\vr_*(X);\mathbb{F})$.
	\end{restatable}
	
	If $X$ is a totally bounded metric space, then we denote the barcode corresponding to $\PH_k(\vr_*(X);\mathbb{F})$ by $\dgmR_k(X;\mathbb{F})$.
	
	\medskip
	See Section \ref{sec:vrendpts} for the proof of Theorem \ref{theorem:pbarcode}. As we already mentioned earlier, our proof of Theorem \ref{theorem:pbarcode} does not depend on Crawley-Boevey's theorem since we circumvented verifying the pointwise finite-dimensionality of $\PH_k(\vr_*(X);\mathbb{F})$ for a totally bounded $X$. Therefore, the following question seems interesting.
	
	\begin{customquestion}{3}
		Assume a metric space $X$ is totally bounded and $k\geq 0$ is any integer. Then, is $\PH_k(\vr_*(X);\mathbb{F})$ pointwise finite-dimensional?
	\end{customquestion}
	
	From now on, unless specified otherwise, we will always assume that $T=\R$. For a given metric space and integer $k\geq 0$, we will occasionally view $V_*=\PH_k(\vr_*(X);\mathbb{F})$ as a persistence module defined over the whole real line $\R$ by trivially extending it to the left of $0\in\R$, that is, we  set $V_t=0$ for $t\leq 0$.

	We now recall a notion of distance between persistence modules.
	\begin{definition}[Interleaving distance]\label{def:int-dist}
		Two persistence modules $V_*$ and $W_*$  are said to be $\delta$-interleaved for some $\delta\geq 0$ if there are natural transformations $f:V_* \to W_{*+\delta}$ and $g: W_* \to V_{*+\delta}$ such that $f \circ g$ and $g \circ f$ are equal to the structure maps $W_* \to W_{*+2\delta}$ and $V_* \to V_{*+2\delta}$, respectively. The interleaving distance between $V_\ast$ and $W_\ast$ is defined as $$d_\mathrm{I}(V_\ast,W_\ast):=\inf\{\delta\geq 0|\,\mbox{$V_\ast$ and $W_\ast$ are $\delta$-interleaved}\}.$$
	\end{definition}
	
	It is known \cite{bl14} that $d_\mathrm{I}$ is an extended pseudo-metric on the collection of all persistence modules.
	
	\begin{example}\label{ex:int-zero}
		Consider $0_*$, the zero persistence module. Then, for any finite dimensional $V_*$ one has $$d_\mathrm{I}(V_*,0_*) = \frac{1}{2} \sup\{\mathrm{length}(I),\,I\in \dgm(V_*)\}.$$ 
	\end{example}
	
	\begin{definition}[Bottleneck distance]\label{def:bottledis}
		Let $M$ and $M'$ be two possibly empty multisets of intervals. A subset $P\subseteq M\times M'$ is said to be a partial matching between $M$ and $M'$ if it satisfies the following constraints:
		\begin{itemize}
			\item every interval $I\in M$ is matched with at most one interval of $M'$, i.e. there is at most one interval $I'\in M'$ such that $(I,I')\in P$,
			
			\item every interval $I'\in M'$ is matched with at most one interval of $M$, i.e. there is at most one interval $I\in M$ such that $(I,I')\in P$.
		\end{itemize}
		The bottleneck distance between $M$ and $M'$ is defined as
		$$d_\mathrm{B}(M,M'):=\inf_{P\textrm{ partial matching}}\mathrm{cost}(P)$$
		where
		$$\mathrm{cost}(P):=\max\left\{\sup\limits_{(I,I')\in P}\Vert I-I' \Vert_\infty,\sup\limits_{I=\langle a,b \rangle\in M \sqcup M'\textrm{ unnmatched}}\frac{\vert a-b \vert}{2}\right\}$$
		and
		$$\Vert I-I' \Vert_\infty:=\max\{\vert a-a' \vert,\vert b-b' \vert\}$$
		for $I=\langle a,b \rangle, I'=\langle a',b' \rangle$ (here, $\langle,\rangle$ means either open or closed endpoint).
	\end{definition}

	\begin{theorem}[{Isometry theorem, \cite[Theorem 5.14]{chazal2016structure}}]\label{thm:isometry}
		For any two q-tame persistence modules $V_*$ and $W_*$, the following equality holds:
		$$d_\mathrm{B}(\dgm(V_*),\dgm(W_*))=d_\mathrm{I}(V_*,W_*).$$
	\end{theorem}
	
	For the proof of the following theorem, see \cite[Lemma 4.3]{cdo14} or \cite{hmtint,dgh-rips,fm-tripods}.
	
	\begin{theorem}\label{thm:stab-barcodes}
		Let $X,Y$ be compact metric spaces and $\mathbb{F}$ be an arbitrary field. Then, for any $k \in \Z_{\geq 0}$ 
		$$d_\mathrm{I}(\PH_k(\vr_*(X);\mathbb{F}),\PH_k(\vr_*(Y);\mathbb{F})) \leq 2\,\dgh(X,Y). $$
	\end{theorem}
	
	\subsection{Injective (Hyperconvex) metric spaces}\label{sec:injective}
	
	A hyperconvex metric space is one where any collection of balls with  non-empty  pairwise intersections forces the non-empty intersection of all balls. These were studied by  Aronszajn and Panitchpakdi \cite{aronszajn1956extension} who showed that every hyperconvex space is an absolute 1-Lipschitz
	retract. Isbell \cite{isbell1964six} proved that every metric space admits a \emph{smallest} hyperconvex \emph{hull} (cf. the definition of tight span below). Dress rediscovered this concept in \cite{dress} and subsequent work provided much development in the context of phylogenetics \cite{phylo,dress-book}.  More recently, in \cite{joharinad2020topological} Joharinad and Jost considered relaxations of hyperconvexity and related it to a certain notion of curvature applicable to general metric spaces.
	
	\medskip
	References for this section are \cite{dress,dress-book,l13}.
	
	\begin{definition}[Injective metric space]\label{def:Injme}
		A metric space $E$ is called \emph{injective} if for each $1$-Lipschitz map $f: X \to E$ and isometric embedding of $X$ into $\tilde{X}$, there exists a $1$-Lipschitz map $\tilde{f}: \tilde{X} \to E$ extending $f$:
		$$\begin{tikzcd}
			X \arrow[r, hook] \arrow[dr, "f", rightarrow]
			& \tilde{X} \arrow[d, "\tilde{f}"]\\
			& E
		\end{tikzcd}$$
	\end{definition}
	
	\begin{definition}[Hyperconvex space] \label{def:hyper}
		A metric space $X$ is called  \emph{hyperconvex} if for every family $(x_i,r_i)_{i \in I}$ of $x_i$ in $X$ and $r_i \geq 0$ such that $d_X(x_i,x_j) \leq r_i + r_j$ for each $i,j$ in $I$, there exists a point $x$ such that $d_X(x_i,x) \leq r_i$ for each $i$ in $I$. 
	\end{definition}
	
	The following lemma is easy to deduce from the definition of hyperconvex space.
	
	\begin{lemma}\label{lemma:intercloballhyper}
		Any nonempty intersection of closed balls in hyperconvex space is hyperconvex.
	\end{lemma}
	
	For a proof of the following proposition, see \cite{aronszajn1956extension} or \cite[Proposition 2.3]{l13}.
	
	\begin{proposition}\label{prop:hypinj}
		A metric space is injective if and only if it is hyperconvex.
	\end{proposition}
	
	Moreover, every injective metric space is a contractible geodesic metric space, as one can see in Lemma \ref{lemma:geobicomb} and Corollary \ref{cor:injectivecont}.
	
	\begin{definition}[Geodesic bicombing]
		By a \emph{geodesic bicombing} $\gamma$ on a metric space $(X,d_X)$, we mean a continuous map $\gamma: X\times X\times [0,1]\rightarrow X$ such that, for every pair $(x,y)\in X\times X$, $\gamma(x,y,\cdot)$ is a geodesic from $x$ to $y$ with constant speed. In other words, $\gamma$ satisfies the following:
		\begin{enumerate}
			\item $\gamma(x,y,0)=x$ and $\gamma(x,y,1)=y$.
			\item $d_X(\gamma(x,y,s),\gamma(x,y,t))=(t-s)\cdot d_X(x,y)$ for any $0\leq s\leq t\leq 1$.
		\end{enumerate}
	\end{definition}
	
	\begin{lemma}[{\cite[Proposition 3.8]{l13}}]\label{lemma:geobicomb}
		Every injective metric space $(E,d_E)$ admits a geodesic bicombing $\gamma$ such that, for any $x,y,x',y'\in E$ and $t\in [0,1]$, it satisfies:
		\begin{enumerate}
			\item \textbf{Conical:} $d_E(\gamma(x,y,t),\gamma(x',y',t))\leq (1-t)\,d_E(x,x')+t\,d_E(y,y')$. 
			\item \textbf{Reversible:} $\gamma(x,y,t)=\gamma(y,x,1-t)$.
			\item \textbf{Equivariant:} $L\circ \gamma(x,y,\cdot)=\gamma(L(x),L(y),\cdot)$ for every isometry $L$ of $E$.
		\end{enumerate}
	\end{lemma}
	
	\begin{corollary}\label{cor:injectivecont}
		Every injective metric space $E$ is contractible.
	\end{corollary}
	\begin{proof}
		By Lemma \ref{lemma:geobicomb}, there is a geodesic bicombing $\gamma$ on $E$. Fix arbitrary point $x_0\in E$. Then, restricting $\gamma$ to $E \times \{x_0 \} \times [0,1]$ gives a deformation retraction of $E$ onto $x_0$. Hence $E$ is contractible.
	\end{proof}
	
	\begin{example}
		For any set $S$, the Banach space $L^\infty(S)$ consisting of all the bounded real-valued functions on $S$ with the  $\ell^\infty$-norm is injective. 
	\end{example}
	
	\begin{definition}\label{def:kuratowski}
		For a compact metric space $(X,d_X)$, the map $\kappa:X \to L^\infty(X)$, $x \mapsto d_X(x,\cdot)$ is an isometric embedding and it is called the \emph{Kuratowski embedding}. Hence every compact metric space can be isometrically embedded into an injective metric space.
	\end{definition}
	
	Let us introduce some notations which will be used throughout this paper. Suppose that $X$ is a subspace of a metric space $(E,d_E)$. For any $r>0$, let $B_r(X,E):=\{z\in E|\,\exists x\in X\mbox{ with } d_E(z,x)<r\}$ denote the open $r$-neighborhood of $X$ in $E$. In particular, if $X=\{x\}$ for some $x\in E$, it is just denoted by $B_r(x,E)$, the usual open $r$-ball around $x$ in $E$.
	
	As one more convention, whenever there is an isometric embedding $\iota:X\longhookrightarrow E$, we will use the notation $B_r(X,E)$ instead of $B_r(\iota(X),E)$. For instance, in the sequel we will use $B_r(X,L^\infty(X))$ rather than $B_r(\kappa(X),L^\infty(X))$.
	
	\begin{definition}\label{def:Cechcpx}
		For any metric space $E$, a nonempty subspace $X$, and $r>0$, the \v{C}ech complex $\check{\mathrm{C}}_r(X,E)$ is defined as the nerve of the open covering $\mathcal{U}_r:=\{B_r(x,E):\,x\in X\}$. In other words, $\check{\mathrm{C}}_r(X,E)$ is the simplicial complex whose vertices are the points of $X$, and $\{x_0,\dots,x_n\}\subseteq X$ is a simplex in $\check{\mathrm{C}}_r(X,E)$ if and only if $\cap_{i=0}^n B_r(x_i,E)\neq\emptyset$.
	\end{definition}
	
	The following observation is simple, yet it plays an important role in our paper. 
	
	\begin{proposition}\label{prop:cechvrsame}
		If $(E,d_E)$ is an injective metric space, and $\emptyset \neq X\subseteq E$ then, for any $r>0$,
		$$\check{\mathrm{C}}_r(X,E)=\vr_{2r}(X).$$
	\end{proposition}
	
	\begin{Remark}
		Note that  Proposition \ref{prop:cechvrsame} is optimal 
		in the sense that if $\check{\mathrm{C}}_r(X,E)=\vr_{2r}(X)$ holds true \emph{for all $\emptyset\neq X\subseteq E$}, then this condition itself resembles  hyperconvexity of $E$ (cf. Definition \ref{def:hyper}).
		
		Also note that Proposition \ref{prop:cechvrsame} is a generalization of both \cite[Lemma 4]{ghrist2005coverage} and \cite[Lemma 2.9]{dgh-rips} in that those papers only consider the case when $X$ is finite and $E=\ell^\infty(X)$.
	\end{Remark}
	
	\begin{proof}[Proof of Proposition \ref{prop:cechvrsame}]
		Because of the triangle inequality, it is obvious that $\check{\mathrm{C}}_r(X,E)$ is a subcomplex of $\vr_{2r}(X)$. Now, fix an  arbitrary simplex $\{x_0,\dots,x_n\}\in \vr_{2r}(X)$. It means that $d_E(x_i,x_j)<2r$ for any $i,j=0,\dots,n$. Now, since $E$ is hyperconvex by Proposition \ref{prop:hypinj}, there exists $\overline{x}\in E$ such that $d_X(x_i,\overline{x})<r$ for any $i=0,\dots,n$ (note that, since $\{x_0,\dots,x_n\}$ is finite, one can use $<$ instead of $\leq$ when  invoking the hyperconvexity property). Therefore,  $\{x_0,\dots,x_n\}\in\check{\mathrm{C}}_r(X,E)$. Hence $\vr_{2r}(X)$ is a subcomplex of $\check{\mathrm{C}}_r(X,E)$. This concludes the proof.
	\end{proof}
	
	In particular, Proposition \ref{prop:cechvrsame} implies the following result.
	
	\begin{proposition}\label{prop:vrhyperconvex}
		Let $X$ be a subspace of an injective metric space $(E,d_E)$. Then, for any $r>0$, the Vietoris-Rips complex $\mathrm{VR}_{2r}(X)$ is homotopy equivalent to $B_r(X,E)$.
	\end{proposition}
	
	The proof of Proposition \ref{prop:vrhyperconvex} will use the following lemma:
	
	\begin{lemma}\label{lem:intersectballs}
		In an injective metric space $E$, every non-empty intersection of open balls is contractible.
	\end{lemma}
	\begin{proof}
		Let $\gamma$ be a geodesic bicombing on $E$, whose existence is guaranteed by Lemma \ref{lemma:geobicomb}. Then, for each $x,y,x',y'$ in $E$ and $t$ in $[0,1]$,
		$$d_E(\gamma(x,y,t),\gamma(x',y',t)) \leq (1-t)\,d_E(x,x')+t\,d_E(y,y').$$
		In particular, by letting $x'=y'=z$, we obtain
		$$d_E(\gamma(x,y,t),z) \leq \max \big\{ d_E(x,z), d_E(y,z)\big\}$$
		for any $t\in [0,1]$. Hence, if $x,y$ are contained in an open ball with center $z$, then $\gamma(x,y,t)$ is contained in the same ball for each $t$ in $[0,1]$. Therefore, if $U$ is a non-empty intersection of open balls in $E$, then $\gamma$ restricts to $U \times U \times [0,1] \to U$, which implies that $U$ is contractible. 
	\end{proof}
	\begin{proof}[Proof of Proposition \ref{prop:vrhyperconvex}]
		Let $\mathcal{U}_r:=\{B_r(x,E):\,x\in X\}.$
		By Lemma \ref{lem:intersectballs}, $\mathcal{U}_r$ is a good cover of $B_r(X,E)$. Hence, by the nerve lemma (see \cite[Corollary 4G.3]{h01}), $B_r(X,E)$ is homotopy equivalent to the nerve of $\mathcal{U}_r$, which is the same as the \v{C}ech complex $\check{\mathrm{C}}_r(X,E)$. By Proposition \ref{prop:cechvrsame}, $\check{\mathrm{C}}_r(X,E)=\vr_{2r}(X).$ This concludes the proof.
	\end{proof}

	\section{Persistence via metric pairs}\label{sec:pairs}
	
	One of the insights leading to the notion of persistent homology associated to metric spaces was considering neighborhoods of a metric space in a nice (for example Euclidean) embedding \cite{niyogi}. In this section we formalize this idea in a categorical way.
	
	\begin{definition}[Category of metric pairs]\label{def:metric-pairs}
		\begin{itemize}
			\item A \emph{metric pair} is an ordered pair $(X,E)$ of metric spaces such that $X$ is a metric subspace of $E$.
			\item  Let $(X,E)$ and $(Y,F)$ be metric pairs. A $1$-Lipschitz map from $(X,E)$ to $(Y,F)$ is a $1$-Lipschitz map from $E$ to $F$ mapping $X$ into $Y$.
			\item Let $(X,E)$ and $(Y,F)$ be metric pairs and $f$ and $g$ be $1$-Lipschitz maps from $(X,E)$ to $(Y,F)$. We say that $f$ and $g$ are \emph{equivalent}  if there exists a continuous family $(h_t)_{t \in [0,1]}$ of $1$-Lipschitz maps from $E$ to $F$ and a $1$-Lipschitz map $\phi: X \to Y$ such that $h_0=f,h_1=g$ and $h_t|_X = \phi$ for each $t$.
			\item We define $\pmet$ as the category whose objects are metric pairs and whose morphisms are defined as follows. Given metric pairs $(X,E)$ and $(Y,F)$, the morphisms from $(X,E)$ to $(Y,F)$ are equivalence classes of $1$-Lipschitz maps from $(X,E)$ to $(Y,F)$.
		\end{itemize}
	\end{definition}
	
	Recall the definition of persistence families, Definition \ref{def:pers-fams}. We let $\htop$ denote the category of persistence families with morphisms specified as in Definition \ref{def:pers-fams}.
	
	\begin{Remark}
		Let $(X,E)$ and $(Y,F)$ be persistent pairs and let  $f$ be a $1$-Lipschitz morphism between them. Then, $f$ maps $B_r(X,E)$ into $B_r(Y,F)$ for each $r>0$. Furthermore, if $g$ is equivalent to $f$, then they reduce to homotopy equivalent maps from $B_r(X,E)$ to $B_r(Y,F)$ for each $r>0$.
	\end{Remark}
	By the remark above, we obtain the following functor from $\pmet$ to $\htop$.
	
	\begin{definition}[Persistence functor]
		Define the \emph{persistence functor} $\per: \pmet \to \htop$ sending $(X,E)$ to the persistence family obtained by the filtration $(B_r(X,E))_{r>0}$ and sending a morphism between metric pairs to the homotopy classes of maps it induces between the filtrations.
	\end{definition}
	
	\begin{Remark}
		Suppose a metric pair $(X,E)$ is given. For any $k\geq 0$, one can apply the $k$-th homology functor (with coefficients in a given field $\mathbb{F}$) to a persistence family $\per(X,E)$. This induces a persistence module where the morphisms are induced by inclusions. As a persistence module, it is denoted by $\PH_k(\per(X,E);\mathbb{F})$.
	\end{Remark}
	
	Let $\met$ be the category of metric spaces where morphisms are given by $1$-Lipschitz maps. There is a forgetful functor from $\pmet$ to $\met$ mapping $(X,E)$ to $X$ and mapping a morphism defined on $(X,E)$ to its restriction to $X$. Although forgetful functors often have left adjoints, we are going to see that this one has a right adjoint. 
	
	\begin{theorem}\label{theorem:adjoint}
		The forgetful functor from $\pmet$ to $\met$ has a right adjoint.
	\end{theorem}
	
	First we need to prove a few results. The reader should consult Section \ref{sec:injective} for background on injective metric spaces.
	
	\begin{lemma}\label{lem:bicombing}
		Let $(X,E)$ and $(Y,F)$ be metric pairs such that $F$ is an injective metric space. Let $f$ and $g$ be $1$-Lipschitz maps from $(X,E)$ to $(Y,F)$. Then, $f$ is equivalent to $g$ if and only if $f|_X \equiv g|_X$.
	\end{lemma}
	\begin{proof}
		The ``only if" part is obvious from Definition \ref{def:metric-pairs}. Now assume that $f|_X \equiv g|_X$. By Lemma \ref{lemma:geobicomb}, there exists a geodesic bicombing $\gamma: F \times F \times [0,1] \to F$ such that for each $x,y,x',y'$ in $F$ and $t$ in $[0,1]$, $$d_F(\gamma(x,y,t),\gamma(x',y',t)) \leq (1-t)\,d_F(x,x')+t\,d_F(y,y').$$ For $t$ in $[0,1]$, define $h: E\times[0,1] \to F$ by $h_t(x)=\gamma(f(x),g(x),t)$. Note that $h_0=f,h_1=g$ and $(h_t)|_X$ is the same map for all $t$. The inequality above implies that $h_t$ is $1$-Lipschitz for all $t$. This completes the proof.
	\end{proof}
	
	\begin{lemma}\label{lem:injective}
		Let $(X,E)$ and $(Y,F)$ be metric pairs such that $F$ is an injective metric space. Then for each $1$-Lipschitz map $\phi: X \to Y$, there exists a unique $1$-Lipschitz map from $(X,E)$ to $(Y,F)$ extending $\phi$ up to equivalence.
	\end{lemma}
	\begin{proof}
		The uniqueness up to equivalence part follows from Lemma \ref{lem:bicombing}. The existence part follows from the injectivity of $F$.
	\end{proof}
	
	\begin{proof}[Proof of Theorem \ref{theorem:adjoint}]
		Let $\kappa : \met \to \pmet$ be the functor sending $X$ to $(X,L^\infty(X))$ where $L^\infty(X)$ is the Banach space consisting of all the bounded real-valued functions on $X$ with $\ell^\infty$-norm (see Definition \ref{def:kuratowski} in Section \ref{sec:injective}). A $1$-Lipschitz map $f:X \to Y$ is sent to the unique morphism (see Lemma \ref{lem:injective}) extending $f$. This functor $\kappa$ is said to be the  \textbf{Kuratowski functor}.

		There is a natural morphsim $$\mathrm{Hom}((X,E),(Y,L^\infty(Y))) \to \mathrm{Hom}(X,Y), $$
		sending a morphism to its restriction to $X$. By Lemma \ref{lem:injective}, this is a bijection. Hence $\kappa$ is a right adjoint to the forgetful functor.
	\end{proof}
	
	Recall that any two right adjoints of a same functor must be isomorphic \cite[Proposition 9.9]{awodey2010category}.
	
	\begin{definition}[Metric homotopy pairing]
		A functor $\eta: \met \to \pmet$ is called a \emph{metric homotopy pairing} if it is a right adjoint to the forgetful functor.
	\end{definition}
	
	\begin{example}\label{ex:inj}
		Let $(X,d_X)$ be a metric space. $L^\infty(X)$ is an injective space associated to $X$; see Section \ref{sec:injective} for the precise definition. Consider also the following additional spaces associated to $X$:
		\begin{align*}
			&\Delta(X):=\{f\in L^\infty(X):f(x)+f(x')\geq d_X(x,x')\text{ for all }x,x'\in X\},\\
			&E(X):=\{f\in\Delta(X):\text{if }g\in\Delta(X)\text{ and }g\leq f,\text{ then }g=f\},\\
			&\Delta_1(X):=\Delta(X)\cap\textrm{Lip}_1(X,\mathbb{R}),
		\end{align*}
		with $\ell^\infty$-metrics for all of them (cf. \cite[Section 3]{l13}). Then,
		$$(X,L^\infty(X)), (X,E(X)), (X,\Delta(X)), (X,\Delta_1(X))$$
		are all metric homotopy pairings, since the second element in each pair is an injective metric space (see \cite[Section 3]{l13}) into which $X$ isometrically embeds via the map $\kappa:x\mapsto d_X(x,\cdot)$. Here, $E(X)$ is said to be the \textbf{tight span} of $X$ \cite{dress,isbell1964six} and it is a especially interesting space. $E(X)$ is the smallest injective metric space into which $X$ can be embedded and it is unique up to isometry. Furthermore, if $X$ is a tree metric space (i.e., a metric space with $0$-hyperbolicity; see Definition \ref{def:hyperbolicity}), then $E(X)$ is the smallest metric tree containing $X$. This special property has recently been used to the application of phylogenetics,  \cite{dress-book}.
	\end{example}

	\section{Isomorphism and stability}\label{sec:isom}
	
	Recall that $\met$ is the category of metric spaces with $1$-Lipschitz maps as morphisms. We have the functor $\vr_*: \met \to \htop$ induced by the Vietoris-Rips filtration. The main theorem we prove in this section is the following:
	
	\isotheorem
	
	Recall the precise definitions of $\mathcal{U}_r$ and $\check{\mathrm{C}}_r(X,E)$ from Definition \ref{def:Cechcpx}. We denote the filtration of \v{C}ech complexes $\big(\check{\mathrm{C}}_r(X,E)\big)_{r>0}$ by $\check{\mathrm{C}}_*(X,E)$.

	The following theorem is the main tool for the proof of Theorem \ref{theorem:isom}. Its proof, being fairly long, is relegated to Appendix \ref{sec:app:functnerve}.
	
	\begin{restatable}[Generalized Functorial Nerve Lemma]{theorem}{functnerve}
		\label{thm:fuctnerve}
		Let $X$ and $Y$ be two paracompact spaces, $\rho:X\longrightarrow Y$ be a continuous map, $\mathcal{U}=\{U_\alpha\}_{\alpha\in A}$ and $\mathcal{V}=\{V_\beta\}_{\beta\in B}$ be good open covers (every nonempty finite intersection is contractible) of $X$ and $Y$ respectively, based on arbitrary index sets $A$ and $B$, and $\pi:A\longrightarrow B$ be a map such that
		$$\rho(U_\alpha)\subseteq V_{\pi(\alpha)}\,\,
		\mbox{for any $\alpha\in A$}$$

		Let $\mathrm{N}\,\mathcal{U}$ and $\mathrm{N}\,\mathcal{V}$ be the nerves of $\mathcal{U}$ and $\mathcal{V}$, respectively. Observe that, since $U_{\alpha_0}\cap\dots \cap\,U_{\alpha_n}\neq\emptyset$ implies $V_{\pi(\alpha_0)}\cap\dots \cap\,V_{\pi(\alpha_n)}\neq\emptyset$, $\pi$ induces the canonical simplicial map $\bar{\pi}:\mathrm{N}\,\mathcal{U}\longrightarrow\mathrm{N}\,\mathcal{V}$.

		Then, there exist homotopy equivalences $X\longrightarrow \mathrm{N}\,\mathcal{U}$ and $Y\longrightarrow \mathrm{N}\,\mathcal{V}$ that commute with $\rho$ and $\bar{\pi}$ up to homotopy:
		$$\begin{tikzcd}[ampersand replacement=\&]
			X \arrow[r] \arrow[d, "\rho"] \& \mathrm{N}\,\mathcal{U}\arrow[d, "\bar{\pi}"] \\
			Y \arrow[r] \& \mathrm{N}\,\mathcal{V}
		\end{tikzcd}$$
	\end{restatable}
	
	The next corollary is an important special case of Theorem \ref{thm:fuctnerve}. 
	
	\begin{corollary}[Functorial Nerve Lemma]\label{cor:injfuctnerve}
		Let $X\subseteq X'$ be two paracompact spaces. Let $\mathcal{U}=\{U_\alpha\}_{\alpha\in\Lambda}$ and $\mathcal{U}'=\{U_\alpha'\}_{\alpha\in \Lambda}$ be good open covers (every nonempty finite intersection is contractible) of $X$ and $X'$ respectively, based on the same index set $\Lambda$, such that $U_\alpha\subseteq U_\alpha'$ for all $\alpha\in \Lambda$. Let $\mathrm{N}\,\mathcal{U}$ and $\mathrm{N}\,\mathcal{U}'$ be the nerves of $\mathcal{U}$ and $\mathcal{U}'$, respectively.

		Then, there exist homotopy equivalences $X\longrightarrow \mathrm{N}\,\mathcal{U}$ and $X'\longrightarrow \mathrm{N}\,\mathcal{U}'$ that commute with the canonical inclusions $X\longhookrightarrow X'$ and $\mathrm{N}\,\mathcal{U}\longhookrightarrow\mathrm{N}\,\mathcal{U}'$, up to homotopy:
		$$\begin{tikzcd}
			X \arrow[r] \arrow[d, hook] & \mathrm{N}\,\mathcal{U}\arrow[d, hook] \\
			X' \arrow[r] & \mathrm{N}\,\mathcal{U}'
		\end{tikzcd}$$
	\end{corollary}
	\begin{proof}
		Choose the canonical inclusion map $X\longhookrightarrow X'$ as $\rho$, the identity map on $\Lambda$ as $\pi$, and apply Theorem \ref{thm:fuctnerve}.
	\end{proof}
	
	\begin{Remark} 
		A result which is similar to Corollary \ref{cor:injfuctnerve} was already proved in \cite[Lemma 3.4]{chazal2008towards} for finite index sets whereas in our version index sets can have arbitrary cardinality. In \cite[Theorem 25, Theorem 26]{chowdhury2018functorial}, the authors prove simplicial complex version of Corollary \ref{cor:injfuctnerve} for finite index sets and invoke a certain functorial version of Dowker's  theorem.

		Finally, recently we became aware of \cite[Lemma 5.1]{virk2019rips}, which is  similar to Theorem \ref{thm:fuctnerve}. The author considers spaces with \emph{numerable} covers (i.e. the spaces admit locally finite partition of unity subordinate to the covers), whereas in our version that condition is automatically satisfied since we only consider paracompact spaces. Our proof technique differs from that of \cite{virk2019rips} in that whereas \cite{virk2019rips} relies on a result from \cite{dieck}, our proof follows the traditional proof of the nerve lemma \cite{h01}.
	\end{Remark}
	
	\begin{proposition}\label{prop:commutative1}
		For each metric pair $(X,E)\in\pmet$, there exist homotopy equivalences $$\phi_*^{(X,E)}:B_*(X,E)\longrightarrow\check{\mathrm{C}}_*(X,E)$$
		such that, for any $0<r\leq s$, the following diagram commutes up to homotopy:
		$$\begin{tikzcd}B_r(X,E)\arrow[d, hook]\arrow[r, "\phi_r^{(X,E)}"] & \check{\mathrm{C}}_r(X,E)\arrow[d, hook]\\B_s(X,E)\arrow[r, "\phi_s^{(X,E)}"] & \check{\mathrm{C}}_s(X,E)\end{tikzcd}$$
		where $B_r(X,E)\longhookrightarrow B_s(X,E)$ and $\check{\mathrm{C}}_r(X,E)\longhookrightarrow\check{\mathrm{C}}_s(X,E)$ are the canonical inclusions.
	\end{proposition}
	\begin{proof}
		Observe that $\check{\mathrm{C}}_r(X,E)$ is the nerve of the open covering $\mathcal{U}_r$ for any $r>0$, and apply Corollary \ref{cor:injfuctnerve}.
	\end{proof}
	
	\begin{proposition}\label{prop:commutative2}
		Let $(X,E)$ and $(Y,F)$ be metric pairs in $\pmet$ and $f:(X,E) \to (Y,F)$ be a $1$-Lipschitz map. Let $\phi_*^{(X,E)}:B_*(X,E)\longrightarrow\check{\mathrm{C}}_*(X,E)$ and $\phi_*^{(Y,F)}:B_*(Y,F)\longrightarrow\check{\mathrm{C}}_*(Y,F)$ be the homotopy equivalences guaranteed by Proposition \ref{prop:commutative1}. Then, for any $r>0$, the following diagram commutes up to homotopy:
		$$\begin{tikzcd}B_r(X,E)\arrow[d, "f_r"]\arrow[r, "\phi_r^{(X,E)}"] & \check{\mathrm{C}}_r(X,E)\arrow[d, "f_r"]\\B_r(Y,F)\arrow[r, "\phi_r^{(Y,F)}"] & \check{\mathrm{C}}_r(Y,F)\end{tikzcd}$$
		where $f_*:B_r(X,E)\longrightarrow B_r(Y,F)$ and $f_*:\check{\mathrm{C}}_r(X,E)\longrightarrow\check{\mathrm{C}}_r(Y,F)$ are the canonical maps induced from $f$.

		Furthermore, if we substitute $f$ by an equivalent map, then the homotopy types of the vertical maps remain unchanged.
	\end{proposition}
	\begin{proof}
		Since $f$ is $1$-Lipschitz, $f(B_r(x,E))\subseteq B_r(f(x),F)$. Hence, if we choose $f\vert_{B_r(X,E)}$ as $\rho$, and $f\vert_X$ as $\pi$, the commutativity of the diagram is the direct result of Theorem \ref{thm:fuctnerve}.

		Furthermore, if $f,g$ are equivalent, then the homotopy $(h_t)$ between $f,g$ induces the homotopy between $f_r:B_r(X,E)\longrightarrow B_r(Y,F)$ and $g_r:B_r(X,E)\longrightarrow B_r(Y,F)$. Moreover, since $f\vert_X=g\vert_X$, both of the induced maps $f_r:\check{\mathrm{C}}_r(X,E)\longrightarrow\check{\mathrm{C}}_r(Y,F)$ and $g_r:\check{\mathrm{C}}_r(X,E)\longrightarrow\check{\mathrm{C}}_r(Y,F)$ are exactly same.
	\end{proof}
	
	We are now ready to prove the main theorem of this section.
	
	\begin{proof}[Proof of Theorem \ref{theorem:isom}]
		Since all metric homotopy pairings are naturally isomorphic, without loss of generality we can assume that $\eta=\kappa$, the Kuratowski functor. Note that, by Proposition \ref{prop:cechvrsame}, $\check{\mathrm{C}}_r(X,E)=\vr_{2r}(X)$ for any $X\in\met$ and $r>0$.

		Let's construct the natural transformation $\tau$ from $\per \circ \kappa: \met \to \htop$ to $\vr_{2*}$ in the following way: Fix an arbitrary metric space $X\in\met$. Then, let $\tau_X$ to be the homotopy equivalences $\phi_*^{(X,L^\infty(X))}:B_*(X,L^\infty(X))\longrightarrow \vr_{2*}(X)$
		guaranteed by Proposition \ref{prop:commutative1}. Then, when $f:X\longrightarrow Y$ is $1$-Lipschitz, the functoriality between $\tau_X$ and $\tau_Y$ is the direct result of Proposition \ref{prop:commutative2}. So $\tau$ is indeed a natural transformation. Finally, since each $\phi_r^{(X,L^\infty(X))}$ is a homotopy equivalence for any $X\in\met$ and $r>0$, $\tau$ is natural isomorphism.
	\end{proof}
	

	\subsection{Stability of metric homotopy pairings}\label{subsec:stability}
	
	In this subsection, we consider a distance between metric pairs by invoking the homotopy interleaving distance introduced by Blumberg and Lesnick \cite{hmtint} and then show that metric homotopy pairings are $1$-Lipschitz with respect to this distance and the Gromov-Hausdorff distance.
	
	Let us give a quick review of homotopy interleaving distance between $\R$-spaces. For more details, please see \cite[Section 3.3]{hmtint}. An $\R$-space is a functor from the poset $(\R,\leq)$ to the category of topological spaces. Note that given a metric pair $(X,E)$, the filtration of open neighborhoods $B_*(X,E)$ is an $\R$-space. Two $\R$-spaces $A_*$ and $B_*$ are said to be $\delta$-interleaved for some $\delta>0$ if there are natural transformations $f:A_* \to B_{*+\delta}$ and $g: B_* \to A_{*+\delta}$ such that $f \circ g$ and $g \circ f$ are equal to the structure maps $B_* \to B_{*+2\delta}$ and $A_* \to A_{*+2\delta}$, respectively.
	
	A natural transformation $f:R_* \to A_*$ is called a weak homotopy equivalence if $f$ induces a isomorphism between homotopy groups at each index. Two $\R$-spaces $A_*$ and $A'_*$ are said to be weakly homotopy equivalent if there exists an $\R$-space $R_*$ and weak homotopy equivalences $f:R_* \to A_*$ and $f':R_* \to A'_*$. The homotopy interleaving distance $\dhi(A_*,B_*)$ is then defined as the infimal $\delta>0$ such that there exists $\delta$-interleaved $\mathbb{R}$-spaces $A'_*$ and $B'_*$ with the property that $A'_*$ and $B'_*$ are weakly homotopy equivalent to $A_*$ and $B_*$, respectively.
	
	We now adapt this construction to metric pairs. Given metric pairs $(X,E)$ and $(Y,F)$, we define the homotopy interleaving distance between them by $$\dhi((X,E),(Y,F)):=\dhi(B_*(X,E),B_*(Y,F)).$$
	The main theorem that we are going to prove in this section is the following. Below $\dgh$ denotes the Gromov-Hausdorff distance between metric spaces (see \cite{burago2001course}) and $d_\mathrm{I}$ denotes the interleaving distance between persistence modules (see Section \ref{sec:vietoris}):
	
	\begin{theorem}\label{thm:stabilitymh}
		Let $\eta: \met \to \pmet$ be a metric homotopy pairing. Then, for any compact metric spaces $X$ and $Y$, 
		$$\dhi(\eta(X),\eta(Y)) \leq \dgh(X,Y). $$
	\end{theorem}
	
	\begin{Remark} \label{rem:stab-gh}
		Note that through combining Theorem \ref{thm:stabilitymh} and the isomorphism theorem (Theorem \ref{theorem:isom}), one obtains another proof of Theorem \ref{thm:stab-barcodes}: for any compact metric spaces $X$ and $Y$, a field $\mathbb{F}$, and $k\in\mathbb{Z}_{\geq 0}$,
		$$\di(\PH_k(\vr_*(X);\mathbb{F}),\PH_k(\vr_*(Y);\mathbb{F})) \leq 2\,\dgh(X,Y). $$
	\end{Remark}
	
	\begin{lemma}\label{lem:isomhi}
		If $(X,E)$ and $(Y,F)$ are isomorphic  in $\pmet$, then $\dhi((X,E),(Y,F))=0.$
	\end{lemma}
	\begin{proof}
		Let $f: (X,E) \to (Y,F)$ and $g: (Y,F) \to (X,E)$ be $1$-Lipschitz maps such that $f \circ g$ and $g \circ f$ are equivalent to the respective identities. Then, the result follows since $f$ and $g$ induce isomorphism between the $\R$-spaces $B_*(X,E)$ and $B_*(Y,F)$. 
	\end{proof}
	
	\begin{lemma}\label{lem:isomhyp}
		Let $E$ and $F$ be injective metric spaces containing $X$. Then $(X,E)$ is isomorphic to $(X,F)$ in $\pmet$.
	\end{lemma}
	\begin{proof}
		By injectivity of $E$ and $F$, there are $1$-Lipschitz maps $f: E \to F$ and $g: F \to E$ such that $f|_X$ and $g|_X$ are equal to $\mathrm{id}_X$. Hence, by Lemma \ref{lem:bicombing}, $f \circ g: (X,F) \to (X,F)$ and $g \circ f: (X,E) \to (X,F)$ are equivalent to identity. 
	\end{proof}
	
	\begin{proof}[Proof of Theorem \ref{thm:stabilitymh}]
		Since all metric homotopy pairings are naturally isomorphic, by Lemma \ref{lem:isomhi}, without loss of generality we can assume that $\eta=\kappa$, the Kuratowski functor. Let $r>\dgh(X,Y)$. Let us show that $$\dhi((X,L^\infty(X)),(Y,L^\infty(Y)) \leq r. $$
		
		By assumption (see \cite{burago2001course}), there exists a metric space $Z$ containing $X$ and $Y$ such that the Hausdorff distance between $X$ and $Y$ as subspaces of $Z$ is less than or equal to $r$. Hence, the $\R$-spaces $B_*(X,L^\infty(Z))$ and $B_*(Y,L^\infty(Z))$ are $r$-interleaved as $B_\epsilon(X,L^\infty(Z)) \subseteq B_{r+\epsilon}(Y,L^\infty(Z))$ and $B_\epsilon (Y,L^\infty(Z)) \subseteq B_{r+\epsilon}(X,L^\infty(Z))$ for each $\epsilon$. Now, by Lemma \ref{lem:isomhyp}, we have $$\dhi((X,L^\infty(X)), (Y, L^\infty(Y))) = \dhi((X,L^\infty(Z),(Y,L^\infty(Z))) \leq r. $$ 
	\end{proof}


	\section{Application: the study of endpoints of intervals in $\dgmR_k(X)$ and proof of Theorem \ref{theorem:pbarcode}}\label{sec:vrendpts}
	
	We now establish the existence of a persistence barcode associated to $\PH_k(\vr_*(X);\mathbb{F})$ whenever $X$ is totally bounded.
	
	\vrintdcp*
	\begin{proof}
		We will invoke  \cite[Theorem 1.2]{schmahl2022}. Schmahl proved that a q-tame persistence module $V_*=(V_r,v_{r,s})_{r\leq s\in T}$ admits a persistence barcode whenever the index set $T$ satisfies some mild technical conditions (all of which are satisfied by $\R$) and, furthermore,    the canonical map $\mathrm{colim}_{r<s}V_r\rightarrow V_s$ is an isomorphism for all $s>0$ -- a condition which we refers to as \emph{lower semi-continuity} of  $V_*$.\footnote{In order to establish this result, Schmahl employed \cite[Corollary 3.6]{chazal2014observable} which guarantees that, if the index set $T$ satisfies the aforementioned technical assumptions, then the persistence module arising as the image of the canonical map from $\overline{V_*}$ to $V_*$ always admits a persistence barcode, where $\overline{V_*}$ is the persistence module with $\{\mathrm{colim}_{r<s}V_r\}_s$ as objects and the canonical structure maps as morphisms.}

		By our main (isomorphism) theorem (Theorem \ref{theorem:isom}), it is enough to show that $\PH_k(\per \circ \eta(X);\mathbb{F})$ admits a persistence barcode. First of all, note that $\PH_k(\per \circ \eta(X);\mathbb{F})$ is q-tame by Remark \ref{rmk:ttbddqtame} and Theorem \ref{theorem:isom}. In order to apply  \cite[Theorem 1.2]{schmahl2022}, it suffices to verify the lower semi-continuity of $\PH_k(\per \circ \eta(X);\mathbb{F})$. Note that, since $$\mathrm{B}_s \circ \eta(X)=\bigcup\limits_{r<s}\mathrm{B}_r \circ \eta(X)\,\,\,\,\mbox{for all $s>0$,}$$ by \cite[Section 14.6]{may99},  the canonical map $\mathrm{colim}_{r<s}\Hom_k(\mathrm{B}_r \circ \eta(X);\mathbb{F})\rightarrow\Hom_k(\mathrm{B}_s \circ \eta(X);\mathbb{F})$ is an isomorphism for all $s>0$. This completes the proof.
	\end{proof}
	
	It is known that, in some cases, the intervals in the Vietoris-Rips barcode of a metric space are of the form $(u,v]$ or $(u,\infty)$ for $0\leq u< v<\infty$.

	\begin{example}\label{Ex:intervalform}
		In the following examples, any $I\in\dgmR_k(X;\mathbb{F})$ has a form of $(u,v]$ or $(u,\infty)$ for some $0\leq u<v<\infty$.
		\begin{enumerate}
			\item When $X$ is a finite metric space, for any $k\geq 0$.
			\item When $X=\Sp^1$, for any $k\geq 0$ (see \cite[Theorem 7.4]{adams}).
			\item When $X$ is a compact geodesic metric space, for $k=1$ (see \cite[Theorem 8.2]{Virk18}).
		\end{enumerate}
	\end{example}

	As far as we know the general statement given in Theorem \ref{thm:intervalform} below is first proved in this paper. Our proof crucially exploits the isomorphism theorem (Theorem \ref{theorem:isom}).

	\begin{theorem}\label{thm:intervalform}
		Suppose a compact metric space $(X,d_X)$, a field $\mathbb{F}$, and an nonnegative integer $k$ are given. Then, for any $I\in\dgmR_k(X;\mathbb{F})$, $I$ must be either of the form  $(u,v]$ or $(u,\infty)$ for some $0\leq u<v<\infty$.
	\end{theorem}
	
	\paragraph{The proof of Theorem \ref{thm:intervalform}.} In this section we first state and prove two lemmas which will be combined in order to furnish the proof of Theorem \ref{thm:intervalform}.
	
	\begin{lemma}\label{lemma:singchaincpt}
		Let $X$ be a topological space and $G$ be an abelian group. Then, for any $k\geq 0$ and any $k$-dimensional singular chain $c$ of $X$ with coefficients in $G$, there exist a compact subset $K_c\subseteq X$ and $k$-dimensional singular chain $c'$ of $K_c$ with coefficients in $G$ such that
		$$(\iota)_\sharp(c')=c,$$
		where $\iota:K_c\longhookrightarrow X$ is the canonical inclusion map.
	\end{lemma}
	\begin{proof}
		Recall that one can express $c$ as a sum of finitely many $k$-dimensional singular simplices with coefficients in $G$. In other words,
		$$c=\sum_{i=1}^l\alpha_i\sigma_i$$
		where $\alpha_i\in G$ and $\sigma_i:\Delta_k\rightarrow X$ is a continuous map for each $i=1,\dots,l$. Next, let $K_c:=\cup_{i=1}^l\sigma_i(\Delta_k)$. This $K_c$ is the compact subspace that we required.
	\end{proof}
	
	For the remainder of this section, given any field $\mathbb{F}$ and a metric pair $(X,E)$, for each $0<r<\infty$ we will denote by $(SC_\ast^{(r)},\partial_\ast^{(r)})$ the singular chain complex of $B_r(X,E)$ with coefficients in $\mathbb{F}$. For each $0<r\leq s<\infty$ we will denote by $i_{r,s}$ the canonical inclusion map $B_r(X,E)\subseteq B_s(X,E)$. By $(i_{r,s})_\sharp$ we will denote the (injective) map induced at the level of singular chain complexes.
	
	\begin{lemma}\label{lemma:left-right}
		Suppose that a compact metric space $(X,d_X)$, a field $\mathbb{F}$, a metric homotopy pairing $\eta$, and a nonnegative integer $k$ are given. Then, for every $I\in\dgm(\PH_k(\per \circ \eta(X);\mathbb{F})),$
		\begin{description}
			\item[(i)] if $u\in[0,\infty)$ is the left endpoint of $I$, then $u\notin I$ (i.e. $I$ is left-open).
			
			\item[(ii)] if $v\in[0,\infty)$ is the right endpoint of $I$, then $v\in I$ (i.e. $I$ is right-closed).
		\end{description}
	\end{lemma}
	
	\begin{Remark}
		The lower semi-continuity property of $\PH_k(\per \circ \eta(X);\mathbb{F})$ invoked in the proof of Theorem \ref{theorem:pbarcode} can be seen (with some work) to be equivalent to Lemma \ref{lemma:left-right}: the injectivity of the canonical maps $\mathrm{colim}_{r<s}\Hom_k(\mathrm{B}_r \circ \eta(X);\mathbb{F})\rightarrow\Hom_k(\mathrm{B}_s \circ \eta(X);\mathbb{F})$ is equivalent to item (ii) of Lemma \ref{lemma:left-right} and their surjectivity  is equivalent to item (i) of Lemma \ref{lemma:left-right}. However, we provide a  direct and elementary proof of Lemma \ref{lemma:left-right} below.
	\end{Remark}
	
	\begin{proof}[Proof of \textbf{\emph{(i)}}]
		Let $\eta(X)=(X,E)$. The fact that $I\in\dgm(\mathrm{PH}_k(\per \circ \eta(X);\mathbb{F}))$ implies that, for each $r\in I$, there exists a singular $k$-cycle $c_r$ on $B_r(X,E)$ with coefficients in $\mathbb{F}$ satisfying the following:
		
		\begin{enumerate}
			\item $[c_r]\in \mathrm{H}_k(B_r(X,E);\mathbb{F})$ is nonzero for any $r\in I$.
			\item $(i_{r,s})_\ast([c_r])=[c_s]$ for any $r\leq s$ in $I$.
		\end{enumerate}
		
		Now, suppose that $u$ is a closed left endpoint of $I$ (so, $u\in I$).  In particular, by the above there exists a singular $k$-cycle $c_u$ on $B_u(X,E)$ with coefficients in $\mathbb{F}$ with the above two properties.

		Then, by Lemma \ref{lemma:singchaincpt}, we know that there is a compact subset $K_{c_u}\subseteq B_u(X,E)$ and a singular $k$-cycle $c_u'$ on $K_{c_u}$ with coefficients in $\mathbb{F}$ such that $(\iota)_\sharp(c_u')=c_u$ where $\iota:K_{c_u}\longrightarrow B_u(X,E)$ is the canonical inclusion. Moreover, since $K_{c_u}$ is compact, there exists small $\varepsilon>0$ such that $$K_{c_u}\subseteq B_{u-\varepsilon}(X,E).$$
		Now, define $c_{u-\varepsilon}:=(\iota')_\sharp(c_u')$ where $\iota':K_{c_u}\longrightarrow B_{u-\varepsilon}(X,E)$ is the canonical inclusion. Then, this singular chain satisfies $$(i_{u-\varepsilon,u})_\sharp(c_{u-\varepsilon})=(i_{u-\varepsilon,u})_\sharp\circ(\iota')_\sharp(c_u')=(\iota)_\sharp(c_u')=c_u.$$
		Moreover,  $c_{u-\varepsilon}$ cannot be null-homologous. Otherwise, there would exist a singular $(k+1)$-chain $d_{u-\varepsilon}$ on $B_{u-\varepsilon}(X,E)$ with coefficients in $\mathbb{F}$ such that $\partial_{k+1}^{(u-\varepsilon)} d_{u-\varepsilon}=c_{u-\varepsilon}$. However, this would imply that $$\partial_{k+1}^{(u)}\circ(i_{u-\varepsilon,u})_\sharp(d_{u-\varepsilon})=(i_{u-\varepsilon,u})_\sharp\circ\partial_{k+1}^{(u-\varepsilon)}(d_{u-\varepsilon})=(i_{u-\varepsilon,u})_\sharp(c_{u-\varepsilon})=c_u$$
		because of the naturality of the boundary operators $\partial_{k+1}^{(u-\varepsilon)}$ and $\partial_{k+1}^{(u)}$. This would in turn contradict  the property $[c_u]\neq 0$.

		So, we must have $[c_{u-\varepsilon}]\neq 0$. But, the existence of such $c_{u-\varepsilon}$ contradicts  the fact that $u$ is the left endpoint of $I$. Therefore, one  concludes that $u$ cannot be a  closed left endpoint, so it must be an open endpoint.
	\end{proof}
	
	\begin{proof}[Proof of \textbf{\emph{(ii)}}]
		Now, suppose that $v$ is an open right endpoint of $I$ (so that $v\notin I$ and therefore $c_v$ is not defined by the above two conditions). Choose small enough $\varepsilon>0$ so that $v-\varepsilon\in I$, and let
		$$c_v:=(i_{v-\varepsilon,v})_\sharp(c_{v-\varepsilon}).$$
		Then, $c_v$ must be null-homologous.

		This means that there exists a singular $(k+1)$-dimensional chain $d_v$ on $B_v(X,E)$ with coefficients in $\mathbb{F}$ such that $\partial_{k+1}^{(v)} d_v=c_v$. By Lemma \ref{lemma:singchaincpt}, we know that there is a compact subset $K_{d_v}\subseteq B_v(X,E)$ and a singular $(k+1)$-chain $d_v'$ of $K_{d_v}$ with coefficients in $\mathbb{F}$ such that $(\iota)_\sharp(d_v')=d_v$ where $\iota:K_{d_v}\longrightarrow B_v(X,E)$ is the canonical inclusion. Moreover, since $K_{d_v}$ is compact, there exists $\varepsilon'\in (0,\varepsilon]$ such that $K_{d_v}\subseteq B_{v-\varepsilon'}(X,E)$.

		Let $d_{v-\varepsilon'}:=(\iota')_\sharp(d_v')$ where $\iota':K_{d_v}\longhookrightarrow B_{v-\varepsilon'}(X,E)$ is the canonical inclusion. Then, again by the naturality of boundary operators,
		\begin{align*}
			(i_{v-\varepsilon',v})_\sharp\circ\partial_{k+1}^{(v-\varepsilon')}(d_{v-\varepsilon'})&=\partial_{k+1}^{(v)}\circ(i_{v-\varepsilon',v})_\sharp(d_{v-\varepsilon'})\\
			&=\partial_{k+1}^{(v)}\circ (i_{v-\varepsilon',v})_\sharp\circ (\iota')_\sharp (d_v')=\partial_{k+1}^{(v)}\circ (\iota)_\sharp(d_v')=\partial_{k+1}^{(v)} d_v=c_v.
		\end{align*}

		Since $(i_{v-\varepsilon',v})_\sharp$ is injective and $(i_{v-\varepsilon',v})_\sharp\circ (i_{v-\varepsilon,v-\varepsilon'})_\sharp (c_{v-\varepsilon})=(i_{v-\varepsilon,v})_\sharp(c_{v-\varepsilon})=c_v$, one can conclude that $\partial_{k+1}^{(v-\varepsilon')}(d_{v-\varepsilon'})=(i_{v-\varepsilon,v-\varepsilon'})_\sharp (c_{v-\varepsilon})$. This indicates that 
		$$0=[(i_{v-\varepsilon,v-\varepsilon'})_\sharp (c_{v-\varepsilon})]=(i_{v-\varepsilon,v-\varepsilon'})_\ast ([c_{v-\varepsilon}])=[c_{v-\varepsilon'}],$$
		but it contradicts the fact that $[c_{v-\varepsilon'}]\neq 0$. Therefore, $v$ must be a closed endpoint.
	\end{proof}
	
	Finally, the proof of Theorem \ref{thm:intervalform} follows from the lemmas above.
	
	\begin{proof}[Proof of Theorem \ref{thm:intervalform}]
		Apply Lemma \ref{lemma:left-right} and Theorem \ref{theorem:isom}.
	\end{proof}
	
	\paragraph{A (false) conjecture.} Actually, we first expected the following conjecture to be true. Observe that, if true, the conjecture would imply Theorem \ref{thm:intervalform}. Also, it is obvious that this conjecture is true when $X$ is a finite metric space.
	
	\begin{conjecture}[Lower semicontinuity of the homotopy type of Vietoris-Rips complexes]
		Suppose $X$ is a compact metric space. Then, for any $r\in\mathbb{R}_{>0}$, $\vr_r(X)$ is homotopy equivalent to $\vr_{r-\varepsilon}(X)$ whenever $\varepsilon>0$ is small enough.
	\end{conjecture}
	
	However, the following example shows that this conjecture is false.
	
	\begin{example}
		By \cite[Theorem 7.4]{adams}, we know that $\vr_r(\mathbb{S}^1)$ is homotopy equivalent to $\mathbb{S}^{2m+1}$ if $r\in\left(\frac{2\pi m}{2m+1},\frac{2\pi(m+1)}{2m+3}\right]$ for $m=0,1,\cdots$. Observe that $\lim_{m\rightarrow\infty}\frac{2\pi m}{2m+1}=\pi$. Therefore, $\vr_\pi(\mathbb{S}^1)$ cannot be homotopy equivalent to $\vr_{\pi-\varepsilon}(\mathbb{S}^1)$ for all small enough $\varepsilon$, since for $r$ in the interval $ [\pi-\varepsilon,\pi]$, $\vr_r(\mathbb{S}^1)$ attains infinitely many different homotopy types.
	\end{example}
	
	Then, one might now wonder whether the conjecture holds when we restrict the range of $r$ to $(0,\diam(X))$. But, again this new conjecture is false as the following example shows.
	
	\begin{example}
		Let $X:=\Sp^1\vee \alpha\cdot\Sp^1$ for some $\alpha\in (0,1)$. Observe that $\diam(X)=\pi$. Also, by Lemma \ref{lem:wedge}, $E\vee F$ will be an injective metric space containing $X$ whenever $E$ is an injective metric space containing $\Sp^1$ (e.g., $E(\Sp^1)$) and $F$ is an injective metric space containing $\alpha\cdot\Sp^1$ (e.g., $E(\alpha\cdot\Sp^1)$). Hence, by Proposition \ref{prop:vrhyperconvex}, $\vr_{2r}(X)\simeq B_r(X,E\vee F)=B_r(\Sp^1,E)\vee B_r(\alpha\cdot\Sp^1,F)$ and $\vr_{2r}(\alpha\cdot\Sp^1)\simeq B_r(\alpha\cdot\Sp^1,F)$ for any $r>0$. Therefore, $\vr_{\alpha\pi}(X)$ cannot be homotopy equivalent to $\vr_{\alpha\pi-\varepsilon}(X)$ for  small enough $\varepsilon$, since  $\vr_r(\alpha\cdot \Sp^1)$ attains infinitely many homotopy types for $r\in [\alpha\pi-\varepsilon,\alpha\pi]$.
	\end{example}
	
	\begin{Remark}
		Some time after we discovered the above proof of Theorem \ref{thm:intervalform}, we realized that one can actually directly prove the  result at the simplicial level which makes the proof slightly simpler; see Appendix \ref{sec:app:otherpfintvltype}.
	\end{Remark}
	
	\section{Application: products and metric gluings}\label{sec:vrapp}
	
	The following statement regarding products of filtrations are obtained at the simplicial level (and in more generality) in \cite[Proposition 2.6]{perea-icaasp} and in \cite{hitesh-perea,perea}. The statement about metric gluings appeared in 
	\cite[Proposition 4]{wedge17} and \cite[Proposition 4.4]{mo18}. These proofs operate at the simplicial level.
	
	Here we give alternative proofs through the consideration of neighborhoods in an injective metric space via Theorem \ref{theorem:isom}.
	
	We first recall the notion of metric gluing: Given two metric spaces $X$ and $Y$ and points $p\in X$ and $q\in Y$, the \emph{metric gluing} $X\vee Y:=X\sqcup Y\slash p\sim q$ is defined through the metric: $$d_{X \vee Y}(z,z'):=\begin{cases}d_X(z,z')&\text{if }z,z'\in X\\d_Y(z,z')&\text{if }z,z'\in Y\\ d_X(z,p)+d_Y(z',q)&\text{if }z\in X\text{ and }z'\in Y.\end{cases}$$
	
	\begin{theorem}[Persistent K\"unneth Formula]\label{thm:vrapp} Let $X$ and $Y$ be  metric spaces, and $\mathbb{F}$ be a field.
		\begin{enumerate}
			\item[(1)] (Persistent K\"unneth formula) Let $X \times Y$ denote the $\ell^\infty$-product of $X,Y$. Then,
			$$\PH_*(\vr_*(X \times Y);\mathbb{F}) \cong \PH_*(\vr_*(X);\mathbb{F}) \otimes \PH_*(\vr_*(Y);\mathbb{F}).$$
			
			\item[(2)] Let $p$ and $q$ be points in $X$ and $Y$ respectively. Let $X \vee Y$ denote the metric gluing of metric spaces $X$ and $Y$ along $p$ and $q$. Then\footnote{We use ``reduced'' homology functor for this metric gluing case.} $$\PH_*(\vr_*(X \vee Y);\mathbb{F}) \cong \PH_*(\vr_*(X);\mathbb{F}) \oplus \PH_*(\vr_*(Y);\mathbb{F}). $$
		\end{enumerate}
	\end{theorem}
	
	\begin{Remark}
		Note that \cite[Corollary 5.2 and 5.8]{adams2021operations} establish results analogous to Theorem \ref{thm:vrapp} for the products and metric gluings of Vietoris-Rips metric thickenings.
	\end{Remark}
	
	\begin{Remark}\label{rem:productbarcode}
		Note that the tensor product of two simple persistence modules corresponding to intervals $I$ and $J$ is the simple persistence module corresponding to the interval $I \cap J$. Therefore, the first part of Theorem \ref{thm:vrapp} implies that
		$$\dgmR_k(X \times Y;\mathbb{F}):=\{I \cap J: I \in \dgmR_i(X;\mathbb{F}), J \in \dgmR_j(Y;\mathbb{F}), i+j=k\}. $$
		for any nonnegative integer $k$.
	\end{Remark}
	
	\begin{example}[Tori]\label{example:torusbarcode}
		For a given choice of $\alpha_1,\ldots,\alpha_n>0$, let $X$ be the $\ell^\infty$-product $\Pi_{i=1}^n (\alpha_i\cdot \mathbb{S}^1)$. Then, by \cite[Theorem 7.4]{adams} and Remark \ref{rem:productbarcode}:
		$$\dgmR_0(X;\mathbb{F})=\{(0,\infty)\},$$
		and
		\begin{align*}
			&\dgmR_k(X;\mathbb{F})\\
			&\quad=\left\{ \left(\max_{1\leq j\leq m} \frac{2\pi\alpha_{i_j} l_{i_j}}{2l_{i_j} + 1}, \min_{1\leq j\leq m} \frac{2\pi\alpha_{i_j} (l_{i_j}+1)}{2l_{i_j} + 3}\right]: \{i_j\}_{j=1}^m\subseteq \{1,\dots,n\}, l_{i_j}\in\mathbb{Z}_{\geq 0}, \sum_{j=1}^m (2\,l_{i_j}+1) =k  \right\}
		\end{align*}
		for any $k\in\mathbb{Z}_{>0}$.

		Note that above we are defining a multiset, hence if an element appears more than once in the definition, then it will appear more than once in the multiset. In particular, in the case of $X=\mathbb{S}^1 \times \mathbb{S}^1$, for all integers $k\geq 0$ we have the following:
		
		\begin{align*}
			\dgmR_0(X;\mathbb{F})&=\{(0,\infty)\},\\
			\dgmR_{2k+1}(X;\mathbb{F})&=\left\{ \left(\frac{2\pi k}{2k + 1},\frac{2 \pi (k+1)}{2k+3} \right], \left(\frac{2\pi k}{2k + 1},\frac{2 \pi (k+1)}{2k+3} \right] \right\},\\
			\dgmR_{4k+2}(X;\mathbb{F})&=\left\{\left(\frac{2\pi k}{2k + 1},\frac{2 \pi (k+1)}{2k+3} \right] \right\},\\
			\dgmR_{4k+4}(X;\mathbb{F}) &=\emptyset.
		\end{align*}
		
		See also the remarks on homotopy types of Vietoris-Rips  complexes of tori after \cite[Proposition 10.2]{adams} and \cite{chaparro}.
	\end{example}

	To be able to prove Theorem \ref{thm:vrapp}, we need the following lemmas:
	\begin{lemma}\label{lem:prod}
		If $E$ and $F$ are injective metric spaces, then so is their $\ell^\infty$-product.
	\end{lemma}
	\begin{proof}
		Let $X$ be a metric space. Note that $(f,g): X \to E \times F$ is $1$-Lipschitz if and only if $f$ and $g$ are $1$-Lipschitz. Given such $f$ and $g$ and a metric embedding $X$ into $Y$, we have $1$-Lipschitz extensions $\tilde{f}$ and $\tilde{g}$ of $f$ and $g$ from $Y$ to $E$ and $F$, respectively. Hence, $(\tilde{f},\tilde{g}): Y \to E \times F$ is a $1$-Lipschitz extension of $(f,g)$. Therefore $E \times F$ is injective.
	\end{proof}
	\begin{lemma}\label{lem:wedge}
		If $E$ and $F$ are injective metric spaces, then so is their metric gluing along any two points.
	\end{lemma}
	\begin{proof}
		Let $p$ and $q$ be points in $E$ and $F$ respectively and $E \vee F$ denote the metric gluing of $E$ and $F$ along $p$ and $q$. We are going to show that $E \vee F$ is hyperconvex, hence injective (see Proposition \ref{prop:hypinj}). We denote the metric on $E \vee F$  by $d$, the metric on $E$ by $d_E$ and the metric on $F$ by $d_F$.
		
		Let $(x_i,r_i)_i, (y_j,s_j)_j$ be such that each $x_i$ is in $E$, each $y_j$ is in $F$, $r_i\geq 0$, $s_j \geq 0$, $$d_E(x_i,x_{i'}) \leq r_i+r_{i'},$$ $$d_F(y_j,y_{j'}) \leq s_j+s_{j'}$$ and $$d(x_i,y_j) \leq r_i + s_j$$ for each $i,i',j,j'$. Define $\epsilon$ by
		
		$$\epsilon:=\max \left\{\inf_i \big(r_i-d_E(x_i,p)\big), \inf_j \big(s_j-d_F(y_j,q)\big)\right\}. $$
		
		Let us show that $\epsilon \geq 0$. If the second element inside the maximum is negative, then there exists $j_0$ such that $d_F(y_{j_0},q)-s_{j_0} > 0.$ Since $d(x_i,y_{j_0})=d_E(x_i,p)+d_F(q,y_{j_0})$ for all $i$, we have
		$$r_i - d_E(x_i,p)=d_F(y_{j_0},q)+(r_i-d(x_i,y_{j_0})) \geq d_F(y_{j_0},q)-s_{j_0}>0.$$
		Therefore the first element inside the maximum is nonnegative. Hence $\epsilon\geq 0$.
		
		Without loss of generality, let us assume that $$\epsilon=\inf_i \big( r_i-d_E(x_i,p)\big) \geq 0. $$
		This implies that the non-empty closed ball $\overline{B}_\epsilon(q,F)$ is contained in $\overline{B}_{r_i}(x_i,E \vee F)$ for all $i$. Now, for each $j$, we have
		\begin{align*}
			\epsilon+s_j&= \inf_i \big(r_i - d_E(x_i,p) + s_j \big)\\
			&\geq \inf_i\big( d(x_i,y_j) - d_E(x_i,p)\big) \\
			&= d_F(y_j,q).
		\end{align*}
		Therefore,
		$$\left(\bigcap_i \overline{B}_{r_i}(x_i, E \vee F)\right) \cap \left(\bigcap_j \overline{B}_{s_j}(y_j,E \vee F)\right)  \supseteq \overline{B}_\epsilon(q,F) \cap \left(\bigcap_j \overline{B}_{s_j}(y_j, F)\right)\neq\emptyset$$
		where the right-hand side is non-empty by hyperconvexity of $F$.
	\end{proof}
	
	\begin{proof}[Proof of Theorem \ref{thm:vrapp}]
		
		\noindent
		(1) Let $E$ and $F$ be injective metric spaces containing $X$ and $Y$ respectively. Let $E \times F$ denote the $\ell^\infty$-product of $E$ and $F$. Note that for each $r>0$, $$B_r(X \times Y,E \times F)=B_r(X,E) \times B_r(Y,F).$$ Hence, by the (standard) K\"unneth formula \cite[Theorem 58.5]{munkres2018elements}, $$\Hom_*(B_r(X \times Y, E \times F);\mathbb{F}) \cong \Hom_*(B_r(X,E);\mathbb{F}) \otimes \Hom_*(B_r(Y,F);\mathbb{F}). $$ Now, the result follows from Lemma \ref{lem:prod} and Theorem \ref{theorem:isom}.
		
		\smallskip
		\noindent
		(2) Let $E$ and $F$ be as above and $E \vee F$ denote metric gluing of $E$ and $F$ along $p$ and $q$. Note that $$B_r(X \vee Y,E \vee F) = B_r(X,E) \vee B_r(Y,F). $$ Hence, by \cite[Corollary 2.25]{h01}, $$\Hom_*(B_r(X \vee Y, E \vee F);\mathbb{F}) \cong \Hom_*(B_r(X,E);\mathbb{F}) \oplus \Hom_*(B_r(Y,F);\mathbb{F}). $$
		Now, the result follows from Lemma \ref{lem:wedge} and Theorem \ref{theorem:isom}.
	\end{proof}
	
	\section{Application: homotopy types of $\mathrm{VR}_r(X)$ for $X\in\{\mathbb{S}^1,\mathbb{S}^2,\mathbb{CP}^n\}$}\label{sec:hom-types}
	
	In a series of papers \cite{k83,katz-s2,katz-s1,katz-cpn} M. Katz studied the filling radius of spheres and complex projective spaces. In this sequence of papers, Katz developed a notion of Morse theory for the diameter function $\mathrm{diam}:\mathrm{pow}(X)\rightarrow \mathbb{R}$ over a given metric space. By characterizing critical points of the diameter function on each of the spaces $\mathbb{S}^1,\mathbb{S}^2,$ and $\mathbb{CP}^n$ he was able to prove some results about the different homotopy types attained by $B_r(X,L^\infty(X))$ for $X\in\{\mathbb{S}^1,\mathbb{S}^2,\mathbb{CP}^n\}$ as $r$ increases. Here, we obtain  some corollaries that follow from combining the   work of M. Katz \cite{katz-s2,katz-s1} with Theorem \ref{theorem:isom}. 
	
	\subsection{The case of spheres with geodesic distance}\label{sec:spheres-geod}
	
	In \cite[Theorem 3.5]{h95}, Hausmann introduced the quantity $r(M)$ for a Riemannian manifold $M$, which is the supremum of those $r>0$ satisfying the following three conditions:
	
	\begin{enumerate}
		\item For all $x,y\in M$ such that $d_M(x,y)<2r$ there is a unique shortest geodesic joining $x$ to $y$. Its length is $d_M(x,y)$.
		
		\item Let $x,y,z,w\in M$ with $d_M(x,y),d_M(y,z),d_M(z,x)<r$ and $w$ be any point on the shortest geodesic joining $x$ to $y$. Then, $d_M(z,w)\leq\max\{d_M(y,z),d_M(z,x)\}$.
		
		\item If $\gamma$ and $\gamma'$ are arc-length parametrized geodesics such that $\gamma(0)=\gamma'(0)$ and if $0\leq s,s' <r$ and $0\leq t \leq1$, then $d_M(\gamma(ts),\gamma'(ts'))\leq d_M(\gamma(s),\gamma(s'))$.
	\end{enumerate}
	
	In particular, it can be checked that  $r(\Sp^n)=\frac{\pi}{2}$ for any $n\geq 1$. Hausmann then proved that if $r(M)>0$, $\vr_r(M)$ is homotopy equivalent to $M$ for any $r\in (0,r(M))$. This theorem is one of the foundational results in Topological Data Analysis, since it provides theoretical basis for the use of the Vietoris-Rips filtration for recovering the homotopy type of the underlying space.

	Then, via Proposition \ref{prop:vrhyperconvex} we obtain that $B_r(M,L^\infty(M))\simeq M$ for $r\in(0,\frac{1}{2}r(M)]$, and therefore $B_r(\Sp^n,L^\infty(\Sp^n))\simeq \Sp^n$ for all $r\in(0,\frac{\pi}{4}]$. In \cite[Remark p.508]{k83} Katz constructs a retraction from $B_r(\mathbb{S}^n,L^\infty(\mathbb{S}^n))$ to $\mathbb{S}^n$ for $r$ in the range $\left(0,\frac{1}{2}\arccos\left(\frac{-1}{n+1}\right)\right]$, which is a larger range than the one guaranteed by Hausmann's result. This suggests that an improvement of Hausmann's results might be possible for the particular case of spheres.

	Indeed, in the special case of spheres, by a refinement of Hausmann's method of proof  (critically relying upon Jung's theorem) we obtain the following theorem which also improves the aforementioned claim by Katz:
	
	\thmSnfirsttype
	
	That this result improves upon Hausmann's follows from the fact that $\arccos\left(\frac{-1}{n+1}\right)\geq \frac{\pi}{2}$ for all integers $n\geq 1$. The proof follows from the fact that with the aid of Jung's theorem, one can modify the lemmas that Hausmann originally used.  See Appendix \ref{sec:app:SnHausmann} for  a detailed proof along these lines which we believe is of independent interest.
	
	\begin{Remark}
		Note that \cite[Proposition 5.3]{AAF17} establishes a result analogous to Theorem \ref{cor:Snfirsttype} for  Vietoris-Rips metric thickenings of $\Sp^n$.
	\end{Remark}
	
	\begin{Remark}\label{rmk:onfirsttype}
		Note that the above theorem implies that for every $n$, $\dgmR_n(\mathbb{S}^n;\mathbb{F})$ contains an interval $I_n$ of the form $(0,d_n]$ where $d_n\geq \arccos\left(\frac{-1}{n+1}\right)$. The corollary does not however guarantee that $d_n$ equals its lower bound, nor that $I_n$ is the unique interval in $\dgmR_n(\mathbb{S}^n;\mathbb{F})$. Cf. Figure \ref{fig:non-unique-int} for an example of a 2-dimensional sphere (with non-round metric) having more than one interval in its 2-dimensional persistence barcode, and see Proposition \ref{prop:filradpersistence} for a general result about $I_n$.
	\end{Remark}
	
	For the particular cases of $\mathbb{S}^1$ and $\mathbb{S}^2$, we have additional information regarding the homotopy types of their Vietoris-Rips $r$-complexes when $r$ exceeds the range contemplated in the above Corollary.
	
	\begin{figure}
		\centering
		\includegraphics[width=.8\textwidth]{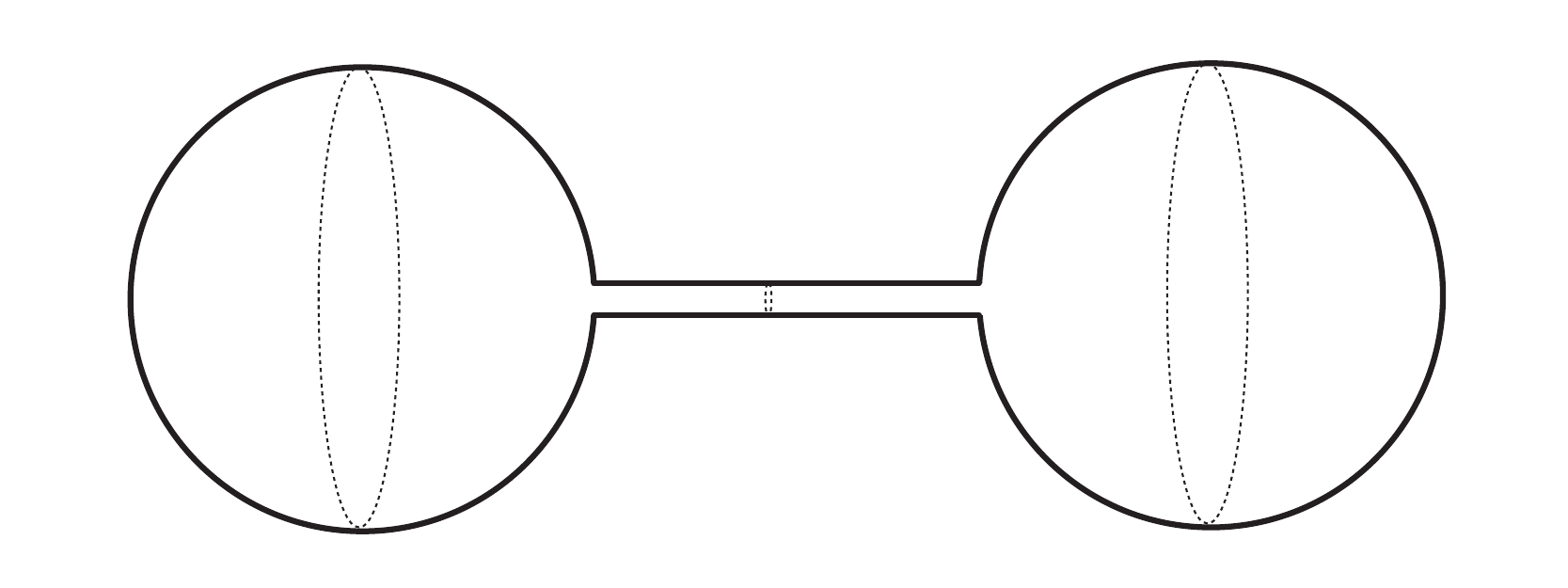}
		\caption{A 2-dimensional sphere with more than one interval in its $\dgmR_2
			$.}
		\label{fig:non-unique-int}
	\end{figure}
	
	\paragraph{The case of $\Sp^1$.}   The complete characterization of the different homotopy types of $\mathrm{VR}_r(\mathbb{S}^1)$ as $r>0$ grows was obtained by  Adamaszek and Adams in \cite{adams}. Their proof is combinatorial in nature and takes place at the simplicial level. 
	
	Below, by invoking Theorem \ref{theorem:isom},  we show how partial results can be obtained from the  work of Katz who directly analyzed the filtration $\big(B_r(\mathbb{S}^1,L^\infty(\mathbb{S}^1))\big)_{r>0}$ via a Morse theoretic argument.

	For each integer $k\geq 1$ let $\lambda_k := \frac{2\pi k}{2k+1}$. Katz proved in \cite{katz-s1} that $B_r(\mathbb{S}^1,L^\infty(\mathbb{S}^1))$ changes homotopy type only when $r = \frac{1}{2}\lambda_k$ for some $k$. In particular, his results imply:

	\begin{corollary}\label{cor:S1secondtype}
		For $r\in \left(\frac{2\pi}{3},\frac{4\pi}{5}\right)$, $\mathrm{VR}_{r}(\mathbb{S}^1)\simeq \mathbb{S}^3$.
	\end{corollary}
	
	\begin{proof}
		$B_r(\mathbb{S}^1,L^\infty(\mathbb{S}^1))$ is homotopy equivalent to $\mathbb{S}^3$ for $r\in\left(\frac{1}{2}\cdot\frac{2\pi}{3},\frac{1}{2}\cdot\frac{4\pi}{5}\right)$ by \cite[Theorem 1.1]{katz-s1}. Hence, the required result follows through Theorem \ref{theorem:isom}.
	\end{proof}

	\paragraph{The case of $\mathbb{S}^2$.} Similar arguments hold for the case of $\mathbb{S}^2$. Whereas the homotopy types of $\vr_r(\mathbb{S}^1)$ for any $r>0$ are known \cite{adams}, we are not aware of  similar results for $\mathbb{S}^2$. Below, $E_6$ is the binary tetrahedral group. 
	
	\begin{corollary}\label{cor:S2secondtype}
		For $r\in \left(\arccos\left(-\frac{1}{3}\right),\arccos \left(-\frac{1}{\sqrt{5}}\right)\right)$, 
		$\mathrm{VR}_{r}(\mathbb{S}^2)\simeq \mathbb{S}^2*\mathbb{S}^3\slash E_6$.
	\end{corollary}
	\begin{proof}
		$B_r(\mathbb{S}^2,L^\infty(\mathbb{S}^2))$ is homotopy equivalent to the topological join of $\mathbb{S}^2$ and $\mathbb{S}^3\slash E_6$  for $r\in \left(\frac{1}{2}\cdot\arccos\left(-\frac{1}{3}\right),\frac{1}{2}\cdot\arccos\left(-\frac{1}{\sqrt{5}}\right)\right)$ by \cite[Theorem 7.1]{katz-s1}. Hence, if we apply Theorem \ref{theorem:isom}, we have the required result.
	\end{proof}
	
	\begin{Remark}
		Observe that $\mathbb{S}^3\slash E_6=\mathrm{SO}(3)/A_4$ where $A_4$ is the tetrahedral group (see \cite[Remark 5.6]{AAF17}).
	\end{Remark}
	
	\begin{Remark}\label{rmk:snintlbd}
		As already pointed out in Remark \ref{rmk:onfirsttype}, by virtue of Theorem \ref{cor:Snfirsttype}, $(0,d_n]\in\dgmR_n(\mathbb{S}^n;\mathbb{F})$ for some $d_n\geq\arccos\left(\frac{-1}{n+1}\right)$. Moreover, since for $n=1$ and $n=2$ we know (by Corollary \ref{cor:S1secondtype} and \ref{cor:S2secondtype}) that the homotopy type changes after $\arccos\left(\frac{-1}{n+1}\right)$, we  conclude that $\dgmR_n(\mathbb{S}^n;\mathbb{F})$ contains $\big(0,\arccos\left(\frac{-1}{n+1}\right)\big]$ for $n=1$ and $n=2$ and  that this is the unique interval in $\dgmR_n(\mathbb{S}^n;\mathbb{F})$ starting at $0$. Surprisingly, it is currently unknown how the homotopy type of $\mathrm{VR}_{r}(\mathbb{S}^n)$ changes after $\arccos\left(\frac{-1}{n+1}\right)$ for $n\geq 3$. But, still, in Section \ref{sec:fil} we will be able to show that $\big(0,\arccos\left(\frac{-1}{n+1}\right)\big]\in\dgmR_n(\mathbb{S}^n;\mathbb{F})$ for general $n$  via arguments involving the filling radius (cf. Proposition \ref{prop:filradpersistence}). In particular, this implies that the homotopy type of $\mathrm{VR}_{r}(\mathbb{S}^n)$ must change after the critical point $r=\arccos\left(\frac{-1}{n+1}\right)$ since the fundamental class dies after that point, even though we still do not know ``how'' the homotopy type changes. Moreover, since $\mathrm{VR}_{r}(\mathbb{S}^n)$ is homotopy equivalent to $\mathbb{S}^n$ for any $r\in(0,\arccos\left(\frac{-1}{n+1}\right)]$, we know that for any interval $I\in\dgmR_n(\mathbb{S}^n;\mathbb{F})$ with $I\neq (0,\arccos\left(\frac{-1}{n+1}\right)]$, the left endpoint of $I$ must be greater than or equal to $\arccos\left(\frac{-1}{n+1}\right)$.
	\end{Remark}
	
	The following sub-conjecture of \cite[Conjecture 5.7]{AAF17} is still open except for the  $n=1$ and $n=2$ cases; see also \cite[Theorem 5.4]{AAF17}.
	
	\begin{conjecture}\label{conj:sn2ndhttp}
		For any $n\in\mathbb{Z}_{> 0}$, there exists an $\varepsilon>0$ such that
		$$\mathrm{VR}_{r}(\mathbb{S}^n)\simeq\mathbb{S}^n*\big(\mathrm{SO}(n+1)/A_{n+2}\big)$$
		for any $r\in\left(\arccos\left(\frac{-1}{n+1}\right),\arccos\left(\frac{-1}{n+1}\right)+\varepsilon\right)$, where $A_{n+2}$ is the alternating group of degree $n+2$.
	\end{conjecture}
	
	\begin{Remark}
		To see that Conjecture \ref{conj:sn2ndhttp} is a sub-conjecture of \cite[Conjecture 5.7]{AAF17}, observe that $\mathbb{S}^n*\big(\mathrm{SO}(n+1)/A_{n+2}\big)\cong\Sigma^{n+1}\,(\mathrm{SO}(n+1)/A_{n+2})$ for any non-negative integer $n$. It is a special case of the following more general homeomorphism:
		$$\mathbb{S}^n*X\cong\Sigma^{n+1}\,X$$
		for any Hausdorff and locally compact space $X$. This fact can be proved  by induction on $n$ and the associativity of the topological join (see \cite[Lecture 2.4]{fomenko2016homotopical}).
	\end{Remark}

	\subsection{The case of $\mathbb{CP}^n$}
	
	Partial information can be provided for the case of $\mathbb{CP}^n$ as well. First of all, recall that the complex projective line $\mathbb{CP}^1$ with its canonical metric  actually coincides with the sphere $\mathbb{S}^2$. Hence, one can apply Theorem \ref{cor:Snfirsttype} and Corollary \ref{cor:S2secondtype} to $\mathbb{CP}^1$.   The following results can be derived for general $\mathbb{CP}^n$.
	
	\begin{corollary}
		Let $\mathbb{CP}^n$ be the complex projective space with sectional curvature  between $1/4$ and $1$ with canonical metric. Then,
		\begin{enumerate}
			\item There exist $\alpha_n\in \left(0,\arccos\left(-\frac{1}{3}\right)\right]$ such that $\vr_{r}(\mathbb{CP}^n)$ is homotopy equivalent to $\mathbb{CP}^n$ for any $r\in(0,\alpha_n]$.
			
			\item Let $A$ be the space of equilateral 4-tuples in projective lines of $\mathbb{CP}^n$. Let $X$ be the partial join of $A$ and $\mathbb{CP}^n$ where $x \in \mathbb{CP}^n$ is joined to a tuple $a \in A$ by a line segment if $x$ is contained in the projective line determined by $a$. There exists a constant $\beta_n>0$ such that if $$\arccos\left(-\frac{1}{3}\right) < r < \arccos\left(-\frac{1}{3}\right) + \beta_n, $$ then $\vr_{r}(\mathbb{CP}^n)$ is homotopy equivalent to $X$.
		\end{enumerate}    
	\end{corollary}
	\begin{proof}
		By  Hausmann's theorem \cite[Theorem 3.5]{h95}, there exist $\alpha_n>0$ such that $\vr_{r}(\mathbb{CP}^n)$ is homotopy equivalent to $\mathbb{CP}^n$ for any $r\in(0,\alpha_n]$. Also, by \cite[Theorem 8.1]{katz-s1}, $\alpha_n$ cannot be greater than $\arccos(-1/3)$. The second claim is a direct result from Theorem \ref{theorem:isom} and \cite[Theorem 8.1]{katz-s1}.
	\end{proof}   
	
	\subsection{The case of spheres with $\ell^\infty$-metric} \label{sec:spheres-inf}

	The Vietoris-Rips filtration of $\mathbb{S}^1$ with the usual geodesic metric is quite challenging to understand \cite{adams}. However, it turns out that if we change its underlying metric, the situation becomes very simple. Throughout  this section, all metric spaces of interest are embedded in $(\mathbb{R}^n,\ell^\infty)$ and are endowed with the restriction of the ambient space metric. In particular, in this section, we use the following conventions: for any $n\in\mathbb{Z}_{>0}$,
	
	\begin{enumerate}
		\item $\R^n_\infty=(\R^n,\ell^\infty)$.
		
		\item $\mathbb{D}^n_\infty:=(\{(x_1,\dots,x_n)\in\mathbb{R}^n:\sum_{i=1}^n x_i^2\leq 1\},\ell^\infty)$. 
		
		\item $\mathbb{S}^{n-1}_\infty:=(\{(x_1,\dots,x_n)\in\mathbb{R}^n:\sum_{i=1}^n x_i^2=1\},\ell^\infty)$.
		
		\item $\blacksquare^n_\infty:=(\{(x_1,\dots,x_n)\in\mathbb{R}^n:x_i\in [-1,1]\textrm{ for every }i=1,\dots,n\},\ell^\infty)$.
		
		\item $\square^{n-1}_\infty:=(\{(x_1,\dots,x_n)\in\mathbb{R}^n:(x_1,\dots,x_n)\in\blacksquare^n_\infty \textrm{ and }x_i=\pm 1\textrm{ for some }i=1,\dots,n\},\ell^\infty)$.
	\end{enumerate}
	
	Note that $\blacksquare^n_\infty$ is just the unit closed $\ell^\infty$-ball around the origin in $\R^n_\infty$ and $\square^{n-1}_\infty$ is its boundary.

	The following theorem by K{\i}l{\i}{\c{c}} and Ko{\c{c}}ak is the motivation of this subsection.
	
	\begin{theorem}[{\cite[Theorem 2]{kilicc2016tight}}]\label{thm:tslinftyplane}
		Let $X$ and $Y$ be subspaces of $\R^2_\infty$. If $Y$ contains $X$, is closed, geodesically convex,\footnote{I.e., for any $p,q\in Y$, there exists at least one geodesic in $\R^2_\infty$ between $p$ and $q$ which is fully contained in $Y$} and minimal (with respect to  inclusion) with these properties, then $Y$ is the tight span of $X$.
	\end{theorem}
	
	Theorem \ref{thm:tslinftyplane} has a number of interesting consequences.
	
	\begin{lemma}\label{lemma:squaretightspan}
		$\blacksquare^2_\infty$ is the tight span of $\square^1_\infty$. Moreover,
		$$B_r(\square^1_\infty,\blacksquare^2_\infty)=\begin{cases}[-1,1]^2\backslash [-(1-r),(1-r)]^2  & \text{if }r\in(0,1]\\ [-1,1]^2  & \text{if }r>1.\end{cases}$$
	\end{lemma}

	\begin{proof}
		By Theorem \ref{thm:tslinftyplane}, the first claim is straightforward. The second claim, namely the explicit expression of $B_r(\square^1_\infty,\blacksquare^2_\infty)$ is also obvious since we are using the $\ell^\infty$-norm.
	\end{proof}

	\begin{corollary}
		$B_r(\square^1_\infty,\blacksquare^2_\infty)$ is homotopy equivalent to $\mathbb{S}^1$ for $r\in(0,1]$ and contractible for $r>1$. Hence, for any field $\mathbb{F}$, it holds that:
		$$\dgmR_k(\square^1_\infty,\mathbb{F})=\begin{cases}\{(0,\infty)\} & \text{if }k=0\\ \{(0,2]\} & \text{if }k=1\\ \emptyset & \text{if }k\geq 2.\end{cases}$$
	\end{corollary}
	\begin{proof}
		Apply Lemma \ref{lemma:squaretightspan} and Theorem \ref{theorem:isom}.
	\end{proof}
	
	Interestingly, one can also prove the following result.
	
	\begin{lemma}\label{lemma:circletightspan}
		$\mathbb{D}^2_\infty$ is the tight span of $\mathbb{S}^1_\infty$. Moreover,
		$$B_r(\mathbb{S}^1_\infty,\mathbb{D}^2_\infty)=\mathbb{D}^2_\infty\backslash V_r$$
		for any $r>0$, where
		$$V_r:=\bigcap_{(p,q)\in\{r,-r\}^2}\big\{(x,y)\in\mathbb{R}^2:(x-p)^2+(y-q)^2\leq 1 \big\}.$$
		In particular, for $r>\frac{1}{\sqrt{2}}$ we have $V_r=\emptyset$ so that $B_r(\mathbb{S}^1_\infty,\mathbb{D}^2_\infty)=\mathbb{D}^2_\infty$ (see Figure \ref{fig:d2-s1-tight-span}).
	\end{lemma}
	\begin{proof}
		By Theorem \ref{thm:tslinftyplane}, the first claim is straightforward.

		First, fix arbitrary $(z+t,w+s)\in B_r(\mathbb{S}^1_\infty,\mathbb{D}^2_\infty)$ where $z^2+w^2=1$ and $t,s\in (-r,r)$. Suppose $z\geq 0$ and $w\geq 0$. Then,
		\begin{align*}
			(z+t+r)^2+(w+s+r)^2&=z^2+w^2+(t+r)^2+(s+r)^2+2z(t+r)+2w(s+r)> 1
		\end{align*}
		because of the assumptions on $z,w,t$, and $s$. Therefore, $(z+t,w+s)\notin V_r$ so that $(z+t,w+s)\in\mathbb{D}^2_\infty\backslash V_r$. By symmetry, the same result holds for other possible sign combinations of $z$ and $w$. Hence, we have $B_r(\mathbb{S}^1_\infty,\mathbb{D}^2_\infty)\subseteq\mathbb{D}^2_\infty\backslash V_r$.

		Now, fix arbitrary $(x,y)\in\mathbb{D}^2_\infty\backslash V_r$. Since $(x,y)\notin V_r$, without loss of generality, one can assume that $(x+r)^2+(y+r)^2>1$. Also, $x^2+y^2\leq 1$ since $(x,y)\in\mathbb{D}^2_\infty$. Then, there must be some $t\in[0,r)$ such that $(x+t)^2+(y+t)^2=1$. It follows that $(x,y)\in B_r(\mathbb{S}^1_\infty,\mathbb{D}^2_\infty)$. Since $(x,y)$ is an arbitrary point in $\mathbb{D}^2_\infty\backslash V_r$ it follows that $\mathbb{D}^2_\infty\backslash V_r\subseteq B_r(\mathbb{S}^1_\infty,\mathbb{D}^2_\infty)$.

		With this we conclude that $B_r(\mathbb{S}^1_\infty,\mathbb{D}^2_\infty)=\mathbb{D}^2_\infty\backslash V_r$, as we wanted.
	\end{proof}
	
	\begin{corollary}
		$B_r(\mathbb{S}^1_\infty,\mathbb{D}^2_\infty)$ is homotopy equivalent to $\mathbb{S}^1$ for $r\in\left(0,\frac{1}{\sqrt{2}}\right]$ and contractible for $r>\frac{1}{\sqrt{2}}$. Hence, for any field $\mathbb{F}$, it holds that:
		$$\dgmR_k(\mathbb{S}^1_\infty,\mathbb{F})=\begin{cases}\{(0,\infty)\} & \text{if }k=0\\ \{(0,\sqrt{2}]\} & \text{if }k=1\\ \emptyset & \text{if }k\geq 2.\end{cases}$$
	\end{corollary}
	\begin{proof}
		Apply Lemma \ref{lemma:circletightspan} and Theorem \ref{theorem:isom}.
	\end{proof}
	
	Moreover, it turns out that, despite the fact that Theorem \ref{thm:tslinftyplane} is restricted to subsets of $\R^2$, Lemma \ref{lemma:squaretightspan} can be generalized to  arbitrary dimensions.
	
	\begin{lemma}\label{lemma:nsquaretightspan}
		For any $n\in\mathbb{Z}_{>0}$, $\blacksquare^n_\infty$ is the tight span of $\square^{n-1}_\infty$. Moreover, 
		$$B_r(\square^{n-1}_\infty,\blacksquare^n_\infty)=\begin{cases}[-1,1]^n\backslash [-(1-r),1-r]^n & \text{if }r\in(0,1]\\ [-1,1]^n & \text{if }r>1\end{cases}$$
	\end{lemma}
	\begin{proof}
		When $n\geq 3$ one cannot invoke Theorem \ref{thm:tslinftyplane} since it does not hold for general $n$ (see \cite[Example 5]{kilicc2016tight}). We will instead directly prove that $\blacksquare^n_\infty$ is the tight span of $\square^{n-1}_\infty$.

		First, observe that $\blacksquare^n_\infty=\overline{B}_1(O,\R^n_\infty)$, where $O=(0,\dots,0)$ is the origin, is hyperconvex by Lemma \ref{lemma:intercloballhyper}.

		Therefore, in order to show that $\blacksquare^n_\infty$ is indeed the tight span of $\square^{n-1}_\infty$, it is enough to show that there is no proper hyperconvex subspace of $\blacksquare^n_\infty$ containing $\square^{n-1}_\infty$. Suppose this is not true. Then there exists a proper hyperconvex subspace $X$ such that $\square^{n-1}_\infty\subseteq X\subsetneq\blacksquare^n_\infty$. Choose $p=(x_1,\dots,x_n)\in\blacksquare^n_\infty\backslash X$. Without loss of generality, one can assume $x_1\geq\dots\geq x_n$. Now, let $$p_0:=(x_1-(x_n+1),x_2-(x_n+1),\dots,-1)$$ and $$p_1:=(1,x_2+(1-x_1),\dots,x_n+(1-x_1)).$$ See Figure \ref{fig:p0p1}. Then, it is clear that $p_0,p_1\in\square^{n-1}_\infty\subseteq X$ and
		$$\Vert p_0-p_1\Vert_\infty=(x_n+1)+(1-x_1)=\Vert p_0-p\Vert_\infty+\Vert p-p_1\Vert_\infty.$$ 
		\begin{figure}
			\centering
			\includegraphics[width=0.2\textheight]{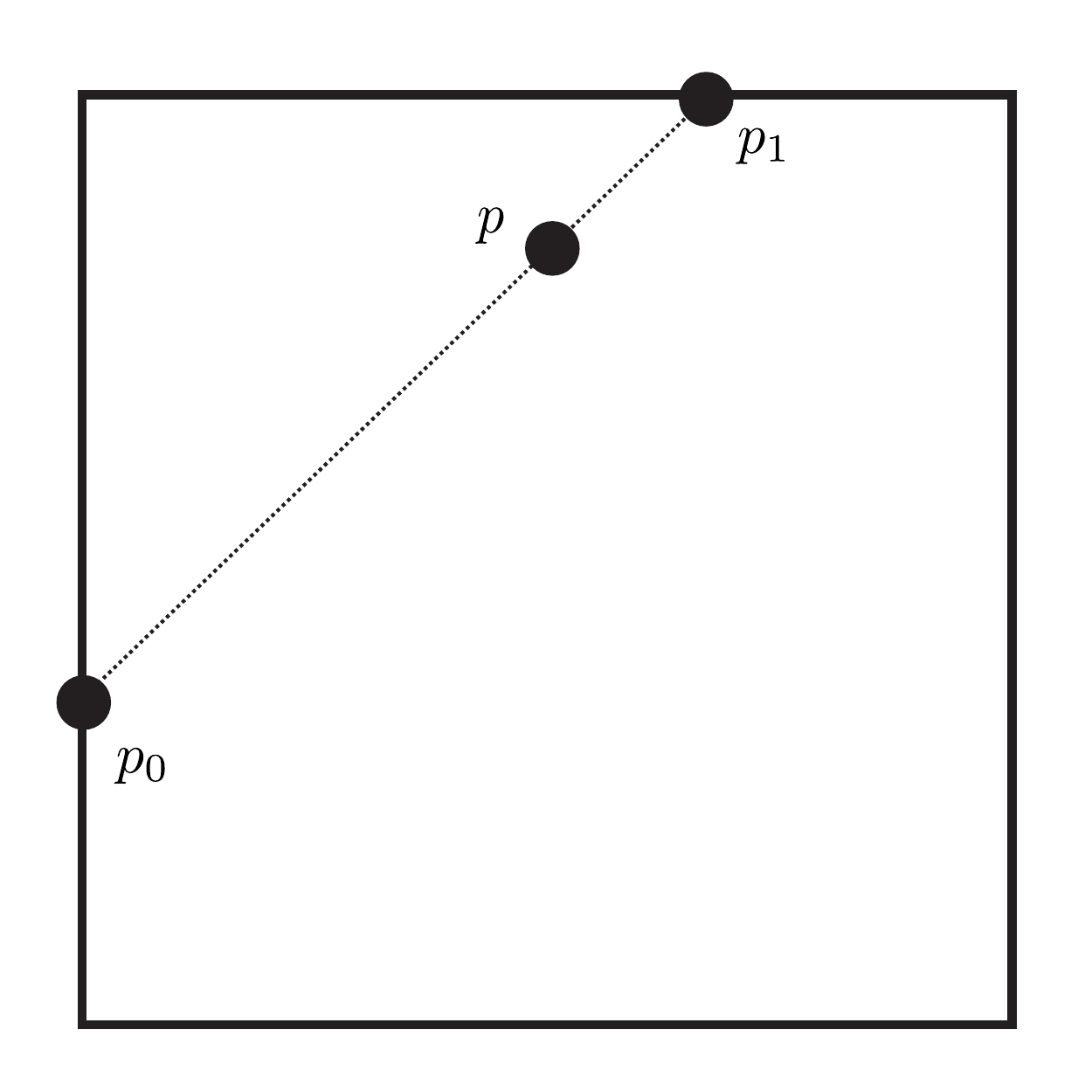}
			\caption{The points $p,p_0$ and $p_1$ in the proof of Lemma \ref{lemma:nsquaretightspan}.}
			\label{fig:p0p1}
		\end{figure}
		Therefore, since $X$ is hyperconvex, we know that
		$$\overline{B}_{\Vert p_0-p\Vert_\infty}(p_0,X)\cap\overline{B}_{\Vert p-p_1\Vert_\infty}(p_1,X)\neq\emptyset.$$
		However, note that
		$$\overline{B}_{\Vert p_0-p\Vert_\infty}(p_0,X)\cap\overline{B}_{\Vert p-p_1\Vert_\infty}(p_1,X)\subseteq\overline{B}_{\Vert p_0-p\Vert_\infty}(p_0,\R^n_\infty)\cap\overline{B}_{\Vert p-p_1\Vert_\infty}(p_1,\R^n_\infty)=\{p\}.$$
		This means that $p\in X$, which is contradiction. Hence, no such $X$ exists. Therefore, $\blacksquare^n_\infty$ is the tight span of $\square^{n-1}_\infty$, as we required.

		The second claim, namely the explicit expression of $B_r(\square^{n-1}_\infty,\blacksquare^n_\infty)$ is obvious since we are using the $\ell^\infty$-norm.
	\end{proof}
	
	\begin{corollary}
		For any $n\in\mathbb{Z}_{>0}$, $B_r(\square^{n-1}_\infty,\blacksquare^n_\infty)$ is homotopy equivalent to $\mathbb{S}^{n-1}$ for $r\in(0,1]$ and contractible for $r>1$. Hence, for any field $\mathbb{F}$, it holds that:
		$$\dgmR_k(\square^{n-1}_\infty,\mathbb{F})=\begin{cases}\{(0,\infty)\} & \text{if }k=0\\ \{(0,2]\} & \text{if }k=n-1\\ \emptyset & \text{otherwise}\end{cases}$$
		for $n\geq 2$, and
		$$\dgmR_k(\square^0_\infty,\mathbb{F})=\begin{cases}\{(0,\infty),(0,2]\} & \text{if }k=0\\ \emptyset & \text{if }k\geq 1.\end{cases}$$
	\end{corollary}
	\begin{proof}
		Apply Lemma \ref{lemma:nsquaretightspan} and Theorem \ref{theorem:isom}.
	\end{proof}

	\begin{Remark}
		It seems of interest to study the homotopy types of Vietoris-Rips complexes of ellipsoids with the $\ell^\infty$-metric, cf. \cite{henry-ellipses}.
	\end{Remark}

	Here, observant readers would have already  noticed that we do not need to use the tight spans of $\mathbb{S}^1_\infty$ and $\square^{n-1}_\infty$ in order to apply Theorem \ref{theorem:isom} since $\R^n_\infty$ itself is an injective metric space for any $n\in\mathbb{Z}_{>0}$. In particular, the persistent homology of $\square^{n-1}_\infty$ is simpler to compute if we use $\R^n_\infty$ as an ambient space. However, we believe that it is worth  clarifying  what are the tight spans of $\mathbb{S}^1_\infty$ and $\square^{n-1}_\infty$ since the exact shape of tight spans are largely mysterious in general. 
	\linebreak
	
	We do not know whether $\mathbb{D}^n_\infty$ is the tight span of $\mathbb{S}^{n-1}_\infty$ for general $n$. However, if we use $\R^n_\infty$ as an ambient injective metric space, we are still able to compute its persistent homology.

	\thmsphinf
	
	\begin{proof}
		First, let's prove that $B_r(\mathbb{S}^{n-1}_\infty,\mathbb{D}^n_\infty)$ is a deformation retract of $B_r(\mathbb{S}^{n-1}_\infty,\R^n_\infty)$. Consider the map $P_n:\{(x_1,\dots,x_n)\in\mathbb{R}^n:\sum_{i=1}^n x_i^2\geq 1\}\rightarrow\mathbb{S}^{n-1}_\infty$ such that for any $(x_1,\dots,x_n)\in\{(x_1,\dots,x_n)\in\mathbb{R}^n:\sum_{i=1}^n x_i^2\geq 1\}$, $P_n(x_1,\dots,x_n)$ is the unique point of $\mathbb{S}^{n-1}_\infty$ such that $\Vert (x_1,\dots,x_n)-P_n(x_1,\dots,x_n) \Vert_\infty=\inf_{(y_1,\dots,y_n)\in\mathbb{S}^{n-1}_\infty}\Vert (x_1,\dots,x_n)-(y_1,\dots,y_n)\Vert_\infty$. Observe that it is easy (but, very tedious) to prove that $P_n$ is well-defined, continuous, and that $P_n|_{\mathbb{S}^{n-1}_\infty}=\mathrm{id}_{\mathbb{S}^{n-1}_\infty}$.

		Now, for any $r>0$, consider the following homotopy 
		\begin{align*}
			h_{n,r}:B_r(\mathbb{S}^{n-1}_\infty,\R^n_\infty)\times [0,1]&\longrightarrow B_r(\mathbb{S}^{n-1}_\infty,\R^n_\infty)\\
			(x_1,\dots,x_n,t)&\longmapsto \begin{cases}(x_1,\dots,x_n)&\text{if }(x_1,\dots,x_n)\in\mathbb{D}^n_\infty\\(1-t)(x_1,\dots,x_n)+tP_n(x_1,\dots,x_n)&\text{if }(x_1,\dots,x_n)\notin\mathbb{D}^n_\infty\end{cases}
		\end{align*}
		
		The only subtle point is ascertaining whether the image of this map is contained in $B_r(\mathbb{S}^{n-1}_\infty,\R^n_\infty)$. But, for this note that $\Vert (x_1,\dots,x_n)-P_n(x_1,\dots,x_n) \Vert_\infty<r$ because of the definition of $P_n$ and the fact that $(x_1,\dots,x_n)\in B_r(\mathbb{S}^{n-1}_\infty,\R^n_\infty)$. Therefore, both $(x_1,\dots,x_n)$ and $P_n(x_1,\dots,x_n)$ belong to $B_r(P_n(x_1,\dots,x_n),\R^n_\infty)$ so that the linear interpolation is also contained in $B_r(P_n(x_1,\dots,x_n),\R^n_\infty)\subset B_r(\mathbb{S}^{n-1}_\infty,\R^n_\infty)$. Hence, one can conclude that $B_r(\mathbb{S}^{n-1}_\infty,\mathbb{D}^n_\infty)$ is a deformation retract of $B_r(\mathbb{S}^{n-1}_\infty,\R^n_\infty)$.\medskip

		Next, let's prove that $B_r(\mathbb{S}^{n-1}_\infty,\mathbb{D}^n_\infty)=\mathbb{D}^n_\infty\backslash V_{n,r}$. First, fix arbitrary $(z_1+t_1,\dots,z_n+t_n)\in B_r(\mathbb{S}^{n-1}_\infty,\mathbb{D}^n_\infty)$ where $\sum_{i=1}^n z_i^2=1$ and $t_i\in (-r,r)$ for all $i=1,\dots,n$. Consider the case of $z_i\geq 0$ for all $i=1,\dots,n$. Then,
		
		\begin{align*}
			\sum_{i=1}^n (z_i+t_i+r)^2&=\sum_{i=1}^n z_i^2+\sum_{i=1}^n (t_i+r)^2+\sum_{i=1}^n 2z_i(t_i+r)>1
		\end{align*}
		
		\noindent because of the assumptions on $\{z_i\}_{i=1}^n$ and $\{t_i\}_{i=1}^n$. Therefore, $(z_1+t_1,\dots,z_n+t_n)\notin V_{n,r}$ so that $(z_1+t_1,\dots,z_n+t_n)\in\mathbb{D}^n_\infty\backslash V_{n,r}$. By symmetry, the same result holds for other possible sign combinations of the $z_i$'s. Hence, we have $B_r(\mathbb{S}^{n-1}_\infty,\mathbb{D}^n_\infty)\subseteq\mathbb{D}^n_\infty\backslash V_{n,r}$.

		Now, fix arbitrary $(x_1\dots,x_n)\in\mathbb{D}^n_\infty\backslash V_{n,r}$. Since $(x_1,\dots,x_n)\notin V_{n,r}$, without loss of generality, one can assume that $\sum_{i=1}^n (x_i+r)^2>1$. Also, $\sum_{i=1}^n x_i^2\leq 1$ since $(x_1,\dots,x_n)\in\mathbb{D}^n_\infty$. Then, there must be some $t\in[0,r)$ such that $\sum_{i=1}^n (x_i+t)^2=1$. It follows that $(x_1,\dots,x_n)\in B_r(\mathbb{S}^{n-1}_\infty,\mathbb{D}^n_\infty)$. Since $(x_1,,\dots,x_n)$ is an arbitrary point in $\mathbb{D}^n_\infty\backslash V_{n,r}$ it follows that $\mathbb{D}^n_\infty\backslash V_{n,r}\subseteq B_r(\mathbb{S}^{n-1}_\infty,\mathbb{D}^n_\infty)$.

		With this we conclude that $B_r(\mathbb{S}^{n-1}_\infty,\mathbb{D}^n_\infty)=\mathbb{D}^n_\infty\backslash V_{n,r}$. This completes the proof.
	\end{proof}
	
	\begin{corollary}
		For any $n\geq 2$, $B_r(\mathbb{S}^{n-1}_\infty,\R^n_\infty)$ is homotopy equivalent to $\mathbb{S}^{n-1}$ for $r\in\left(0,\frac{1}{\sqrt{n}}\right]$ and contractible for $r>\frac{1}{\sqrt{n}}$. Hence, for any field $\mathbb{F}$, it holds that:
		$$\dgmR_k(\mathbb{S}^{n-1}_\infty,\mathbb{F})=\begin{cases}\{(0,\infty)\} & \text{if }k=0\\ \left\{\left(0,\frac{2}{\sqrt{n}}\right]\right\} & \text{if }k=n-1\\ \emptyset & \text{otherwise.}\end{cases}$$
	\end{corollary}
	\begin{proof}
		Apply Theorem \ref{thm:ncircletightspan} and Theorem \ref{theorem:isom}.
	\end{proof}

	\section{Application: hyperbolicity and persistence}\label{sec:hyperbolicity}
	
	One can reap benefits from the fact that one can choose \emph{any} metric homotopy pairing in the statement of Theorem \ref{theorem:isom}, not just the Kuratowski functor. 
	
	In this section, we will see one such example which arises from the interplay between the \emph{hyperbolicity} of the geodesic metric space $X$ and its tight span $E(X)$ (see Example \ref{ex:inj} to recall the definition of  tight span).

	We first recall the notion of hyperbolicity. 
	
	\begin{definition}[$\delta$-hyperbolicity]\label{def:hyperbolicity}
		A metric space $(X,d_X)$ is called $\delta$-hyperbolic, for some constant $\delta\geq 0$, if
		$$d_X(w,x)+d_X(y,z)\leq\max\{d_X(w,y)+d_X(x,z),d_X(x,y)+d_X(w,z)\}+\delta,$$
		for all quadruples of points $w,x,y,z\in X$. If a metric space is $\delta$-hyperbolic for some $\delta\geq 0$, it is said to be hyperbolic. 
		
		The \emph{hyperbolicity} $\mathrm{hyp}(X)$ of $X$ is defined as the infimal $\delta\geq 0$ such that $X$ is $\delta$-hyperbolic. A metric space is said to be \emph{hyperbolic} if $\hyp(X)$ is finite.
	\end{definition}
	
	For a more concrete development on the geometry of hyperbolic metric spaces and its applications (especially to group theory), see \cite{bonk2011embeddings,ghg}.
	
	\begin{example}
		Here are some examples of hyperbolic spaces:
		\begin{enumerate}
			\item Metric trees are $0$-hyperbolic spaces.
			
			\item All compact Riemannian manifolds are trivially hyperbolic spaces. More interestingly, among unbounded manifolds,  Riemannian manifolds with strictly negative sectional curvature are hyperbolic spaces. Observe that ``strictly negative''  sectional curvature is a necessary condition (for example, consider the Euclidean plane $\mathbb{R}^2$).
		\end{enumerate}
	\end{example}
	
	The following proposition guarantees that the tight span $E(X)$ preserves the hyperbolicity of the underyling space $X$ with controlled distortion.
	
	\begin{proposition}[{\cite[Proposition 1.3]{l13}}]\label{prop:hyperkura}
		If $X$ is a $\delta$-hyperbolic geodesic metric space for some $\delta\geq 0$, then its tight span $E(X)$ is also $\delta$-hyperbolic. 
		
		Moreover,
		$$B_r(X,E(X))=E(X)$$
		for any $r>\delta$. 
	\end{proposition}
	
	\begin{Remark}
		Note that since $X$ embeds isometrically into $E(X)$, the above implies that $$\mathrm{hyp}(E(X))=\mathrm{hyp}(X).$$
	\end{Remark}
	
	The following corollary was already established  by Gromov (who attributes it to Rips) in \cite[Lemma 1.7.A]{ghg}. The proof given by Gromov operates at the simplicial level.  By invoking Proposition \ref{prop:hyperkura} we obtain an alternative proof, which instead of operating the simplicial level, exploits the isometric embedding  of $X$ into its tight span $E(X)$ (which is a compact contractible space).
	
	\begin{corollary}\label{cor:Ripslemma}
		If $X$ is a hyperbolic geodesic metric space, then $\vr_{2r}(X)$ is contractible for any $r>\mathrm{hyp}(X)$.
	\end{corollary}
	\begin{proof}
		Choose arbitrary $r>\mathrm{hyp}(X)$. Then, there is $\delta\in[\mathrm{hyp}(X),r)$ such that $X$ is $\delta$-hyperbolic.

		By Proposition \ref{prop:vrhyperconvex}, we know that $\vr_{2r}(X)$ is homotopy equivalent to $B_r(X,E(X))$. But, by Proposition \ref{prop:hyperkura}, $B_r(X,E(X))=E(X)$. Since $E(X)$ is contractible by Corollary \ref{cor:injectivecont}, $\vr_{2r}(X)$ is contractible.
	\end{proof}
	
	As a consequence one can bound the length of intervals in the persistence barcode of hyperbolic spaces.
	
	\begin{corollary}
		If $X$ is a hyperbolic geodesic metric space, then for any $k\geq 1$ and $I=(u,v]\in\dgmR_k(X;\mathbb{F})$, we have $v\leq 2\,\mathrm{hyp}(X)$. In particular, $\mathrm{length}(I)\leq 2\,\mathrm{hyp}(X)$.
	\end{corollary}
	\begin{proof}
		Apply Corollary \ref{cor:Ripslemma}.
	\end{proof}
	
	Observe that metric trees are both $0$-hyperbolic and hyperconvex. A recent paper by Joharinad and Jost \cite{joharinad2020topological} analyzes the persistent homology of metric spaces satisfying the hyperconvexity condition (which is equivalent to  injectivity)  as well as that of spaces satisfying a  relaxed version of hyperconvexity.

	\section{Application: the filling radius, spread, and persistence}\label{sec:fil}

	In this section, we recall the notions of spread and filling radius, as well as their relationship. In particular, we prove a number of statements about the filling radius of a closed connected manifold. Moreover, we consider a generalization of the filling radius and  also define a strong notion of filling radius which is akin to the so called \emph{maximal persistence} in the realm of topological data analysis.

	\subsection{Spread}\label{sec:spread}
	We recall a metric concept called \emph{spread}. The following definition is a variant of the one given in \cite[Lemma 1]{k83}.
	
	\begin{definition}[$N$-spread]
		For any integer $N\in\mathbb{Z}_{>0}$, the \emph{$N$-th spread} $\spread_N(X)$ of a metric space $(X,d_X)$ is the infimal $r>0$ such that there exists a subset $A$ of $X$ with cardinality at most $N$ such that 
		\begin{itemize}
			\item $\diam(A) < r$
			\item $\sup_{x\in X}\inf_{a\in A} d_X(x,a) < r.$
		\end{itemize}
		Finally, the \emph{spread} of $X$ is defined to be $\spread(X):=\inf_N\spread_N(X)$, i.e. the set $A$ is allowed to have arbitrary (finite) cardinality.
	\end{definition}

	\begin{Remark}\label{rem:rad}
		Recall that the \emph{radius} of a compact metric space $(X,d_X)$ is
		$$\rad(X):=\inf_{p\in X}\max_{x\in X}d_X(p,x).$$
		Thus, $\rad(X)=\spread_1(X)$.
	\end{Remark}
	
	\begin{Remark}[The spread of spheres]\label{rem:spread-sn}
		Katz proves in \cite[Theorem 2]{k83} that for all integers $n\geq 1$, $$\spread(\mathbb{S}^n) = \arccos\left(\frac{-1}{n+1}\right).$$
		For example, $\spread(\mathbb{S}^1)=\frac{2\pi}{3}.$ Notice that $\spread(\mathbb{S}^m)\geq \spread(\mathbb{S}^n) \geq \frac{\pi}{2}$ for $m\leq n.$ Katz's proof actually yields that $$\spread_{n+2}(\mathbb{S}^n) = \spread(\mathbb{S}^n)$$ for each $n$.
	\end{Remark}

	\subsection{Bounding barcode length via spread}
	
	Let $(X,d_X)$ be a compact metric space. Recall that for each integer $k\geq 0$, $\dgmR_k(X;\mathbb{F})$ denotes the persistence barcode associated to $\PH_k(\vr_*(X);\mathbb{F})$, the $k$-th persistent homology induced by the Vietoris-Rips filtration of $X$ (see Section \ref{sec:vietoris}).

	The following lemma is due to Katz \cite[Lemma 1]{k83}. 
	
	\begin{lemma}\label{lemma:katzcontraction}
		Let $(X,d_X)$ be a compact metric space. Then, for any $\delta>\frac{\spread(X)}{2}$, there exists a contractible space $U$ such that $X\subseteq U\subseteq B_\delta(X,L^\infty(X))$.
	\end{lemma}
	
	\begin{Remark}
		Note that via the isomorphism theorem, Katz's lemma implies the fact that whenever $I=(0,v]\in \dgmR_*(X)$, then $v\leq \spread(X).$ The lemma does not permit bounding the length of intervals whose left endpoint is strictly greater than zero.
	\end{Remark}
	
	It turns out that we can prove a general version of Lemma \ref{lemma:katzcontraction} for  closed $s$-thickenings $\overline{B}_s(X,L^\infty(X))$ for any $s\geq 0$. 
	
	\begin{lemma}\label{lemma:katzcontractiongen}
		Let $(X,d_X)$ be a compact metric space. Then, for any $s\geq 0$ and $\delta>\frac{\spread(X)}{2}$, there exists a contractible space $U_{s,\delta}$ such that $\overline{B}_s(X,L^\infty(X))\subseteq U_{s,\delta}\subseteq B_{s+\delta}(X,L^\infty(X))$. 
	\end{lemma}
	
	Note that Lemma \ref{lemma:katzcontraction} can be obtained from the case $s=0$ in Lemma \ref{lemma:katzcontractiongen}.  We provide a detailed self-contained proof of this general version in Section \ref{sec:proof-katz-lemma-gen}.

	Armed with Lemma \ref{lemma:katzcontractiongen} and Theorem \ref{theorem:isom} one immediately obtains item (1) in the proposition below: 
	
	\begin{proposition}\label{prop:spread}
		Let $(X,d_X)$ be a compact metric space,  $k\geq 1$, and let $I$ be any interval in $ \dgmR_k(X;\mathbb{F})$. Then,
		\begin{equation}\label{eq:sp-bound}
			\mathrm{length}(I)\leq \spread(X),
		\end{equation}
		and, 
		\begin{equation}\label{eq:rad-bound}
			\mbox{if $I = (u,v]$ for some $0<u< v$, then $v\leq \spread_1(X).$}\end{equation}
	\end{proposition}
	
	\begin{Remark}
		The second part of the proposition above, equation (\ref{eq:rad-bound}), implies that the right endpoint of any interval $I$ (often referred to as the \emph{death time} of $I$) cannot exceed the radius $\mathrm{rad}(X)$ of $X$ (cf. Remark \ref{rem:rad}).
	\end{Remark}
	
	Note that  by \cite[Section 1]{k83}, when $X$ is a geodesic space (e.g. a Riemmanian manifold) we have $\spread(X)\leq \frac{2}{3}\diam(X)$; this means that we have the following universal bound on the length of intervals in the Vietoris-Rips persistence barcode of a geodesic space $X$. 
	
	\begin{corollary}[Bound on length of bars of geodesic spaces] \label{coro:diam-bound}
		Let $X$ be a compact geodesic space. Then, for any $k\geq 1$ and any $I\in\dgmR_k(X;\mathbb{F})$ it holds that $$\mathrm{length}(I)\leq \frac{2}{3}\diam(X).$$
	\end{corollary}
	
	\begin{Remark}We make the following remarks:
		\begin{itemize}
			\item Note that for $k=1$, $\mathbb{S}^1$ achieves equality in the corollary above. Indeed, this follows from \cite{adams} since the longest interval in $\dgmR_k(\mathbb{S}^1)$ corresponds to $k=1$ and is exactly $(0,\frac{2\pi}{3}].$ 
			\item Since $\vr_r(X)$ is contractible for any $r>\diam(X)$, it is clear that $\mathrm{length}(I)\leq \diam(X)$ in general.  The corollary above improves this bound by a factor of $\frac{2}{3}$  when $X$ is geodesic.
			
			\item In \cite{k83} Katz proves that the filling radius of a manifold is bounded above by $\frac{1}{3}$ of its diameter. Our result is somewhat more general than Katz's in two senses: his claim applies to Riemannian manifolds $M$ and only provides information about the interval induced by the fundamental class of the manifold (see Proposition \ref{prop:filradpersistence}).  In contrast, Corollary \ref{coro:diam-bound} applies to any compact geodesic space and in this case it provides the same upper bound for the length any interval in $\dgmR_k(X;\mathbb{F})$, for any $k$.
			\item Besides the proof via Lemma \ref{lemma:katzcontractiongen} and Theorem \ref{theorem:isom} explained above, we provide an alternative direct proof of Proposition \ref{prop:spread} via simplicial arguments. We believe each proof is interesting in its own right.
		\end{itemize}
	\end{Remark}
	
	\begin{proof}[Proof of Proposition \ref{prop:spread} via simplicial arguments]
		Let $\delta>\spread(X)$. It is enough to show that for each $s > 0 $, the map $$\mathrm{H}_k(\vr_s(X);\mathbb{F}) \to \mathrm{H}_k(\vr_{\delta+s}(X);\mathbb{F})$$
		induced by the inclusion is zero. By the definition of spread, we know that there is a nonempty finite subset $A\subseteq X$ such that
		\begin{itemize}
			\item $\diam(A) < \delta$
			\item $\sup_{x\in X}\inf_{a\in A} d_X(x,a) < \delta.$
		\end{itemize}
		Note that then $\mathrm{H}_k(\vr_\delta(A);\mathbb{F})=0$ because $\vr_\delta(A)$ is a simplex. Let $\pi:X \to A$ be a map sending $x$ to a closest point in $A$. Then, $d_X(x,\pi(x))<\delta$ for any $x\in X$ because of the second property of $A$ (moreover, $\pi(x)=x$ if $x\in A$). Observe that, since $\diam(\pi(\sigma))<\delta$ for any simplex $\sigma\in\vr_s(X)$ by the first property of $A$, this map $\pi$ induces a simplicial map from $\vr_s(X)$ to $\vr_\delta(A)$. Hence, one can construct the following composite map $\nu$ from $\vr_s(X)$ to $\vr_{\delta+s}(X)$:
		$$\vr_s(X) \xrightarrow{\pi} \vr_\delta(A)\hookrightarrow \vr_\delta(X)\hookrightarrow \vr_{\delta+s}(X),$$
		
		\noindent where the second and third maps are induced by the canonical inclusions. Observe that this composition of maps induces a map from $\mathrm{H}_k(\vr_s(X))$ to $\mathrm{H}_k(\vr_{\delta+s}(X))$, and this induced map is actually the zero map since $\mathrm{H}_k(\vr_\delta(A);\mathbb{F})=0$. So, it is enough to show that the composite map $\nu$  is contiguous to the canonical inclusion $\vr_{s}(X)\hookrightarrow \vr_{\delta+s}(X)$. Let $\sigma=\{x_0,\dots,x_n\}$ be a subset of $X$ with diameter strictly less than $s$. Let $a_i:=\pi(x_i)$ for $i=0,1,\ldots,n$. Then, we have $$d_X(x_i,a_j) \leq d_X(x_i,x_j)+d_X(x_j,a_j) < \delta +s. $$ Hence, the diameter of the subspace $\{x_1,\dots,x_k,a_1,\dots,a_k \}$ is strictly less than $\delta+s$. This shows the desired contiguity and completes the proof. The proof of equation (\ref{eq:rad-bound}) follows  similar (but simpler) steps and thus we omit it.
	\end{proof}
	
	\begin{Remark}
		Note that whereas the proof of Lemma 1 in \cite{k83} takes place at the level of $L^\infty(X)$,  the proof of Proposition \ref{prop:spread} given above takes place at the level of simplicial complexes and simplicial maps.
	\end{Remark}
	
	\subsubsection{Bounds based on localization of spread}
	One can improve Proposition \ref{prop:spread} by considering a localized version of spread. Note that, in \cite{adams2021geometric}, the authors also built some bounds on the length of barcodes based on certain  notions of size of homology classes.\vspace{\baselineskip}
	
	For an integer $k\geq 0$, a given field $\mathbb{F}$, and a metric space $X$, let $$\mathrm{Spec}_k(X,\mathbb{F}):=\bigcup_{r>0}\bigg(\Hom_k(\vr_r(X);\mathbb{F})\backslash\{0\}\times\{r\}\bigg)$$ be the $k$-th Vietoris-Rips \emph{homological spectrum} of $X$ (with coefficients in $\mathbb{F})$. Note that we only consider nonzero elements of $\Hom_k(\vr_r(X);\mathbb{F})$ in the definition $\mathrm{Spec}_k(X,\mathbb{F})$ to avoid trivial cases (there can be no positive length bars associated to a zero element).
	
	\begin{example}
		Consider $X=\{0,1\}$ equipped with the metric inherited from $\R$. Then, for any field $\mathbb{F}$,
		
		$$\mathrm{Spec}_0(X,\mathbb{F})=\bigg(\bigcup_{r\in (0,1]}\mathrm{Span}_{\mathbb{F}}(\{\mu_r,\nu_r\})\times\{r\}\bigg)\cup\bigg(\bigcup_{r>1}\mathrm{Span}_{\mathbb{F}}(\{\omega_r\})\times\{r\}\bigg)$$
		where $\mu_r$ and $\nu_r\in\Hom_0(\vr_r(X);\mathbb{F})\backslash\{0\}$ are the homology classes homologous to $0$ and $1$, respectively for $r\in (0,1]$, and $\omega_r\in\Hom_0(\vr_r(X);\mathbb{F})\backslash\{0\}$ is the homology class homologous to both $0$ and $1$ for $r>1$ (i.e., $\omega_r=(i_{r',r})_\ast(\mu_{r'})=(i_{r',r})_\ast(\nu_{r'})$ for any $r'\in (0,1]$ and $r>1$).
	\end{example}
	
	\begin{definition}[Pre-localized spread of a homology class]
		For each
		$(\omega,s)\in\mathrm{Spec}_k(X,\mathbb{F})$
		we define the \emph{pre-localized spread} of $(\omega,s)$ as follows:
		$$\mathrm{pspread}(X;\omega,s):=\inf_{B\in S(\omega,s)}\spread(B).$$
		where $S(\omega,s)$ denotes the collection of all $B\subseteq X$ such that $\omega=\iota_\ast([c])$, $c$ is a simplicial $k$-cycle on $\vr_s(B)$, and $\iota:B\longhookrightarrow X$ is the canonical inclusion.
	\end{definition}
	
	Any $B$ as in the  definition above will be said to \emph{support} the homology class $(\omega,s)\in\mathrm{Spec}_k(X,\mathbb{F})$.
	
	\begin{lemma}\label{lemma:locspread}
		Suppose $(\omega,s)\in\mathrm{Spec}_k(X,\mathbb{F})$ and $k\geq 1$ are given. Then, for any $\delta>\mathrm{pspread}(X;\omega,s)$,
		$$(i_{s,(s+\delta)})_\ast(\omega)=0$$
		where $i_{s,(s+\delta)}:\vr_s(X)\longhookrightarrow\vr_{s+\delta}(X)$ is the canonical inclusion.
	\end{lemma}
	\begin{proof}
		By the definition of $\mathrm{pspread}(X;\omega,s)$, there exists $B\subseteq X$ such that $\omega=\iota_\ast([c])$ where $c$ is a simplicial $k$-cycle on $\vr_s(B)$ and $\spread(B)<\delta$. Then, as in the proof of Proposition \ref{prop:spread}, one can prove that $$(j_{s,(s+\delta)})_\ast:\Hom_k(\vr_s(B);\mathbb{F})\longrightarrow\Hom_k(\vr_{s+\delta}(B);\mathbb{F})$$
		is the zero map where $j_{s,(s+\delta)}:\vr_s(B)\longhookrightarrow\vr_{s+\delta}(B)$ is the canonical inclusion. Hence, $(j_{s,(s+\delta)})_\ast([c])=0$. Furthermore, note that the following diagram commutes:
		$$\begin{tikzcd}\Hom_k(\vr_s(B);\mathbb{F})\arrow[d, "\iota_\ast"]\arrow[r, "(j_{s,(s+\delta)})_\ast"] & \Hom_k(\vr_{s+\delta}(B);\mathbb{F})\arrow[d, "\iota_\ast"]\\ \Hom_k(\vr_s(X);\mathbb{F})\arrow[r, "(i_{s,(s+\delta)})_\ast"] & \Hom_k(\vr_{s+\delta}(X);\mathbb{F})\end{tikzcd}$$
		where all the arrows are  maps induced by  canonical inclusions. Hence, one can conclude $(i_{s,(s+\delta)})_\ast(\omega)=0$ as we required.
	\end{proof}

	Now, fix an arbitrary $(\omega,s)\in\mathrm{Spec}_k(X,\mathbb{F})$. Then, let
	\begin{align*}
		&u_{(\omega,s)}:=\inf\{r>0:r\leq s\text{ and }\exists\text{ nonzero }\omega_r\in\Hom_k(\vr_r(X);\mathbb{F})\text{ such that }(i_{r,s})_\ast(\omega_r)=\omega\},\\
		&v_{(\omega,s)}:=\sup\{t>0:t\geq s\text{ and }\exists\text{ nonzero }\omega_t\in\Hom_k(\vr_t(X);\mathbb{F})\text{ such that }(i_{s,t})_\ast(\omega)=\omega_t\}.
	\end{align*}
	With an argument   similar to the one used in Section \ref{sec:vrendpts}, one can prove that $u_{(\omega,s)}<s\leq v_{(\omega,s)}$. Let 
	$$I_{(\omega,s)}:=\begin{cases}
		(u_{(\omega,s)},v_{(\omega,s)}] & \text{if $v_{(\omega,s)}<\infty$}\\
		(u_{(\omega,s)},\infty) & \text{otherwise}
	\end{cases}$$
	Intuitively, the interval $I_{(\omega,s)}$ is the maximal (left open, right closed) interval containing $s$ inside which the class $\omega$ can be ``propagated".

	\begin{figure}
		\centering
		\includegraphics[width=0.8\linewidth]{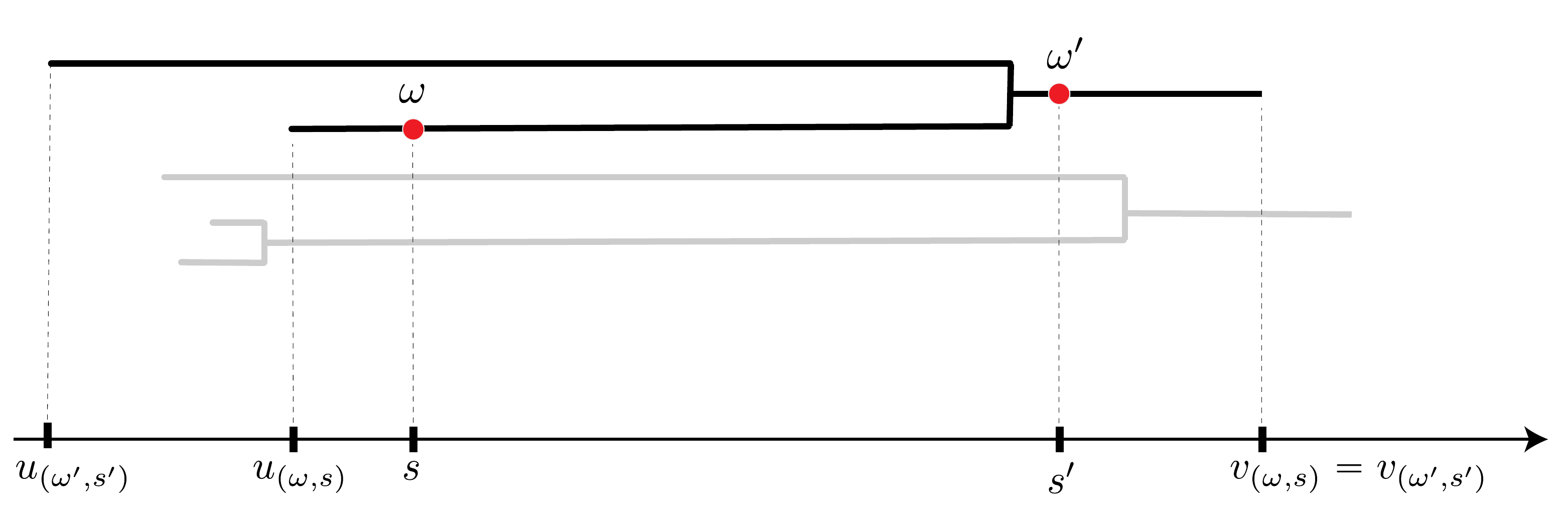}
		\caption{In this illustration, $s<s'$ and $u_{(\omega,s)}>u_{(\omega',s')}.$ See Remark \ref{rem:spec-uv}.}
		\label{fig:spec-uv}
	\end{figure}
	
	\begin{Remark}\label{rem:spec-uv}
		If $(\omega,s),(\omega',s')\in\mathrm{Spec}_k(X,\mathbb{F})$, $s\leq s'$, and $\omega'=(i_{s,s'})_\ast(\omega)$, then $v_{(\omega,s)}=v_{(\omega',s')}$ and $u_{(\omega,s)}\geq u_{(\omega',s')}$. The latter inequality can be strict. However, if $(i_{s,s'})_\ast$ is injective, then $u_{(\omega,s)}=u_{(\omega',s')}$ so that $I_{(\omega,s)}=I_{(\omega',s')}$. See Figure \ref{fig:spec-uv}.
	\end{Remark}
	
	\begin{figure}
		\centering
		\includegraphics[width=0.5\linewidth]{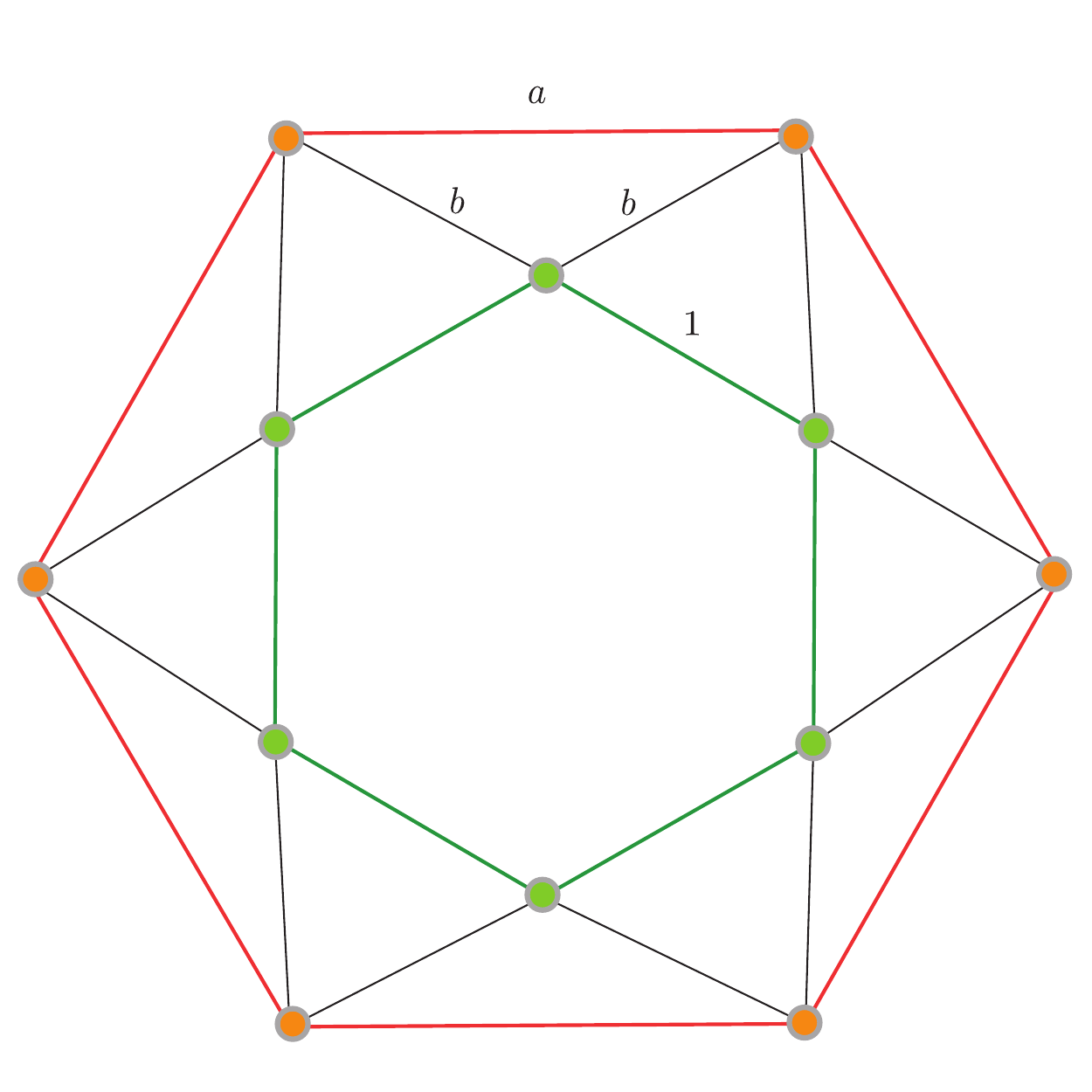}
		\caption{In this example,  $X$ is the set of  vertices of the graph shown above (in green and orange) where $1<a<b <2$. Then, $\dgmR_1(X;\mathbb{F})=\{(1,2],(a,b]\}$, while one can choose $(\omega,s)\in\mathrm{Spec}_1(X,\mathbb{F})$ such that $I_{(\omega,s)}=(a,2]\notin\dgmR_1(X;\mathbb{F})$. See Example \ref{ex:Iwsnotbarcode}.}
		\label{fig:Iomegas}
	\end{figure}
	
	\begin{example}\label{ex:Iwsnotbarcode} In general, $I_{(\omega,s)}$ is not necessarily one of the intervals in $\dgmR_k(X;\mathbb{F})$. Here is a brief sketch of how to construct such an example. Consider the graph consisting of $12$ vertices and $24$ edges as shown in Figure \ref{fig:Iomegas}. Assume that the length of the edge between adjacent inner (green) vertices is $1$, the length of the edge between adjacent outer (orange) vertices is $a$, and the length of the edge between adjacent inner and outer vertices is $b$ where $1<a<b<2$. Now, let $X$ be the set of vertices of this graph, and let $d_X$ be the shortest path metric between them. Then, one can easily check that $\dgmR_1(X;\mathbb{F})=\{(1,2],(a,b]\}$ where $(1,2]$ is associated to the homology class induced by the inner cycle and $(a,b]$ is associated to the homology class induced by the outer cycle. Now, if we choose $\{(\omega_s,s)\}_{s\in (a,b]}\subset\mathrm{Spec}_1(X,\mathbb{F})$ corresponding to the interval $(a,b]\in\dgmR_1(X;\mathbb{F})$, then $I_{(\omega_s,s)}=(a,2]\notin\dgmR_1(X;\mathbb{F})$ for $s\in (a,b]$.
	\end{example}
	
	Motivated by Example \ref{ex:Iwsnotbarcode} above, in  the proposition below we will clarify the relationship between the persistence barcode and the multiset consisting of all $I_{(\omega,s)}$.

	Recall that, for all $0<r<s<\infty$, $(i_{r,s})_*:\Hom_k(\vr_r(X);\mathbb{F})\rightarrow\Hom_k(\vr_s(X);\mathbb{F})$ denotes the morphism induced by the natural inclusion.
	
	\begin{proposition}\label{prop:bcdmtpct}
		Let $X$ be a compact metric space. Then, for all $0 < r < s$, the multiplicity of the interval $(r,s]$ in the barcode $\dgmR_k(X;\mathbb{F})$ is equal to
		
		\begin{equation}\label{eq:informal}
			\max\left\{ m\in\Z_{\geq 0} \Bigg| \begin{array}{l}
				\exists\text{ linearly independent vectors }\omega_1,\dots,\omega_m\in \Hom_k(\vr_s(X);\mathbb{F})\text{ s.t. }
				I_{(\omega_i,s)}=(r,s] \,   \forall i\\\text{ and no nonzero linear combination of  these vectors belongs to }\mathrm{Im}((i_{r,s})_*)
			\end{array}\right\}.
		\end{equation}

		More precisely, the multiplicity of $(r,s]$ in the barcode $\dgmR_k(X;\mathbb{F})$ is equal to
		\begin{equation}\label{eq:formal}
			\dim\mathrm{Span}_{\mathbb{F}}\left(\left\{[\omega]\in\left(\bigcap_{r'\in(r,s]}\mathrm{Im}((i_{r',s})_*)\right)\slash\mathrm{Im}((i_{r,s})_*)\,\,\text{s.t.}\, \,I_{(\omega,s)}=(r,s]\right\}\right).
		\end{equation}
	\end{proposition}

	Notice that $V_{r,s}:=\bigcap_{r'>r}\mathrm{Im}((i_{r',s})_*)$ contains all $\omega\in \Hom_k(\vr_s(X);\mathbb{F})$ such that $u_{(\omega,s)}\leq r$ and $v_{(\omega,s)}\geq s$. The quotient of the vector space  $V_{r,s}$ relative to $\mathrm{Im}((i_{r,s})_*)$ therefore consists of all $[\omega]$, with $\omega \in V_{r,s}$, such that $u_{(\omega,s)}=r$ and $v_{(\omega,s)}\geq s$. The condition that $I_{(\omega,s)}$ be exactly equal to $(r,s]$ further restricts $\omega$ to satisfy $v_{(\omega,s)}=s$.
	
	In terms of Example \ref{ex:Iwsnotbarcode}, the interpretation above can be used to prove that $(a,2]$ is not an interval in $\dgmR_1(X;\mathbb{F})$. Indeed, let $\omega \in \Hom_1(\vr_2(X);\mathbb{F})$ be such that $I_{(\omega,2)}=(a,2]$. Then, notice that $\omega=[c_0]$, where $c_0$ is the $1$-cycle consisting all $6$ inner points (in green). Then $u_{(\omega,2)}=u_{([c_0],2)}=1 < a$ which is a contradiction. 
	
	\begin{proof}[Proof of Proposition \ref{prop:bcdmtpct}]
		We will first prove the claim surrounding Equation (\ref{eq:formal}). By \cite[Lemma 3.16 and Proposition 3.29]{chazal2016structure},  the multiplicity of $(r,s]$ in $\dgmR_k(X;\mathbb{F})$ is equal to the following:\footnote{Whereas \cite[Lemma 3.16]{chazal2016structure} is concerned with the notion of ``persistence diagrams", it is well-known that, for a persistence module  admitting a persistence barcode, the  notions of persistence barcode and persistence diagram coincide, as is established in \cite[Proposition 3.29]{chazal2016structure}.}
		
		\begin{equation*}\label{eq:bcdmtpct}
			\lim\limits_{r'\rightarrow r+}\lim\limits_{s'\rightarrow s+}\big(\mathrm{rank}((i_{r',s})_*)-\mathrm{rank}((i_{r,s})_*)-\mathrm{rank}((i_{r',s'})_*)+\mathrm{rank}((i_{r,s'})_*)\big).
		\end{equation*}
		
		Hence, we will compute the formula above. First of all, it is easy to verify that
		\begin{align*}
			\mathrm{rank}((i_{r,s})_*)&=\dim\mathrm{Im}((i_{r,s})_*)\\
			&=\dim\mathrm{Span}_{\mathbb{F}}\left(\{\omega\in\Hom_k(\vr_s(X);\mathbb{F})\backslash\{0\}: u_{(\omega,s)}<r\text{ and }v_{(\omega,s)}\geq s\}\right).
		\end{align*}
		
		Also, note that
		\begin{align*}
			\lim\limits_{r'\rightarrow r+}\mathrm{rank}((i_{r',s})_*)&=\lim\limits_{r'\rightarrow r+}\dim\mathrm{Im}((i_{r',s})_*)\\
			&=\dim\bigcap_{r'\in(r,s]}\mathrm{Im}((i_{r',s})_*)\\
			&=\dim\mathrm{Span}_{\mathbb{F}}\left(\{\omega\in\Hom_k(\vr_s(X);\mathbb{F})\backslash\{0\}: u_{(\omega,s)}\leq r\text{ and }v_{(\omega,s)}\geq s\}\right).
		\end{align*}
		
		Hence,
		
		\begin{align*}
			&\lim\limits_{r'\rightarrow r+}\mathrm{rank}((i_{r',s})_*)-\mathrm{rank}((i_{r,s})_*)\\
			&\quad=\dim\left(\bigcap_{r'\in(r,s]}\mathrm{Im}((i_{r',s})_*)\right)\slash\mathrm{Im}((i_{r,s})_*)\\
			&\quad=\dim\mathrm{Span}_{\mathbb{F}}\left(\left\{[\omega]\in\left(\bigcap_{r'\in(r,s]}\mathrm{Im}((i_{r',s})_*)\right)\slash\mathrm{Im}((i_{r,s})_*): u_{(\omega,s)}= r\text{ and }v_{(\omega,s)}\geq s\right\}\right).
		\end{align*}
		
		Similarly, for all $s'>s$, we have
		
		\begin{align*}
			&\lim\limits_{r'\rightarrow r+}\mathrm{rank}((i_{r',s'})_*)-\mathrm{rank}((i_{r,s'})_*)\\
			&\quad=\dim\left(\bigcap_{r'\in(r,s']}\mathrm{Im}((i_{r',s'})_*)\right)\slash\mathrm{Im}((i_{r,s'})_*)\\
			&\quad=\dim\mathrm{Span}_{\mathbb{F}}\left(\left\{[\omega]\in\left(\bigcap_{r'\in(r,s']}\mathrm{Im}((i_{r',s'})_*)\right)\slash\mathrm{Im}((i_{r,s'})_*): u_{(\omega,s')}= r\text{ and }v_{(\omega,s')}\geq s'\right\}\right).
		\end{align*}
		
		Hence, finally one can derive
		
		\begin{align*}
			&\lim\limits_{r'\rightarrow r+}\lim\limits_{s'\rightarrow s+}\big(\mathrm{rank}((i_{r',s})_*)-\mathrm{rank}((i_{r,s})_*)-\mathrm{rank}((i_{r',s'})_*)+\mathrm{rank}((i_{r,s'})_*)\big)\\
			&=\dim\bigcap_{s'>s}\mathrm{Ker}(\tilde{i}_{s,s'})\\
			&=\dim\mathrm{Span}_{\mathbb{F}}\left(\left\{[\omega]\in\left(\bigcap_{r'\in(r,s]}\mathrm{Im}((i_{r',s})_*)\right)\slash\mathrm{Im}((i_{r,s})_*): I_{(\omega,s)}=(r,s]\right\}\right)
		\end{align*}
		
		where $\tilde{i}_{s,s'}:\left(\bigcap_{r'\in(r,s]}\mathrm{Im}((i_{r',s})_*)\right)\slash\mathrm{Im}((i_{r,s})_*)\rightarrow\left(\bigcap_{r'\in(r,s']}\mathrm{Im}((i_{r',s'})_*)\right)\slash\mathrm{Im}((i_{r,s'})_*)$ is the canonical morphism induced from $(i_{s,s'})_*$. This completes the derivation of Equation (\ref{eq:formal}).

		Next, we will verify that Equation (\ref{eq:formal}) is equivalent to (more intuitive) Equation (\ref{eq:informal}). For notational simplicity, let
		\begin{align*}
			(\star):\,\,&\exists\text{ linearly independent vectors }\omega_1,\dots,\omega_m\in \Hom_k(\vr_s(X);\mathbb{F})\text{ s.t. }I_{(\omega_i,s)}=(r,s]\text{ for all } i\\
			&\text{ and no nonzero linear combination of these vectors  belongs to }\mathrm{Im}((i_{r,s})_*),
		\end{align*}
		
		$$ M:=\max\{m\in\Z_{\geq 0}\vert\,(\star)\}$$ 
		and $$W:=\left\{[\omega]\in\left(\bigcap_{r'\in(r,s]}\mathrm{Im}((i_{r',s})_*)\right)\slash\mathrm{Im}((i_{r,s})_*)\,\,\text{s.t.}\, \,I_{(\omega,s)}=(r,s]\right\}.$$

		We will show that $M=\dim\mathrm{Span}_{\mathbb{F}}(W)$ which will establish the equivalence between Equations (\ref{eq:informal}) and (\ref{eq:formal}). 
		
		First, choose $\{\omega_1,\dots,\omega_M\}\subset\Hom_k(\vr_s(X);\mathbb{F})$ achieving $M$. Hence, $\omega_1,\dots,\omega_M$ satisfy the conditions in $(\star)$. Then, obviously, $\{[\omega_1],\dots,[\omega_M]\}\subseteq W$.
		
		\begin{claim}
			$[\omega_1],\dots,[\omega_M]$ are linearly independent.
		\end{claim}
		\begin{proof}
			Suppose $\sum_{i=1}^Mc_i[\omega_i]=0$ for some $c_1,\dots,c_M\in\mathbb{F}$. Then, $[\sum_{i=1}^Mc_i\omega_i]=\sum_{i=1}^Mc_i[\omega_i]=0$. This implies $\sum_{i=1}^Mc_i\omega_i\in\mathrm{Im}((i_{r,s})_*)$. Then, by assumption, no nonzero linear combination of the vectors $\omega_i$ belongs to $\mathrm{Im}((i_{r,s})_*)$, we have that $\sum_{i=1}^Mc_i\omega_i=0$. Finally, since the vectors $\omega_i$ are linearly independent, one can conclude that $c_1=\cdots=c_M=0$.
		\end{proof}
		
		Therefore, $M\leq\dim\mathrm{Span}_{\mathbb{F}}(W)$. In order to complete the proof, towards a contradiction, we will assume $M<\dim\mathrm{Span}_{\mathbb{F}}(W)$. Then, there exists $\omega\in\bigcap_{r'\in(r,s]}\mathrm{Im}((i_{r',s})_*)$ such that $I_{(\omega,s)}=(r,s]$ and $[\omega]$ is not a linear combination of the (equivalence classes of) vectors $[\omega_i]$. 
		
		We claim that $\omega_1,\dots,\omega_M,\omega$ satisfy the conditions in $(\star)$. First, let's show that $\omega_1,\dots,\omega_M,\omega$ are linearly independent. Suppose $\sum_{i=1}^Mc_i\omega_i+c\omega=0$ for some $c_1,\dots,c_M,c\in\mathbb{F}$. If $c=0$, by the linear independence of $\omega_1,\dots,\omega_M$, we have $c_1=\cdots=c_M=0$. If $c\neq 0$, without loss of generality one assume $c=1$. Therefore $\omega=-\sum_{i=1}^Mc_i\omega_i$ so $[\omega]=-\sum_{i=1}^Mc_i[\omega_i]$ which contradicts  the assumption on $\omega$. This proves that indeed $\omega_1,\dots,\omega_M,\omega$ are linearly independent. 
		
		Lastly, let's show that no nonzero linear combination of $\omega_1,\dots,\omega_M,\omega$ belongs to $\mathrm{Im}((i_{r,s})_*)$. Suppose $\sum_{i=1}^Mc_i\omega_i+c\omega$ is nonzero and belongs to $\mathrm{Im}((i_{r,s})_*)$ for some $c_1,\dots,c_M,c\in\mathbb{F}$. Note that $c$ cannot be zero by the assumptions on the vectors $\omega_i$, similarly to how we argued above. Hence, without loss of generality, $c=1$. Therefore $\sum_{i=1}^Mc_i\omega_i+\omega\in\mathrm{Im}((i_{r,s})_*)$ so $[\omega]=-\sum_{i=1}^Mc_i[\omega_i]$ which contradicts  the assumption on $\omega$. This proves that indeed no nonzero linear combination of $\omega_1,\dots,\omega_M,\omega$ belongs to $\mathrm{Im}((i_{r,s})_*)$. Hence, $\omega_1,\dots,\omega_M,\omega$ satisfy the conditions in $(\star)$. This contradicts the maximality of $M$. So, we can conclude that $M=\dim\mathrm{Span}_{\mathbb{F}}(W)$.
	\end{proof}
	
	\begin{Remark}\label{rmk:nononzerolincmb}
		\begin{figure}
			\centering
			\includegraphics[width=\linewidth]{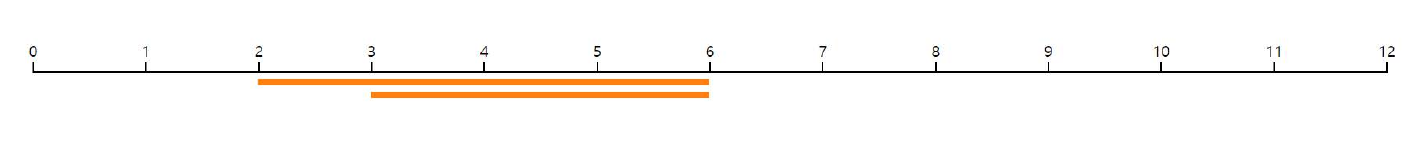}
			\caption{The barcode of the persistence module  $V_*^{2,3,6}$ from Remark \ref{rmk:nononzerolincmb}.}\label{fig:rips-torus-pd}
		\end{figure}
		
		One might wonder why the condition ``no nonzero linear combination of the vectors $\omega_1,\ldots,\omega_m$ belongs to $\mathrm{Im}((i_{r,s})_*)$"  in Equation (\ref{eq:informal}) of Proposition \ref{prop:bcdmtpct} is necessary. However, without this condition the first expression fails to provide the precise multiplicity of an interval. In this remark, we will provide a family of examples which depend on the choice of three real numbers $0<t<r<s$. Suppose $0<t<r<s$ are fixed. Let $V_*^{t,r,s}$ be the persistence module over the real line such that
		$$V_\lambda^{t,r,s}:=\begin{cases}\mathbb{F}&\text{if }t<\lambda\leq r\\ \mathbb{F}^2&\text{if }r<\lambda\leq s\\ 0&\text{otherwise}\end{cases}$$
		and
		$$v_{\lambda,\lambda'}^{t,r,s}:=\begin{cases}\mathrm{id}_{\mathbb{F}}&\text{if } t<\lambda\leq\lambda'\leq r\\ \omega\mapsto (\omega,\omega)&\text{if }\lambda\in (t,r]\text{ and }\lambda'\in (r,s]\\ \mathrm{id}_{\mathbb{F}^2}&\text{if }r<\lambda\leq\lambda'\leq s\\ 0&\text{otherwise}\end{cases}$$
		By \cite[Proposition 3.29]{chazal2016structure}, it is known that the the multiplicity of the interval $(r,s]$ is the following:
		\begin{align*}
			&\lim\limits_{r'\rightarrow r+}\lim\limits_{s'\rightarrow s+}\big(\mathrm{rank}(v_{r',s}^{t,r,s})-\mathrm{rank}(v_{r,s}^{t,r,s})-\mathrm{rank}(v_{r',s'}^{t,r,s})+\mathrm{rank}(v_{r,s'}^{t,r,s})\big)\\
			&=2-1-0+0=1.
		\end{align*}
		Now, consider the following number
		\[
		\max\left\{ m\in\Z_{\geq 0} \Bigg| \begin{array}{l}
			\exists\text{ linearly independent vectors }\omega_1,\dots,\omega_m\in V_t\text{ s.t. }
			I_{(\omega_i,t)}=(r,s]\text{ for all } i\\\text{ and no nonzero linear combination of these vectors belongs to }\mathrm{Im}(v_{r,s})
		\end{array}\right\}.
		\]
		which an analogue of Equation (\ref{eq:informal})  in Proposition \ref{prop:bcdmtpct} for the interval $(r,s]$. It is easy to verify this number is equal to the correct multiplicity $1$, while it becomes $2$ if we drop the condition ``no nonzero linear combination of these vectors belong to $\mathrm{Im}(v_{r,s})$". Figure \ref{fig:rips-torus-pd} depicts the barcode of $V_*^{2,3,6}$.

		Furthermore, the  persistence module  $V_*^{t,r,s}$ can be realized as a Vietoris-Rips persistent homology as described next.
		
		\begin{figure}
			\centering
			\begin{tikzpicture}[scale=0.5]
				\node[draw,circle,minimum size=16,inner sep=0] (0) at (0,0) {$v_0$};
				\node[draw,circle,minimum size=16,inner sep=0] (1) at (3,0) {$v_1$};
				\node[draw,circle,minimum size=16,inner sep=0] (2) at (6,0) {$v_2$};
				\node[draw,circle,minimum size=16,inner sep=0] (3) at (9,0) {$v_3$};
				\node[draw,circle,minimum size=16,inner sep=0] (4) at (12,0) {$v_4$};
				\node[draw,circle,minimum size=16,inner sep=0] (5) at (15,0) {$v_5$};
				\node[draw,circle,minimum size=16,inner sep=0] (6) at (18,0) {$v_0$};
				\node[draw,circle,minimum size=16,inner sep=0] (7) at (0,3) {$v_{10}$};
				\node[draw,circle,minimum size=16,inner sep=0] (9) at (6,3) {$v_{25}$};
				\node[draw,circle,minimum size=16,inner sep=0] (11) at (12,3) {$v_{26}$};
				\node[draw,circle,minimum size=16,inner sep=0] (13) at (18,3) {$v_{10}$};
				\node[draw,circle,minimum size=16,inner sep=0] (14) at (0,6) {$v_{9}$};
				\node[draw,circle,minimum size=16,inner sep=0] (15) at (3,6) {$v_{20}$};
				\node[draw,circle,minimum size=16,inner sep=0] (16) at (6,6) {$v_{21}$};
				\node[draw,circle,minimum size=16,inner sep=0] (17) at (9,6) {$v_{22}$};
				\node[draw,circle,minimum size=16,inner sep=0] (18) at (12,6) {$v_{23}$};
				\node[draw,circle,minimum size=16,inner sep=0] (19) at (15,6) {$v_{24}$};
				\node[draw,circle,minimum size=16,inner sep=0] (20) at (18,6) {$v_{9}$};
				\node[draw,circle,minimum size=16,inner sep=0] (21) at (0,9) {$v_{8}$};
				\node[draw,circle,minimum size=16,inner sep=0] (23) at (6,9) {$v_{18}$};
				\node[draw,circle,minimum size=16,inner sep=0] (25) at (12,9) {$v_{19}$};
				\node[draw,circle,minimum size=16,inner sep=0] (27) at (18,9) {$v_{8}$};
				\node[draw,circle,minimum size=16,inner sep=0] (28) at (0,12) {$v_{7}$};
				\node[draw,circle,minimum size=16,inner sep=0] (29) at (3,12) {$v_{13}$};
				\node[draw,circle,minimum size=16,inner sep=0] (30) at (6,12) {$v_{14}$};
				\node[draw,circle,minimum size=16,inner sep=0] (31) at (9,12) {$v_{15}$};
				\node[draw,circle,minimum size=16,inner sep=0] (32) at (12,12) {$v_{16}$};
				\node[draw,circle,minimum size=16,inner sep=0] (33) at (15,12) {$v_{17}$};
				\node[draw,circle,minimum size=16,inner sep=0] (34) at (18,12) {$v_{7}$};
				\node[draw,circle,minimum size=16,inner sep=0] (35) at (0,15) {$v_{6}$};
				\node[draw,circle,minimum size=16,inner sep=0] (37) at (6,15) {$v_{11}$};
				\node[draw,circle,minimum size=16,inner sep=0] (39) at (12,15) {$v_{12}$};
				\node[draw,circle,minimum size=16,inner sep=0] (41) at (18,15) {$v_{6}$};
				\node[draw,circle,minimum size=16,inner sep=0] (42) at (0,18) {$v_{0}$};
				\node[draw,circle,minimum size=16,inner sep=0] (43) at (3,18) {$v_{1}$};
				\node[draw,circle,minimum size=16,inner sep=0] (44) at (6,18) {$v_{2}$};
				\node[draw,circle,minimum size=16,inner sep=0] (45) at (9,18) {$v_{3}$};
				\node[draw,circle,minimum size=16,inner sep=0] (46) at (12,18) {$v_{4}$};
				\node[draw,circle,minimum size=16,inner sep=0] (47) at (15,18) {$v_{5}$};
				\node[draw,circle,minimum size=16,inner sep=0] (48) at (18,18) {$v_{0}$};
				\node[draw,circle,minimum size=16,inner sep=0] (49) at (2,16) {$v_{27}$};
				\node[draw,circle,minimum size=16,inner sep=0] (50) at (4,14) {$v_{28}$};
				\node[draw,circle,minimum size=16,inner sep=0] (51) at (9,9) {$v_{29}$};
				\node[draw,circle,minimum size=16,inner sep=0] (52) at (14,4) {$v_{30}$};
				\node[draw,circle,minimum size=16,inner sep=0] (53) at (16,2) {$v_{31}$};
				
				\draw (0) -- (1);
				\draw (1) -- (2);
				\draw (2) -- (3);
				\draw (3) -- (4);
				\draw (4) -- (5);
				\draw (5) -- (6);
				\draw (14) -- (15);
				\draw (15) -- (16);
				\draw (16) -- (17);
				\draw (17) -- (18);
				\draw (18) -- (19);
				\draw (19) -- (20);
				\draw (0) -- (7);
				\draw (7) -- (14);
				\draw (14) -- (21);
				\draw (2) -- (9);
				\draw (9) -- (16);
				\draw (16) -- (23);
				\draw (4) -- (11);
				\draw (11) -- (18);
				\draw (18) -- (25);
				\draw (6) -- (13);
				\draw (13) -- (20);
				\draw (20) -- (27);
				\draw (21) -- (28);
				\draw (23) -- (30);
				\draw (25) -- (32);
				\draw (27) -- (34);
				\draw (28) -- (29);
				\draw (29) -- (30);
				\draw (30) -- (31);
				\draw (31) -- (32);
				\draw (32) -- (33);
				\draw (33) -- (34);
				\draw (13) -- (20);
				\draw (28) -- (35);
				\draw (35) -- (42);
				\draw (30) -- (37);
				\draw (37) -- (44);
				\draw (32) -- (39);
				\draw (39) -- (46);
				\draw (34) -- (41);
				\draw (41) -- (48);
				\draw (42) -- (43);
				\draw (43) -- (44);
				\draw (44) -- (45);
				\draw (45) -- (46);
				\draw (46) -- (47);
				\draw (47) -- (48);
				\draw (49) -- (43);
				\draw (49) -- (44);
				\draw (49) -- (35);
				\draw (49) -- (28);
				\draw (50) -- (44);
				\draw (50) -- (37);
				\draw (50) -- (28);
				\draw (50) -- (29);
				\draw (51) -- (31);
				\draw (51) -- (32);
				\draw (51) -- (25);
				\draw (51) -- (23);
				\draw (51) -- (16);
				\draw (51) -- (17);
				\draw (52) -- (19);
				\draw (52) -- (20);
				\draw (52) -- (11);
				\draw (52) -- (4);
				\draw (53) -- (20);
				\draw (53) -- (13);
				\draw (53) -- (4);
				\draw (53) -- (5);
				\draw (45) -- (37);
				\draw (45) -- (39);
				\draw (31) -- (37);
				\draw (31) -- (39);
				\draw (45) -- (31);
				\draw (47) -- (39);
				\draw (47) -- (41);
				\draw (33) -- (39);
				\draw (33) -- (41);
				\draw (47) -- (33);
				\draw (33) -- (25);
				\draw (33) -- (27);
				\draw (19) -- (25);
				\draw (19) -- (27);
				\draw (33) -- (19);
				\draw (29) -- (21);
				\draw (29) -- (23);
				\draw (15) -- (21);
				\draw (15) -- (23);
				\draw (29) -- (15);
				\draw (17) -- (9);
				\draw (17) -- (11);
				\draw (3) -- (9);
				\draw (3) -- (11);
				\draw (17) -- (3);
				\draw (15) -- (7);
				\draw (15) -- (9);
				\draw (1) -- (7);
				\draw (1) -- (9);
				\draw (15) -- (1);
				\draw[blue, thick] (42) -- (49);
				\draw[blue, thick] (49) -- (50);
				\draw[blue, thick] (50) -- (30);
				\draw[blue, thick] (30) -- (51);
				\draw[blue, thick] (51) -- (18);
				\draw[blue, thick] (18) -- (52);
				\draw[blue, thick] (52) -- (53);
				\draw[blue, thick] (53) -- (6);
			\end{tikzpicture}
			\caption{The triangulation $K$ of the torus used in Remark \ref{rmk:nononzerolincmb}. Black edges have weight $3$ while blue edges have a smaller weight $2$.}
			\label{fig:torustriangulation}
		\end{figure}
		
		We are now going to impose further restrictions on $t,r$ and $s$: we will assume that $t=2$, $r=3$, and $s=6$. Let $K$ be the triangulation of the torus given in Figure \ref{fig:torustriangulation}. We assign a weight to each $1$-simplex(=edge) $e$ of $K$ in the following way:
		$$\text{the weight of }e = \begin{cases}2&\text{if }e =\{v_0,v_{27}\},\{v_{27},v_{28}\},\{v_{28},v_{14}\},\{v_{14},v_{29}\},\\
			&\qquad\{v_{29},v_{23}\},\{v_{23},v_{30}\},\{v_{30},v_{31}\},\text{ or }\{v_{31},v_{0}\}\\ 3&\text{otherwise}\end{cases}$$
		
		Now let $(X,d_X)$ be the metric space where $X$ is the vertices of $K$ and $d_X$ is the shortest path metric (i.e. graph distance) induced from the weighted edges. Then, one can verify the following facts:
		
		\begin{enumerate}
			\item $\vr_\lambda(X)=\begin{cases}K^{(0)},\text{ the }0\text{-skeleton of }K&\text{if}\,\,\,0<\lambda\leq 2\\ \text{the cycle graph consisting of the edges } \{v_0,v_{27}\},\{v_{27},v_{28}\},\\ \{v_{28},v_{14}\},\{v_{14},v_{29}\},\{v_{29},v_{23}\},\{v_{23},v_{30}\},\{v_{30},v_{31}\},\text{ and }\{v_{31},v_{0}\}&\text{if}\,\,\,2<\lambda\leq 3\\ K\text{, the triangulation of the torus shown in Figure \ref{fig:torustriangulation}}&\text{if}\,\,\,3<\lambda\leq 6\end{cases}$

			Therefore,

			$\Hom_1(\vr_\lambda(X);\mathbb{F})=\begin{cases}0&\text{if}\,\,\,0<\lambda\leq 2\\ \mathbb{F}&\text{if}\,\,\,2<\lambda\leq 3\\ \mathbb{F}^2&\text{if}\,\,\,3<\lambda\leq 6\end{cases}$
			
			\item $\Hom_1(\vr_\lambda(X);\mathbb{F})=0$ if $\lambda>6$.
			
			\item If $\lambda\in (2,3]$ and $\lambda'\in (3,6]$, we have
			$$(i_{\lambda,\lambda'})_*(\omega_0)=\omega_1+\omega_2$$
			where
			\begin{align*}
				&\omega_0\text{ is the degree-1 homology class induced by }\\
				&\quad\{v_0,v_{27}\}+\{v_{27},v_{28}\}+\{v_{28},v_{14}\}+\{v_{14}v_{29}\}+\{v_{29},v_{23}\}+\{v_{23},v_{30}\}+\{v_{30},v_{31}\}+\{v_{31},v_{0}\},\\
				&\omega_1\text{ is the degree-1 homology class induced by }\\
				&\quad\{v_0,v_{1}\}+\{v_{1},v_{2}\}+\{v_{2},v_{3}\}+\{v_{3},v_{4}\}+\{v_{4},v_{5}\}+\{v_{5},v_{0}\},\\
				&\omega_2\text{ is the degree-1 homology class induced by }\\
				&\quad\{v_0,v_{6}\}+\{v_{6},v_{7}\}+\{v_{7},v_{8}\}+\{v_{8},v_{9}\}+\{v_{9},v_{10}\}+\{v_{10},v_{0}\}.
			\end{align*}
			Note that $\omega_0$ is a generator of $\Hom_1(\vr_\lambda(X);\mathbb{F})=\mathbb{F}$ and $\omega_1$, $\omega_2$ are linearly independent generators of $\Hom_1(\vr_{\lambda'}(X);\mathbb{F})=\mathbb{F}^2$.
		\end{enumerate}
		
		Hence, by the above three facts, the Vietoris-Rips persistent homology $\PH_1(\vr_*(X);\mathbb{F})$ is equal to $V_*^{2,3,6}$.

		For completeness, below we display the distance matrix for this example. This distance matrix was used in Ripser live \url{https://live.ripser.org/} in order to obtain Figure \ref{fig:rips-torus-pd}. Note that the order of rows and columns corresponds to $(v_0,v_1,v_2,\ldots,v_{31})$: 
		$$d_X=\left[\begin{smallmatrix}
			0 & 3 & 5 & 8 & 5 & 3 & 3 & 5 & 8 & 5 & 3 & 7 & 6 & 7 & 6 & 9 & 9 & 6 & 9 & 9 & 6 & 9 & 9 & 6 & 7 & 6 & 7 & 2 & 4 & 8 & 4 & 2 \cr
			3 & 0 & 3 & 6 & 8 & 6 & 6 & 6 & 6 & 6 & 3 & 6 & 9 & 6 & 7 & 9 & 12 & 9 & 6 & 12 & 3 & 6 & 6 & 9 & 9 & 3 & 9 & 3 & 5 & 9 & 7 & 5 \cr
			5 & 3 & 0 & 3 & 6 & 8 & 6 & 6 & 9 & 9 & 6 & 3 & 6 & 6 & 5 & 6 & 9 & 9 & 8 & 10 & 6 & 6 & 6 & 9 & 12 & 3 & 6 & 3 & 3 & 7 & 9 & 7 \cr
			8 & 6 & 3 & 0 & 3 & 6 & 9 & 9 & 9 & 9 & 9 & 3 & 3 & 9 & 6 & 3 & 6 & 6 & 9 & 9 & 6 & 6 & 6 & 6 & 9 & 3 & 3 & 6 & 6 & 6 & 6 & 6 \cr
			5 & 8 & 6 & 3 & 0 & 3 & 6 & 9 & 9 & 6 & 6 & 6 & 3 & 12 & 9 & 6 & 6 & 6 & 10 & 8 & 9 & 9 & 6 & 5 & 6 & 6 & 3 & 7 & 9 & 7 & 3 & 3 \cr
			3 & 6 & 8 & 6 & 3 & 0 & 3 & 6 & 6 & 6 & 6 & 9 & 3 & 9 & 9 & 6 & 6 & 3 & 12 & 6 & 9 & 12 & 9 & 7 & 6 & 9 & 6 & 5 & 7 & 9 & 5 & 3 \cr
			3 & 6 & 6 & 9 & 6 & 3 & 0 & 3 & 6 & 8 & 6 & 8 & 6 & 6 & 7 & 9 & 6 & 3 & 9 & 6 & 9 & 12 & 12 & 9 & 6 & 9 & 9 & 3 & 5 & 9 & 7 & 5 \cr
			5 & 6 & 6 & 9 & 9 & 6 & 3 & 0 & 3 & 6 & 8 & 6 & 6 & 3 & 5 & 8 & 6 & 3 & 6 & 6 & 6 & 9 & 10 & 9 & 6 & 9 & 12 & 3 & 3 & 7 & 9 & 7 \cr
			8 & 6 & 9 & 9 & 9 & 6 & 6 & 3 & 0 & 3 & 6 & 9 & 6 & 3 & 6 & 9 & 6 & 3 & 6 & 6 & 3 & 6 & 9 & 6 & 3 & 6 & 9 & 6 & 6 & 8 & 6 & 6 \cr
			5 & 6 & 9 & 9 & 6 & 6 & 8 & 6 & 3 & 0 & 3 & 12 & 9 & 6 & 9 & 10 & 9 & 6 & 6 & 6 & 3 & 6 & 8 & 5 & 3 & 6 & 6 & 7 & 9 & 7 & 3 & 3 \cr
			3 & 3 & 6 & 9 & 6 & 6 & 6 & 8 & 6 & 3 & 0 & 9 & 9 & 6 & 9 & 12 & 12 & 9 & 6 & 9 & 3 & 6 & 9 & 7 & 6 & 6 & 8 & 5 & 7 & 9 & 5 & 3 \cr
			7 & 6 & 3 & 3 & 6 & 9 & 8 & 6 & 9 & 12 & 9 & 0 & 6 & 6 & 3 & 3 & 6 & 9 & 6 & 8 & 9 & 8 & 8 & 7 & 10 & 6 & 6 & 5 & 3 & 5 & 9 & 9 \cr
			6 & 9 & 6 & 3 & 3 & 3 & 6 & 6 & 6 & 9 & 9 & 6 & 0 & 9 & 6 & 3 & 3 & 3 & 9 & 6 & 9 & 9 & 9 & 8 & 6 & 6 & 6 & 8 & 8 & 6 & 6 & 6 \cr
			7 & 6 & 6 & 9 & 12 & 9 & 6 & 3 & 3 & 6 & 6 & 6 & 9 & 0 & 3 & 6 & 8 & 6 & 3 & 8 & 3 & 6 & 8 & 7 & 6 & 6 & 10 & 5 & 3 & 5 & 9 & 9 \cr
			6 & 7 & 5 & 6 & 9 & 9 & 7 & 5 & 6 & 9 & 9 & 3 & 6 & 3 & 0 & 3 & 5 & 8 & 3 & 5 & 6 & 5 & 5 & 4 & 7 & 8 & 7 & 4 & 2 & 2 & 6 & 8 \cr
			9 & 9 & 6 & 3 & 6 & 6 & 9 & 8 & 9 & 10 & 12 & 3 & 3 & 6 & 3 & 0 & 3 & 6 & 6 & 6 & 9 & 6 & 6 & 5 & 8 & 6 & 6 & 7 & 5 & 3 & 7 & 9 \cr
			9 & 12 & 9 & 6 & 6 & 6 & 6 & 6 & 6 & 9 & 12 & 6 & 3 & 8 & 5 & 3 & 0 & 3 & 6 & 3 & 9 & 6 & 6 & 5 & 6 & 9 & 8 & 9 & 7 & 3 & 7 & 9 \cr
			6 & 9 & 9 & 6 & 6 & 3 & 3 & 3 & 3 & 6 & 9 & 9 & 3 & 6 & 8 & 6 & 3 & 0 & 9 & 3 & 6 & 9 & 9 & 6 & 3 & 9 & 9 & 6 & 6 & 6 & 6 & 6 \cr
			9 & 6 & 8 & 9 & 10 & 12 & 9 & 6 & 6 & 6 & 6 & 6 & 9 & 3 & 3 & 6 & 6 & 9 & 0 & 6 & 3 & 3 & 6 & 5 & 8 & 6 & 8 & 7 & 5 & 3 & 7 & 9 \cr
			9 & 12 & 10 & 9 & 8 & 6 & 6 & 6 & 6 & 6 & 9 & 8 & 6 & 8 & 5 & 6 & 3 & 3 & 6 & 0 & 9 & 6 & 6 & 3 & 3 & 9 & 6 & 9 & 7 & 3 & 5 & 7 \cr
			6 & 3 & 6 & 6 & 9 & 9 & 9 & 6 & 3 & 3 & 3 & 9 & 9 & 3 & 6 & 9 & 9 & 6 & 3 & 9 & 0 & 3 & 6 & 8 & 6 & 3 & 9 & 6 & 6 & 6 & 6 & 6 \cr
			9 & 6 & 6 & 6 & 9 & 12 & 12 & 9 & 6 & 6 & 6 & 8 & 9 & 6 & 5 & 6 & 6 & 9 & 3 & 6 & 3 & 0 & 3 & 5 & 8 & 3 & 6 & 9 & 7 & 3 & 7 & 9 \cr
			9 & 6 & 6 & 6 & 6 & 9 & 12 & 10 & 9 & 8 & 9 & 8 & 9 & 8 & 5 & 6 & 6 & 9 & 6 & 6 & 6 & 3 & 0 & 3 & 6 & 3 & 3 & 9 & 7 & 3 & 5 & 7 \cr
			6 & 9 & 9 & 6 & 5 & 7 & 9 & 9 & 6 & 5 & 7 & 7 & 8 & 7 & 4 & 5 & 5 & 6 & 5 & 3 & 8 & 5 & 3 & 0 & 3 & 6 & 3 & 8 & 6 & 2 & 2 & 4 \cr
			7 & 9 & 12 & 9 & 6 & 6 & 6 & 6 & 3 & 3 & 6 & 10 & 6 & 6 & 7 & 8 & 6 & 3 & 8 & 3 & 6 & 8 & 6 & 3 & 0 & 9 & 6 & 9 & 9 & 5 & 3 & 5 \cr
			6 & 3 & 3 & 3 & 6 & 9 & 9 & 9 & 6 & 6 & 6 & 6 & 6 & 6 & 8 & 6 & 9 & 9 & 6 & 9 & 3 & 3 & 3 & 6 & 9 & 0 & 6 & 6 & 6 & 6 & 8 & 8 \cr
			7 & 9 & 6 & 3 & 3 & 6 & 9 & 12 & 9 & 6 & 8 & 6 & 6 & 10 & 7 & 6 & 8 & 9 & 8 & 6 & 9 & 6 & 3 & 3 & 6 & 6 & 0 & 9 & 9 & 5 & 3 & 5 \cr
			2 & 3 & 3 & 6 & 7 & 5 & 3 & 3 & 6 & 7 & 5 & 5 & 8 & 5 & 4 & 7 & 9 & 6 & 7 & 9 & 6 & 9 & 9 & 8 & 9 & 6 & 9 & 0 & 2 & 6 & 6 & 4 \cr
			4 & 5 & 3 & 6 & 9 & 7 & 5 & 3 & 6 & 9 & 7 & 3 & 8 & 3 & 2 & 5 & 7 & 6 & 5 & 7 & 6 & 7 & 7 & 6 & 9 & 6 & 9 & 2 & 0 & 4 & 8 & 6 \cr
			8 & 9 & 7 & 6 & 7 & 9 & 9 & 7 & 8 & 7 & 9 & 5 & 6 & 5 & 2 & 3 & 3 & 6 & 3 & 3 & 6 & 3 & 3 & 2 & 5 & 6 & 5 & 6 & 4 & 0 & 4 & 6 \cr
			4 & 7 & 9 & 6 & 3 & 5 & 7 & 9 & 6 & 3 & 5 & 9 & 6 & 9 & 6 & 7 & 7 & 6 & 7 & 5 & 6 & 7 & 5 & 2 & 3 & 8 & 3 & 6 & 8 & 4 & 0 & 2 \cr
			2 & 5 & 7 & 6 & 3 & 3 & 5 & 7 & 6 & 3 & 3 & 9 & 6 & 9 & 8 & 9 & 9 & 6 & 9 & 7 & 6 & 9 & 7 & 4 & 5 & 8 & 5 & 4 & 6 & 6 & 2 & 0
		\end{smallmatrix}\right]$$
	\end{Remark}
	
	\begin{definition}[Localized spread of a homology class]
		For each $(\omega,s)\in\mathrm{Spec}_k(X,\mathbb{F})$, we define the \emph{localized spread} of $(\omega,s)$ as follows:
		
		$$\spread(X;\omega,s):=\sup\{\mathrm{pspread}(X;\omega',s'):s'\leq s\text{ and }\omega=(i_{s',s})_\ast(\omega') \}.$$
	\end{definition}
	
	\begin{Remark}\label{rmk:locsprduppbdd}
		It is easy to check that both $\mathrm{pspread}(X;\omega,s)$ and $\spread(X;\omega,s)$ are always upper bounded by $\spread(X)$.
	\end{Remark}
	
	The following Proposition \ref{prop:locspread} is the  ``localized" version of Proposition \ref{prop:spread} we promised in the beginning of this section.
	
	\begin{proposition}\label{prop:locspread}
		Let $(X,d_X)$ be a compact metric space and $k\geq 1$. Then, for any $(\omega,s)\in\mathrm{Spec}_k(X,\mathbb{F})$, we have
		$$\mathrm{length}(I_{(\omega,s)})\leq\spread(X;\omega,s).$$
	\end{proposition}
	\begin{proof}
		Fix arbitrary $\delta>\spread(X;\omega,s)$ and $s'\in(u_{(\omega,s)},s]$. Then, $\exists\,\omega'\in\Hom_k(\vr_{s'}(X);\mathbb{F})$ such that $\omega=(i_{s',s})_\ast(\omega')$. Hence, by Lemma \ref{lemma:locspread}, $(i_{s',(s'+\delta)})_\ast(\omega')=0$. This indicates $v_{(\omega,s)}<s'+\delta$. Since the choice of $\delta$ and $s'$ are arbitrary, one can conclude
		$$\mathrm{length}(I_{(\omega,s)})=v_{(\omega,s)}-u_{(\omega,s)}\leq\spread(X;\omega,s).$$
	\end{proof}
	
	For an arbitrary $I\in\dgmR_k(X;\mathbb{F})$, a family of nonzero homology classes $\{(\omega_s,s)\}_{s\in I}\subseteq\mathrm{Spec}_k(X,\mathbb{F})$  such that $(i_{s,s'})_\ast(\omega_s)=\omega_{s'}$ for any $s\leq s'$ in $I$ where $i_{s,s'}:\vr_s(X)\longhookrightarrow\vr_{s'}(X)$ is the canonical inclusion, will be said to \emph{correspond} to $I$, if there is an isomorphism $$\Phi_\ast:\PH_k(\vr_\ast(X);\mathbb{F})\longrightarrow\bigoplus_{I\in\dgmR_k(X;\mathbb{F})}I_\mathbb{F}$$
	such that
	$$\begin{tikzcd}\mathrm{Span}_{\mathbb{F}}(\{\omega_s\})\arrow[d, "\Phi_s"]\arrow[r, "(i_{s,s'})_\ast"] & \mathrm{Span}_{\mathbb{F}}(\{\omega_{s'}\})\arrow[d, "\Phi_{s'}"]\\ \mathbb{F}\arrow[r,"\mathrm{id}"] & \mathbb{F}\end{tikzcd}$$
	
	Observe that Theorem \ref{theorem:pbarcode} guarantees that at least one such family of nonzero homology classes $\{(\omega_s,s)\}_{s\in I}$  always exists.

	\begin{Remark}\label{rmk:IandIomega}
		Now, given an arbitrary $I\in\dgmR_k(X;\mathbb{F})$,  there is a family of nonzero homology classes $\{(\omega_s,s)\}_{s\in I}\subseteq\mathrm{Spec}_k(X,\mathbb{F})$ corresponding to $I$ as described above. Then, obviously $I\subseteq I_{(\omega_s,s)}$ for each $s\in I$. Hence,
		$$\mathrm{length}(I)\leq\inf_{s\in I}\mathrm{length}(I_{(\omega_s,s)})\leq\inf_{s\in I}\spread(X;\omega_s,s)\leq\spread(X)$$
		so that one recovers the  result in Proposition \ref{prop:spread}. Below we show some examples that highlight cases in which the localized spread is more efficient at estimating the length of bars than its global counterpart.
	\end{Remark}
	
	\begin{example}\label{ex:loc-spread}
		Here are some applications of the notion of localized spread.
		
		Let $X$ be a compact metric space. If for a given $I\in\dgmR_k(X;\mathbb{F})$ a corresponding family   $\{(\omega_s,s)\}_{s\in I}\subseteq\mathrm{Spec}_k(X,\mathbb{F})$ is supported by a subset $B\subseteq X$, then
		\begin{equation}\label{ineq:locsprd}
			\mathrm{length}(I)\leq\inf_{s\in I}\mathrm{length}(I_{(\omega_s,s)})\leq\inf_{s\in I}\spread(X;\omega_s,s)\leq\spread(B)
		\end{equation}
		
		\noindent where the first inequality holds as in Remark \ref{rmk:IandIomega}, the second inequality holds by Proposition \ref{prop:locspread}, and the last inequality follows from Remark \ref{rmk:locsprduppbdd}.
		
		\medskip
		Here are 3 scenarios in which the estimate in inequality (\ref{ineq:locsprd}) is useful.
		\begin{enumerate}
			\item Suppose a closed Riemannian manifold $M$ and a nonzero homology class $\omega\in\Hom_1(M;\mathbb{F})$ are given. Also, let $B\subseteq M$ be the shortest loop representing $\omega$. Recall that there is an interval $I\in\dgmR_1(M;\mathbb{F})$ associated to $\omega$ (cf.  Proposition \ref{prop:genfilradpersistence}). Then,
			$$\mathrm{length}(I)\leq\spread(B)=\frac{\mathrm{length}(B)}{3}$$
			by the inequality (\ref{ineq:locsprd}) and Remark \ref{rem:spread-sn}. Actually, $I=\left(0,\frac{\mathrm{length}(B)}{3}\right]$; see \cite{gasparovic2018complete} and \cite[Theorem 8.10]{Virk18}.
			
			\item Let $X$ be the metric gluing of  a loop of length $l_2$ and an interval of length $l_1$ (glued to the circle at one of its endpoints). Then, by Proposition \ref{prop:spread}, $I\leq\spread(X)$  for any $I\in\dgmR_k(X;\mathbb{F})$. However, observe that one can make $\spread(X)$ arbitrarily large by increasing $l_1$. But, if $J\in\dgmR_1(X;\mathbb{F})$ and a family of nonzero homology classes $\{(\omega_s,s)\}_{s\in J}\subseteq\mathrm{Spec}_1(X,\mathbb{F})$ corresponding to $J$ is supported by the loop, then
			$$\mathrm{length}(J)\leq\spread(B)=\frac{l^2}{3}$$
			by  inequality (\ref{ineq:locsprd}) and Remark \ref{rem:spread-sn}. Again, as in the first item, $J=\left(0,\frac{l^2}{3}\right]$. Note that the existence of the interval $\left(0,\frac{l^2}{3}\right]$ in $\dgmR_1(X;\mathbb{F})$ can also be proved via the ``crushing" technique introduced by Hausmann (see \cite[Proposition 2.2]{h95}) since $X$ can be crushed onto the loop of length $l_2$.
			
			\item An example similar to the one described in the previous item arises from Figure \ref{fig:non-unique-int}.  Consider the tube connecting the two blobs to be large: in that case the standard spread of the space will be large yet the lifetime of the individual $\mathrm{H}_2$ classes will be much smaller.
		\end{enumerate}
	\end{example}

	\subsubsection{The proof of Lemma \ref{lemma:katzcontractiongen}} \label{sec:proof-katz-lemma-gen}
	
	Let us introduce a technical tool for this subsection. It is easy to check that the usual linear interpolation in $L^\infty(X)$ gives a geodesic bicombing on $L^\infty(X)$ satisfying  all  three properties mentioned in Lemma \ref{lemma:geobicomb}. However, in \cite{k83}, Katz introduced an alternative way to construct a geodesic bicombing on $L^\infty(X)$:
	
	\begin{definition}[Katz's geodesic bicombing]\label{def:Katzgeobicomb}
		Let $X$ be a compact metric space. We define the \emph{Katz geodesic bicombing} $\gamma_K$ on $L^\infty(X)$ in the following way:
		\begin{align*}
			\gamma_K:L^\infty(X)\times L^\infty(X)\times [0,1]&\longrightarrow L^\infty(X)\\
			(f,g,t)&\longmapsto \gamma_K(f,g,t)
		\end{align*}
		where
		\begin{align*}
			\gamma_K(f,g,t):X&\longrightarrow\mathbb{R}\\
			x&\longmapsto\begin{cases}\max\{f(x)-t\Vert f-g \Vert_\infty,g(x)\}&\text{if }f(x)\geq g(x)\\\min\{f(x)+t\Vert f-g \Vert_\infty,g(x)\}&\text{if }f(x)\leq g(x)\end{cases}
		\end{align*}
		In other words, $\gamma_K(f,g,\cdot)$ moves from $f$ to $g$ with the same speed at every point.
	\end{definition}
	
	The following proposition establishes that $\gamma_K$ is indeed a (continuous) geodesic bicombing, amongst other properties. The proof is relegated to Appendix \ref{app:proof-propo-katz}.

	\begin{restatable}[]{proposition}{restatepropKatzgeoprop}\label{prop:Katzgeoprop}
		Let $X$ be a compact metric space. Then, the Katz geodesic bicombing $\gamma_K$ on $L^\infty(X)$ satisfies the following properties: for any $f,g,h\in L^\infty(X)$ and $0\leq s\leq t\leq 1$,
		\begin{enumerate}
			\item[(1)] $\gamma_K(f,g,0)=f$ and $\gamma_K(f,g,1)=g$.
			
			\item[(2)] $\Vert\gamma_K(f,g,s)-\gamma_K(f,g,t)\Vert_\infty=(t-s)\cdot \Vert f-g \Vert_\infty$.
			
			\item[(3)] $\Vert\gamma_K(f,g,t)-\gamma_K(h,g,t)\Vert_\infty\leq 2\Vert f-h\Vert_\infty$.
			
			\item[(4)] $\Vert\gamma_K(f,g,t)-\gamma_K(f,h,t)\Vert_\infty\leq \Vert g-h\Vert_\infty$.
			
			\item[(5)]  $\gamma_K(\phi,\psi,\lambda)=\gamma_K(f,g,(1-\lambda)s+\lambda t)$ where $\phi=\gamma_K(f,g,s)$ and $\psi=\gamma_K(f,g,t)$ for any $\lambda\in [0,1]$.  (This property is called \emph{consistency}).
			
			\item[(6)] $\Vert\gamma_K(f,g,r)-h\Vert_\infty\leq\max\{\Vert\gamma_K(f,g,s)-h\Vert_\infty,\Vert\gamma_K(f,g,t)-h\Vert_\infty\}$ for any $r\in[s,t]$.
		\end{enumerate}
	\end{restatable}
	
	Properties (2), (3), and (4) of Proposition \ref{prop:Katzgeoprop} imply the continuity of the Katz geodesic bicombing. In contrast, this bicombing is neither conical nor reversible; See Example \ref{sec:app:katz-bicombing} in the appendix.
	
	\begin{figure}
		\centering
		\includegraphics[width=\linewidth]{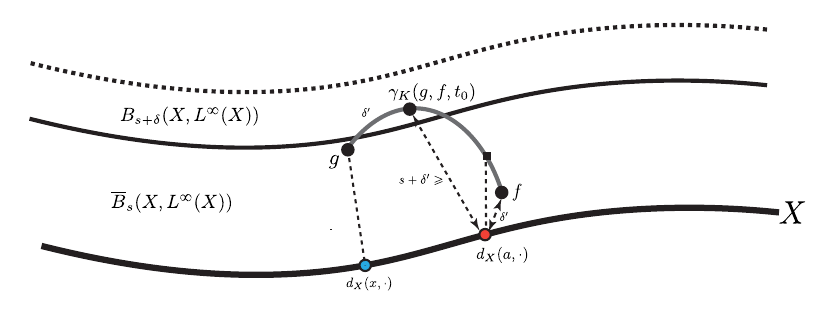}
		\caption{Strategy of the proof of Lemma \ref{lemma:katzcontractiongen}. By construction, the distance between $f$ and $d_X(a,\cdot)$ (represented by the red dot) is $\delta'$ and the distance between $\gamma_K(g,f,t_0)$ and $d_X(a,\cdot)$ is less than or equal to $s+\delta'$. Hence, by Item (6) of Proposition \ref{prop:Katzgeoprop}, we have that the  point represented by a square will be at distance at most $s+\delta'$ from $d_X(a,\cdot)$.}
		\label{fig:spread-proof}
	\end{figure}
	
	\begin{proof}[Proof of Lemma \ref{lemma:katzcontractiongen}]
		By the definition of spread, we know that there is a nonempty finite subset $A\subseteq X$ and $\delta'\in(0,\delta)$ such that $\diam(A) < 2\delta'$ and $\sup_{x\in X}\inf_{a\in A} d_X(x,a) < 2\delta'$.

		Next, we define
		\begin{align*}
			f:X&\longrightarrow\mathbb{R}\\
			x&\longmapsto d_X(x,A)+\delta'.
		\end{align*}
		
		The main strategy of the proof is depicted in Figure \ref{fig:spread-proof}.
		
		\begin{claim}\label{claim:disfa}
			$\Vert d_X(a,\cdot)-f \Vert_\infty=\delta'$ for any $a\in A$.
		\end{claim}
		\begin{proof}[Proof of Claim \ref{claim:disfa}]
			To prove this, fix arbitrary $x\in X$. Note that
			$$d_X(a,x)-f(x)=d_X(a,x)-d_X(x,A)-\delta'.$$
			Since $d_X(x,A)\leq d_X(a,x)$, we have $-\delta'\leq d_X(a,x)-d_X(x,A)-\delta'$. Also, because the diameter of $A$ is smaller than $2\delta'$, we have $d_X(a,x)-d_X(x,A)-\delta'<\delta'$. Therefore, we have $\vert d_X(a,x)-f(x) \vert\leq\delta'$. Furthermore, if we put $x=a$, we have that $\Vert d_X(a,\cdot)-f \Vert_\infty=\delta'$.
		\end{proof}
		
		Now, let $$U_{s,\delta}:=\{\gamma_K(g,f,t):g\in\overline{B}_s(X,L^\infty(X)), t\in[0,1]\}.$$ Then  $U_{s,\delta}$ obviously contains $\overline{B}_s(X,L^\infty(X))$ and can be contracted to the point $f$. 
		
		The lemma will follow once we establish the following claim.
		
		\begin{claim}\label{claim:sec}
			$U_{s,\delta}\subseteq B_{s+\delta}(X,L^\infty(X))$.
		\end{claim}
		
		\begin{proof}[Proof of Claim \ref{claim:sec}]
			To see this, fix arbitrary $g\in\overline{B}_s(X,L^\infty(X))$ and $t\in [0,1]$. Note that one can choose $x\in X$ such that $\Vert g-d_X(x,\cdot)\Vert_\infty\leq s$.
			
			\begin{itemize}
				\item If $\Vert g-f \Vert_\infty\leq\delta'$, then $$\Vert \gamma_K(g,f,t)-d_X(x,\cdot) \Vert_\infty\leq\Vert\gamma_K(g,f,t)-g\Vert_\infty+\Vert g-d_X(x,\cdot)\Vert_\infty\leq s+\delta'<s+\delta$$
				by the triangle inequality and properties (1) and (2) of Proposition \ref{prop:Katzgeoprop}. So, $\gamma_K(g,f,t)\in B_{s+\delta}(X,L^\infty(X))$.
				
				\item Now, assume $\Vert g-f \Vert_\infty>\delta'$. Let us denote $t_0:=\frac{\delta'}{\Vert g-f \Vert_\infty}$. Now, for $t\in[0,t_0]$, $\gamma_K(g,f,t)\in B_{s+\delta}(X,L^\infty(X))$ since
				$$\Vert \gamma_K(g,f,t)-d_X(x,\cdot) \Vert_\infty\leq\Vert \gamma_K(g,f,t)-g \Vert_\infty+\Vert g-d_X(x,\cdot) \Vert_\infty\leq t\Vert g-f \Vert_\infty+s\leq s+\delta'<s+\delta.$$
				Next, we want to show $\gamma_K(g,f,t)\in B_{s+\delta}(X,L^\infty(X))$ for $t\in[t_0,1]$. To do that, choose $a\in A$ such that $d_X(x,a)<2\delta'$. We will prove $\Vert \gamma_K(g,f,t_0)-d_X(a,\cdot) \Vert_\infty\leq s+\delta'$.

				Fix arbitrary $x'\in X$. If $\vert g(x')-f(x')\vert\leq\delta'$, then $\gamma_K(g,f,t_0)(x')=f(x')$. Hence,
				$$\vert\gamma_K(g,f,t_0)(x')-d_X(a,x') \vert=\vert f(x')-d_X(a,x') \vert\leq\delta'$$
				by Claim \ref{claim:disfa}. If $\vert g(x')-f(x')\vert>\delta'$, $g(x')$ cannot be between $d_X(a,x')$ and $f(x')$ since $\vert d_X(a,x')-f(x') \vert\leq\delta'$ by Claim \ref{claim:disfa}. This implies that either
				$$\vert g(x')-d_X(a,x')\vert=\vert d_X(a,x')-f(x') \vert+\vert g(x')-f(x')\vert$$
				or
				$$\vert g(x')-f(x')\vert=\vert d_X(a,x')-f(x') \vert+\vert g(x')-d_X(a,x')\vert.$$
				Either way, it is easy to see that we always have
				$$\vert \gamma_K(g,f,t_0)(x')-d_X(a,x') \vert=\vert\vert g(x')-d_X(a,x') \vert-\delta'\vert\leq s+\delta'.$$
				where the last inequality is true because $\vert g(x')-d_X(a,x') \vert\leq\vert g(x')-d_X(x,x') \vert+d_X(a,x)<s+2\delta'$.
				So,  one can conclude that
				$$\Vert \gamma_K(g,f,t_0)-d_X(a,\cdot) \Vert_\infty\leq s+\delta'.$$
				Therefore, combining this inequality with Claim \ref{claim:disfa} and property (6) of Proposition \ref{prop:Katzgeoprop}, one finally  obtains that
				$$\Vert\gamma_K(d_X(x,\cdot),f,t)-d_X(a,\cdot) \Vert_\infty\leq s+\delta'<s+\delta$$
				so that $\gamma_K(g,f,t)\in B_{s+\delta}(X,L^\infty(X))$ for any $t\in[t_0,1]$.
			\end{itemize}
		\end{proof}
		This concludes the proof.
	\end{proof}
	
	
	\subsection{The filling radius and Vietoris-Rips persistent homology}\label{sec:bound}
	
	Now, we recall the notion of \emph{filling radius}, an invariant for closed connected manifolds introduced by Gromov \cite[pg.8]{gfilling} in the course of proving the systolic inequality (see also \cite{katz-book} for a comprehensive treatment). It turns out to be that this notion can be a bridge between topological data analysis and differential geometry/topology.
	
	\begin{definition}[Filling radius]
		Let $M$ be a closed connected $n$-dimensional manifold with compatible metric $d_M$. One defines the filling radius of $M$ as follows:
		$$\mathrm{FillRad}(M;G):=\inf\{r>0|\,\mathrm{H}_n(\iota_r;G)([M])=0\},$$
		where $\iota_r:M\hookrightarrow B_r(M,L^\infty(M))$ is the (corestriction of the) Kuratowski isometric embedding, $[M]$ is the fundamental class of $M$, with coefficients $G$. We will use the shorthand notation  $\mathrm{FillRad}(M)$ when either $M$ is orientable and $G=\Z$ or when $M$ is not orientable and $G=\Z_2$.
	\end{definition}
	
	\begin{Remark}[Metric manifolds]
		Note that the definition of the filling radius does not require the metric $d_M$ on $M$ to be Riemannian -- it suffices that $d_M$ generates the manifold topology. We call any $(M,d_M)$ satisfying this condition a \textbf{metric manifold}. In particular, one can consider the filling radius of:
		\begin{enumerate}
			\item $\ell^\infty$-metric product of $(M,d_M)$ and $(N,d_N)$ when $M$ and $N$ are Riemannian manifolds and $d_M$ and $d_N$ are their corresponding geodesic distances.
			\item $(N,d_M\vert_{N\times N})$ when $N$ is a submanifold of the Riemannian manifold $(M,d_M)$.
		\end{enumerate}
	\end{Remark}
	
	\begin{Remark}[Relative filling radius and minimality for injective metric spaces.]\label{rem:relative-filling-radius}
		Note that the relative filling radius can be defined for every metric pair $(M,E)$ by considering $r$-neighborhoods of $M$ in $E$ --- it is denoted by $\filrad(M,E)$. Gromov \cite{gfilling} showed that we obtain the minimal possible relative filling radius through the Kuratowski embedding (that is when $E=L^\infty(M)$). This also follows from our work but in greater generality in the context of embeddings into injective metric spaces. If $M$ can be isometrically embedded into an injective metric space $F$, then this embedding can be extended to a $1$-Lipschitz map $f: E \to F$, which induces a map of filtrations $f_r:B_r(M,E) \to B_r(M,F)$, for each $r>0$ (see Definition \ref{def:Injme}). Hence, if the fundamental class of $M$ vanishes in $B_r(M,E)$, then it also vanishes in $B_r(M,F)$. Therefore,
		\begin{equation}\label{eq:rel-fill-rad}
			\filrad(M,F) \leq \filrad(M,E).
		\end{equation}
		If particular, this implies that $\filrad(M,E)=\filrad(M,F)$ whenever   $E$ and $F$ are both injective metric spaces admitting isometric embeddings of $M$.
	\end{Remark}
	
	\begin{Remark}[Filling radius and first change in homotopy type]\label{rmk:firstextremum}
		In \cite[Theorem 2]{k83}  Katz proved that $\mathrm{FillRad}(\mathbb{S}^n)=\frac{1}{2}\arccos{\left(\frac{-1}{n+1}\right)}$. Moreover, by Theorem \ref{theorem:isom} and Theorem \ref{cor:Snfirsttype}, $B_r(\mathbb{S}^n,L^\infty(\mathbb{S}^n))$ is homotopy equivalent to $\mathbb{S}^n$ if $r\in(0,\mathrm{FillRad}(\mathbb{S}^n)]$. Note that this immediately implies that $\filrad(\Sp^n)\geq\frac{1}{2}\arccos{\left(\frac{-1}{n+1}\right)}$.

		One might then ask whether for any closed connected manifold $M$ it holds that $\mathrm{FillRad}(M)$ is the first value of $r$ where the homotopy type of $B_r(M,L^\infty(M))$ changes. In general, however, this is not true as the following two examples show:
		\begin{enumerate}
			\item It is known \cite[Proposition 0.3]{katz-cpn} that $\mathrm{FillRad}(\mathbb{CP}^3)>\mathrm{FillRad}(\mathbb{CP}^1)=\frac{1}{2}\arccos{\left(-\frac{1}{3}\right)}$. Also, by Theorem \cite[Theorem 8.1]{katz-s1}, $B_r(\mathbb{CP}^3,L^\infty(\mathbb{CP}^3))$ is not homotopy equivalent to $\mathbb{CP}^3$ for $r\in\left(\frac{1}{2}\arccos{\left(-\frac{1}{3}\right)},\frac{1}{2}\arccos{\left(-\frac{1}{3}\right)}+\varepsilon_0\right)$ where $\varepsilon_0>0$ is a positive constant. In other words, the homotopy type of $B_r(\mathbb{CP}^3,L^\infty(\mathbb{CP}^3))$  already changed before $r=\mathrm{FillRad}(\mathbb{CP}^3)$.
			
			\item \begin{figure}
				\centering
				\includegraphics[width=0.35\textheight]{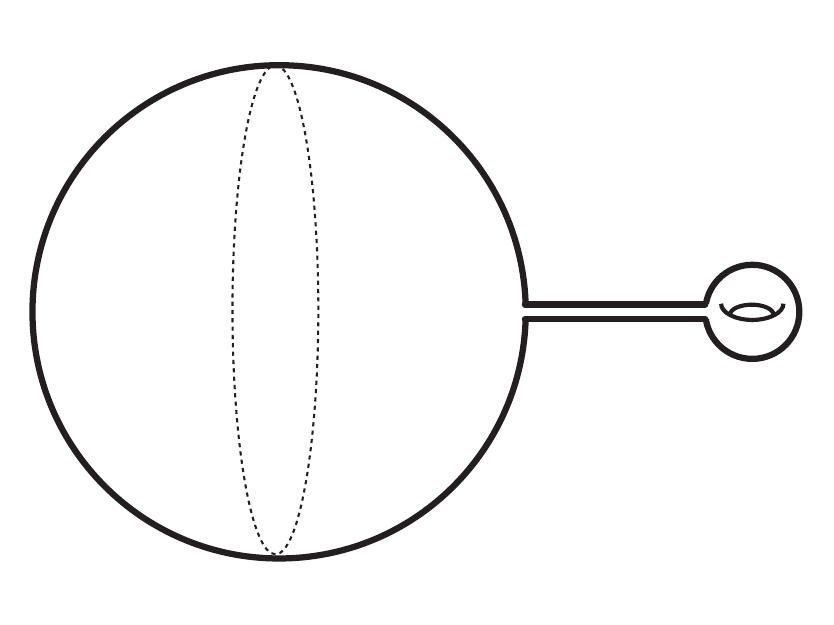}
				\caption{A big sphere $X$ with a small handle. In this case, as $r>0$ increases,  $B_r(X,L^\infty(X))$ changes homotopy type from that of $X$ to that of $\mathbb{S}^2$ as soon as $r>r_0$ for some $r_0<\filrad(X).$}
				\label{fig:spherehandle}
			\end{figure} The following example provides  geometric intuition for how the  homotopy type of Kuratowski neighborhoods may change before $r$ reaches the filling radius. Consider a big sphere with a small handle attached through a long neck (see Figure \ref{fig:spherehandle}). Since the top dimensional hole in this space is big, we expect the filling radius to be big. On the other hand, the $1$-th homology class coming from the small handle dies in a small Kuratowski neighborhood, hence the homotopy type changes at that point.
		\end{enumerate}
	\end{Remark}
	
	We now relate the filling radius of a closed connected $n$-dimensional manifold to its $n$-dimensional Vietoris-Rips persistence barcode.
	
	\propfrfdgm
	
	\begin{framed}
		The unique interval identified by Proposition \ref{prop:filradpersistence} will be henceforth denoted by $$I_{n,\mathbb{F}}^M:=(0,\mathrm{d}_{n,\mathbb{F}}^M].$$
	\end{framed}
	
	\begin{proof}[Proof of Proposition \ref{prop:filradpersistence}]
		First, let us consider the case when $M$ is orientable. Observe that the following diagram commutes:
		$$\begin{tikzcd}\Hom_n(M;\Z)\otimes\mathbb{F}\arrow[d, "j", hook]\arrow[r, "(\iota_r)_*\otimes\mathrm{id}_{\mathbb{F}}"] & \Hom_n(B_r(M,L^\infty(M));\Z)\otimes\mathbb{F}\arrow[d, "j_r", hook]\arrow[r] & \Hom_n(B_s(M,L^\infty(M));\Z)\otimes\mathbb{F}\arrow[d, "j_s", hook]\\ \Hom_n(M;\mathbb{F})\arrow[r, "(\iota_r)_*"] & \Hom_n(B_r(M,L^\infty(M));\mathbb{F})\arrow[r] & \Hom_n(B_s(M,L^\infty(M));\mathbb{F})\end{tikzcd}$$
		for any $0<r\leq s$, where every horizontal arrow is induced by the obvious inclusions, and the vertical arrows ($j,j_r,$ and $j_s$) must be injective by the universal coefficient theorem for homology (see \cite[Theorem 55.1]{munkres2018elements}). Hence, one obtains  that
		
		$$\filrad(M;\mathbb{F})=\inf\{r>0|\,\mathrm{H}_n(\iota_r;\mathbb{F})(j([M]))=0\},$$
		
		and the collection of nonzero homology classes
		$$\{\omega_r:=(\iota_r)_*\circ j([M])\in\Hom_n(B_r(M,L^\infty(M));\mathbb{F})\}_{r\in(0,\filrad(M;\mathbb{F})]}$$
		
		Therefore, with the aid of Theorem \ref{theorem:isom} and Theorem \ref{thm:intervalform}, one  concludes that
		
		$$(0,2\,\filrad(M;\mathbb{F})]\in\dgm_n^\mathrm{VR}(M;\mathbb{F}).$$ 
		
		Also, by Hausmann's theorem \cite[Theorem 3.5]{h95}, $\vr_r(M)$ is homotopy equivalent to $M$ for $r>0$ small enough. Therefore,  $(0,2\,\filrad(M;\mathbb{F})]$ must be the unique interval in $\dgm_n^\mathrm{VR}(M;\mathbb{F})$ with left endpoint $0$. Finally, since $j_r\circ(\iota_r)_*\otimes\mathrm{id}_{\mathbb{F}}=(\iota_r)_*\circ j$ by the commutativity of the diagram, $(\iota_r)_*\circ j([M])\neq 0$ for some $r>0$ implies that $(\iota_r)_*([M])\neq 0$. Therefore, $\filrad(M;\mathbb{F})\leq\filrad(M)$ as we requested.

		The proof of the non-orientable case is similar, so we omit it.
	\end{proof}
	
	\begin{Remark}
		Observe that:
		\begin{enumerate}
			\item $\filrad(\Sp^n;\mathbb{F})=\filrad(\Sp^n)$ for any field $\mathbb{F}$. This can be obtained via Proposition \ref{prop:filradpersistence} and Remark \ref{rmk:onfirsttype}. Alternatively, a more direct proof can be obtained via Jung's Theorem (cf. Theorem \ref{thm:Jung}) following an idea similar to the one used in the proof of \cite[1.2.B. Lemma, 4.5.A Lemma]{gfilling}, \cite[Theorem 2]{k83}. For  completeness, we provide a proof along these lines here.
			
			Since we already know that $\filrad(\Sp^n;\mathbb{F})\leq\filrad(\Sp^n)$ by Proposition \ref{prop:filradpersistence}, it is enough to prove that $\filrad(\Sp^n;\mathbb{F})\geq\filrad(\Sp^n)$. Suppose $\filrad(\Sp^n;\mathbb{F})<\filrad(\Sp^n)$. Then, there is $0<r<\filrad(\Sp^n)$ such that $\mathrm{H}_n(\iota_r;\mathbb{F})([M])=0$ and there are $n$-dimensional singular chain $c_n$ such that $[c_n]=[M]$, and a $(n+1)$-dimensional singular chain $c_{n+1}:=\sum_{i=1}^la_i\sigma_i$ where $a_i\in\mathbb{F}$ and $\sigma_i:\Delta_{n+1}\rightarrow B_r(\Sp^n,L^\infty(\Sp^n))$ for all $i=1,\dots,l$ such that $\partial c_{n+1}=c_n$. Under this assumption, we will show that $[c_n]=0$ which contradicts to the fact that $[M]\neq 0$. One can assume that $c_n$ is a sum of canonical geodesic simplexes. Now, consider the set $V:=\{\sigma_i(v):i=1,\dots,l\text{ and }v\text{ is a vertex of }\Delta_{n+1}\}$. By subdividing $\Delta_{n+1}$ if necessary, one can assume that $\Vert \sigma_i(v)-\sigma_i(w)\Vert_\infty<2\filrad(\Sp^n)-2r$ if $v$ and $w$ are adjacent in the $1$-skeleton of $\Delta_{n+1}$. Next, since $V\subset B_r(\Sp^n,L^\infty(\Sp^n))$, for each $\sigma_i(v)\in V$ there exists a point $x_{\sigma_i(v)}\in\Sp^n$ (not necessarily unique) such that $\Vert \sigma_i(v)-d_{\Sp^n}(x_{\sigma_i(v)},\cdot) \Vert_{\infty}<r$. Hence, by defining $T(\sigma_i(v)):=x_{\sigma_i(v)}$ for each $\sigma_i(v)\in V$, we have a function $T:V\rightarrow\Sp^n$ such that $\Vert \sigma_i(v)-d_{\Sp^n}(T(\sigma_i(v)),\cdot) \Vert_{\infty}<r$ for each $\sigma_i(v)\in V$. Then, $d_{\Sp^n}(T(\sigma_i(v)),T(\sigma_i(w))<2\filrad(\Sp^n)$ if $v$ and $w$ are adjacent because of the triangle inequality. Finally, by Jung's Theorem (cf. Theorem \ref{thm:Jung}) and the  argument employed in \cite[1.2.B. Lemma, 4.5.A Lemma]{gfilling}, \cite[Theorem 2]{k83}, one can define singular simplexes $T_i:\Delta_{n+1}\rightarrow\Sp^n$ such that $T_i$ is the canonical geodesic simplex induced by $\{T_i(v):=T(\sigma_i(v)):v\text{ is a vertex of }\Delta_{n+1}\}$ for all $i=1,\dots,l$. Then, $\bar{c}_{n+1}:=\sum_{i=1}^la_iT_i$ is a $(n+1)$-dimensional singular chain in $\Sp^n$ such that $\partial\bar{c}_{n+1}=c_n$. This means that $[c_n]=[M]=0$ which is a contradiction.
			
			\item It is known that $\filrad(\mathbb{CP}^3;\mathbb{Q})<\filrad(\mathbb{CP}^3)$ by \cite[Theorem 0.1, Proposition 0.3]{katz-cpn}.
		\end{enumerate}
	\end{Remark}
	
	\begin{Remark}
		Actually, one can generalize Proposition \ref{prop:filradpersistence} to metric manifolds. See Proposition \ref{prop:genfilradpersistence} for the full generalization.
	\end{Remark}
	
	\begin{Remark}\label{rem:productfilrad}
		Let $M$ and $N$ be closed connected metric manifolds. Let $M\times N$ denote the $\ell^\infty$-product of $M$ and $N$ (as metric spaces). By Theorem \ref{thm:vrapp}, Remark \ref{rem:productbarcode}, and Proposition \ref{prop:filradpersistence},
		$$\filrad(M \times N;\mathbb{F})=\min (\filrad(M;\mathbb{F}),\filrad(N;\mathbb{F})).$$
		A similar result is true for the $\ell^\infty$-product of more than two metric manifolds.
	\end{Remark}

	\subsubsection{Bounding the filling radius and consequences for Vietoris-Rips persistent homology}\label{sec:bdfilrad}
	
	Proposition \ref{prop:filradpersistence} above permits estimating certain  properties of the barcode $\dgm_n^\mathrm{VR}(M)$ of an $n$-dimensional manifold $M$. 
	
	\paragraph{Injectivity radius and persistence barcodes.} 
	
	If $I_{n,\mathbb{F}}^M=(0,\mathrm{d}_{n,\mathbb{F}}^M]$ is the unique interval in $\dgm_n^\vr(M)$ identified by Proposition \ref{prop:filradpersistence}, then
	\begin{equation}\label{eq:lb-inj}
		\mathrm{d}_{n,\mathbb{F}}^M\geq \frac{\mathrm{Inj}(M)}{n+2}.
	\end{equation}
	This follow from the well known fact that $\filrad(M)\geq \frac{\mathrm{Inj}(M)}{2(n+2)}$, where $\mathrm{Inj}(M)$ is the injectivity radius of $M$ \cite[4.5.A. Lemma]{gfilling}. This implies that for all spheres $M=\mathbb{S}^n$, $\mathrm{d}_{n,\mathbb{F}}^M \geq \frac{\pi}{n+2}.$ Note that Proposition \ref{prop:filradpersistence} indicates that  this estimate is not tight in general since $$\mathrm{d}_n^{\Sp^n} = 2\,\filrad(\mathbb{S}^n)= \arccos\left(\frac{-1}{n+1}\right)\geq \frac{\pi}{2}.$$
	
	\paragraph{Systole and persistence barcodes.} \label{page:sys} The \textbf{systole} $\mathrm{sys}_1(M)$ of a Riemannian manifold $M$ is defined to be the infimal length over non-contractible loops of $M$. In \cite[1.2.B Lemma]{gfilling}, Gromov proved that
	$$\mathrm{sys}_1(M)\leq 6\,\filrad(M)$$
	for any closed \emph{essential} Riemannian manifold $M$.\footnote{See \cite{gfilling} for the definition of essential manifolds. For this paper it suffices to keep in mind that aspherical manifolds are essential.} Moreover, one can also define the \textbf{homological systole} $\mathrm{sysh}_1(M;\mathbb{F})$ to be the infimal length over non null-homologous loops of $M$. In general, $\mathrm{sys}_1(M)\leq\mathrm{sysh}_1(M;\mathbb{F})$ since any contractible loop is null-homologous (see \cite[2.A]{h01}). See Figure \ref{fig:diff-sys} for a space on which the notions differ. In \cite[Theorem 8.10]{Virk18}, Virk proved that $$\left(0,\frac{\mathrm{sysh}_1(M;\mathbb{F})}{3}\right]\in\dgmR_1(M;\mathbb{F})$$ for any closed Riemannian manifold $M$. Observe that the $n$-dimensional torus $\mathbb{T}^n$ is an aspherical, hence essential, manifold. Also, observe that $\mathrm{sys}_1(\mathbb{T}^n)=\mathrm{sysh}_1(\mathbb{T}^n)$ since the fundamental group $\pi_1(\mathbb{T}^n)$ is abelian. Therefore, this permits relating the top dimensional persistence barcode with the first dimensional barcode of any $n$-dimensional Riemannian torus. We  summarize this via the following
	
	\begin{corollary}For \emph{any} Riemannian metric on the $n$-dimensional torus $\mathbb{T}^n$:
		\begin{itemize}
			\item  the interval $I_1^{\mathbb{T}^n}:=\left(0,\frac{\mathrm{sys}_1(\mathbb{T}^n)}{3}\right]$ is an element of $\dgmR_1(\mathbb{T}^n;\mathbb{F})$,
			\item the interval $I_n^{\mathbb{T}^n}:=\left(0,2\,\filrad(\mathbb{T}^n)\right]$ is an element of $\dgmR_n(\mathbb{T}^n;\mathbb{F})$, and
			\item  $I_1^{\mathbb{T}^n}\subseteq I_n^{\mathbb{T}^n}.$
		\end{itemize}
	\end{corollary}
	
	\begin{figure}
		\centering
		\includegraphics[width=.8\textwidth]{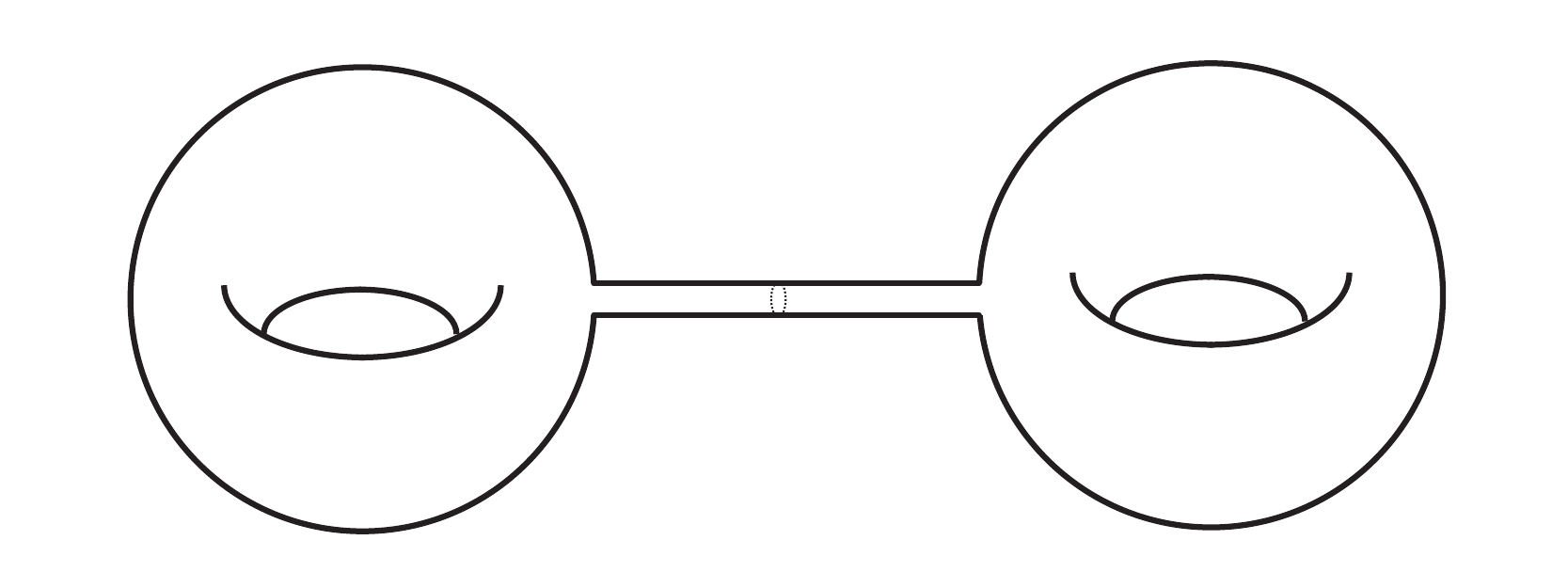}
		\caption{A space $X$ for which $\mathrm{sys}(X)\neq \mathrm{sysh}(X)$. }
		\label{fig:diff-sys}
	\end{figure}
	
	\paragraph{Volume and persistence barcodes.}
	An inequality proved by Gromov in \cite[1.2.A Main Theorem]{gfilling}  states that for each $n$ natural there exists a constant $c_n>0$ such that if $M$ is any $n$-dimensional complete Riemannian manifold then,  
	\begin{equation}\label{eq:ineqs-volfr}
		\filrad(M)\leq c_n\,\big(\mathrm{vol}(M)\big)^{1/n}.
	\end{equation}
	It then follows that
	\begin{equation}\label{eq:vol-ub}
		\mathrm{d}_{n,\mathbb{F}}^M\leq 2\,c_n \big(\mathrm{vol}(M)\big)^{1/n}.
	\end{equation}

	In particular, this bound improves upon the one given by Corollary \ref{coro:diam-bound}, $\mathrm{d}_{n,\mathbb{F}}^M\leq \frac{2}{3}\mathrm{diam}(M)$,  when $M$ is ``thin" like in the case of a thickened embedded graph \cite{osman-reeb}.

	\paragraph{Spread and persistence barcodes.}
	The following proposition is proved in \cite[Lemma 1]{k83} (for integer coefficients). Here we provide a different proof which easily follows from the persistent homology perspective that we have adopted in this paper.
	
	\begin{proposition}\label{prop:spreadfilrad}
		Let $M$ be a closed connected metric manifold. Then, $$\filrad(M;\mathbb{F}) \leq \frac{1}{2}\spread(M).$$
	\end{proposition}
	\begin{proof}
		Follows from Proposition \ref{prop:spread} and Proposition \ref{prop:filradpersistence}.
	\end{proof}
	
	\begin{Remark}
		One can also use Lemma \ref{lemma:katzcontraction} to prove Proposition \ref{prop:spreadfilrad}.
	\end{Remark}
	
	\begin{Remark}
		The inequality in the statement above becomes an equality for spheres \cite{k83}.
	\end{Remark}
	
	By Corollary \ref{coro:diam-bound}, Proposition \ref{prop:filradpersistence}, and the fact that $\filrad(\mathbb{S}^1)=\frac{\pi}{3}$, we know that
	$$\mathrm{length}(I)\leq  \frac{2\pi}{3}=\mathrm{length}(I_1^{\mathbb{S}^1})$$
	for any $k\geq 1$, and any $I\in\dgmR_k(\mathbb{S}^1;\mathbb{F})$. This motivates the following conjecture:
	
	\begin{conjecture}\label{con:barcodlengtbdd}
		Let $M$ be a closed connected $n$-dimensional metric manifold. Then, 
		$$\mathrm{length}(I)\leq \mathrm{length}(I_{n,\mathbb{F}}^M)$$
		for any $I\in\dgmR_k(M;\mathbb{F})$ and any $k\geq 1$.
	\end{conjecture}
	
	However, this conjecture is not true in general, as the following example shows.
	
	\begin{Remark}
		Consider the $\ell^\infty$-product $X = \mathbb{S}^1\times\mathbb{S}^2$. Then, by Remark \ref{rem:productfilrad}, we have
		$$\filrad(X)=\min (\filrad(\mathbb{S}^1),\filrad(\mathbb{S}^2))=\frac{1}{2}\arccos{\left(-\frac{1}{3}\right)}.$$
		This implies that $\mathrm{length}(I_3^X)=2\filrad(X)=\arccos{\left(-\frac{1}{3}\right)}$. Now, we will prove that there is longer interval in $\dgmR_1(X;\mathbb{F})$. First, observe that there is an infinite length interval in $\dgmR_0(\mathbb{S}^2;\mathbb{F})$. Also, $I_1^{\mathbb{S}^1}=(0,2\,\filrad(\mathbb{S}^1)]=(0,\frac{2\pi}{3}]$. Therefore, by the persistent K\"unneth formula (Theorem \ref{thm:vrapp}.(1)), and Remark \ref{rem:productbarcode}, the interval $I=(0,\frac{2\pi}{3}]$ exists in $\dgmR_1(X;\mathbb{F})$.

		Therefore, since $\frac{2\pi}{3}>\arccos{\left(-\frac{1}{3}\right)}$,  Conjecture \ref{con:barcodlengtbdd} is false.
	\end{Remark}

	\subsubsection{Application to obtaining lower bounds for the Gromov-Hausdorff distance}
	
	With the aid of the stability of barcodes (cf. Theorem \ref{thm:stab-barcodes}) and the notion of filling radius, one can obtain the following result:
	
	\begin{proposition}\label{prop:lb-gh-fr}
		Let $M$ be a closed connected $m$-dimensional orientable (resp. non-orientable) Riemannian manifold, and let $X$ be a compact metric space such that:
		\begin{enumerate}
			\item $\Hom_m(X;\mathbb{F})=0$ for some arbitrary field $\mathbb{F}$ (resp. $\Hom_m(X;\mathbb{F})=0$ for $\mathbb{F}=\Z_2$), and
			
			\item $\vr_r(X)\simeq X$ for every $r\in (0,\filrad(M;\mathbb{F})]$.
		\end{enumerate}
		Then, $$d_{\mathrm{B}}(\dgmR_m(M;\mathbb{F}),\dgmR_m(X;\mathbb{F}))\geq \filrad(M;\mathbb{F})$$ and, as a consequence,
		$$\dgh(M,X)\geq\frac{1}{2}\filrad(M;\mathbb{F}).$$
	\end{proposition}
	\begin{proof}
		Observe that by Theorem \ref{thm:isometry} and Theorem \ref{thm:stab-barcodes},
		$$\dgh(M,X)\geq\frac{1}{2}d_{\mathrm{B}}(\dgmR_m(M;\mathbb{F}),\dgmR_m(X;\mathbb{F})).$$
		Hence, it is enough to establish that
		
		$$d_{\mathrm{B}}(\dgmR_m(M;\mathbb{F}),\dgmR_m(X;\mathbb{F}))\geq\filrad(M;\mathbb{F}).$$
		
		Recall that the special interval $I_{m,\mathbb{F}}^M:=(0,2\,\filrad(M;\mathbb{F})]$ belongs to $\dgmR_m(M;\mathbb{F})$ by Proposition \ref{prop:filradpersistence}. Moreover, if $I:=(u,v]\in\dgmR_m(X;\mathbb{F})$, then $u\geq\filrad(M;\mathbb{F})$ by the two assumptions on $X$.

		Now, fix an arbitrary partial matching $P$ between $\dgmR_m(M;\mathbb{F})$ and $\dgmR_m(X;\mathbb{F})$. If $I_{m,\mathbb{F}}^M$ is unmatched to any interval in $\dgmR_m(X;\mathbb{F})$, then $\mathrm{cost}(P)\geq\frac{\vert0-2\,\filrad(M;\mathbb{F})\vert}{2}=\filrad(M;\mathbb{F})$. If $(I_{m,\mathbb{F}}^M,I:=(u,v])\in P$, then $\mathrm{cost}(P)\geq\vert 0-u \vert=u\geq\filrad(M;\mathbb{F})$. Since $P$ is arbitrary, one can conclude $d_{\mathrm{B}}(\dgmR_m(M;\mathbb{F}),\dgmR_m(X;\mathbb{F}))\geq\filrad(M;\mathbb{F})$ as we required.
	\end{proof}
	
	By combining Proposition \ref{prop:lb-gh-fr} with Proposition \ref{prop:spread} we now obtain the \emph{exact value} of the lower bound for $\dgh(\Sp^m,\Sp^n)$ given by invoking the stability of Vietoris-Rips barcodes:
	
	\begin{corollary}\label{coro:dgh-spheres}
		For any positive integers $1\leq m<n$, we have:
		$$\sup_k d_{\mathrm{B}}\left(\dgmR_k(\Sp^m;\mathbb{F}),\dgmR_k(\Sp^n;\mathbb{F})\right) = \filrad(\Sp^m)= \frac{1}{2}\arccos\left(\frac{-1}{m+1}\right)$$
		and, as a consequence,
		$$\dgh(\Sp^m,\Sp^n)\geq \frac{1}{4}\arccos{\left(\frac{-1}{m+1}\right)}\geq \frac{\pi}{8}.$$
	\end{corollary}
	\begin{proof}
		Notice that $\Sp^m$ is orientable, $\Hom_m(\Sp^n;\mathbb{F})=0$ for any field $\mathbb{F}$, $\vr_r(\Sp^n)\simeq\Sp^n$ for any $r\in\left(0,\arccos\left(\frac{-1}{n+1}\right)\right]$ by Theorem \ref{cor:Snfirsttype}, and $\arccos\left(\frac{-1}{n+1}\right)\geq\frac{\pi}{2}\geq\frac{1}{2}\arccos\left(\frac{-1}{m+1}\right)=\filrad(\Sp^m)$. Hence, by Proposition \ref{prop:lb-gh-fr},
		$$\sup_k d_{\mathrm{B}}\left(\dgmR_k(\Sp^m;\mathbb{F}),\dgmR_k(\Sp^n;\mathbb{F})\right)\geq\filrad(\Sp^m)=\frac{1}{2}\arccos{\left(\frac{-1}{m+1}\right)}.$$
		The reverse inequality follows from  Proposition \ref{prop:spread} and Remarks \ref{rem:spread-sn} and \ref{rmk:firstextremum} relating the spread to the filling radius of spheres. Indeed, by basic properties of the bottleneck distance,\footnote{The cost of the empty matching upper bounds the bottleneck distance.} for every integer $k\geq 0$, $$d_{\mathrm{B}}(\dgmR_k(\Sp^m;\mathbb{F}),\dgmR_k(\Sp^n;\mathbb{F}))\leq \frac{1}{2}\max\left(\max_{I\in \dgmR_k(\Sp^m;\mathbb{F})}\mathrm{length}(I),\max_{J\in \dgmR_k(\Sp^n;\mathbb{F})}\mathrm{length}(J)\right).$$
		Now, by Proposition \ref{prop:spread}, the RHS is bounded above by $\frac{1}{2}\max\big(\spread(\Sp^m),\spread(\Sp^n)\big)$ which by Remark \ref{rem:spread-sn} is equal to $\frac{1}{2}\arccos\left(\frac{-1}{m+1}\right)$ and in turn equal to $\filrad(\Sp^m)$ by Remark \ref{rmk:firstextremum}.
	\end{proof}

	\begin{Remark}
		The lower bounds provided by Corollary \ref{coro:dgh-spheres} are  non-optimal:  See \cite{dgh-spheres} for improved lower bounds via considerations based on a certain version of the Borsuk-Ulam theorem. In fact there the factor $\frac{1}{2}$ is removed leading for example to the bound $\dgh(\Sp^m,\Sp^n)\geq \filrad(\Sp^{\min(m,n)})$ for all $0\leq m<n\leq\infty$. This bound is therein shown to be tight when $(m,n)\in\{(1,2),(1,3),(2,3)\}$ via the construction of suitable correspondences.
		Via Example \ref{ex:int-zero}, Theorem \ref{thm:stab-barcodes}, and Proposition \ref{prop:filradpersistence}, one can directly see that for any geodesic, compact, and simply connected space $Y$,
		$$\dgh(\Sp^1,Y)\geq \frac{\pi}{6}.$$ 
		This, taken together with the comments above leads to the following conjecture.
	\end{Remark}
	
	\begin{conjecture}
		For any geodesic, compact, and simply connected space $Y$ we have
		$$\dgh(\Sp^1,Y)\geq \frac{\pi}{3}.$$
	\end{conjecture}
	
	\subsubsection{A generalization of the  filling radius}
	
	\begin{figure}
		\centering
		\includegraphics[width=\linewidth]{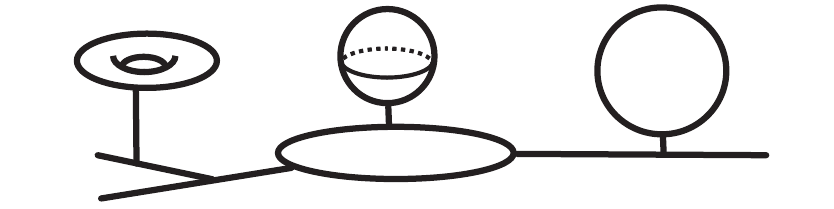}
		\caption{An ANR space contemplated by Proposition \ref{prop:genfilradpersistence}.}
		\label{fig:anr}
	\end{figure}
	
	The goal of this section is to provide some partial results regarding the structure of $\dgmR_\ast(\cdot)$ for non-smooth spaces; see Figure \ref{fig:anr}. In order to do so we consider a generalization of the notion of filling radius for arbitrary compact ANR metric spaces and  arbitrary homology dimension. See \cite{hu1965theory} for an  introduction to the general theory of ANRs.
	
	\begin{definition}[Absolute neighborhood retract]
		A metric space $(X,d_X)$ is said to be \emph{ANR (Absolute Neighborhood Retract)} if, whenever $X$ is a subspace of another metric space $Y$, there exists an open set $X\subset U\subseteq Y$ such that $X$ is a retract of $U$.
	\end{definition}

	It is known that every topological manifold with compatible metric (so, a metric manifold) is an ANR. Not only that, every locally Euclidean metric space is an ANR (see \cite[Theorem III.8.1]{hu1965theory}). Also, every compact, (topologically) finite dimensional, and locally contractible metric space is ANR (see \cite[Section 1]{dugundji}). The following example is one application of this fact.
	
	\begin{example}
		Let $\mathcal{G}$ be a compact metric graph and $M_1,\dots,M_n$ be closed connected metric manifolds. Choose points $v_1,\dots,v_n\in G$ and $p_i\in M_i$ for each $i=1,\dots,n$ and consider the geodesic metric space
		$X:=\mathcal{G}\vee M_1\vee\cdots\vee M_n$
		arising from metric gluings via $v_1\sim p_1,\dots,v_n\sim p_n$. Since $X$ is compact, (topologically) finite dimensional, and locally contractible, it is an ANR. See Figure \ref{fig:anr}.
	\end{example}
	
	Finally, we are ready to define a generalized filling radius.
	
	\begin{definition}[Generalized filling radius]
		Let $(X,E)$ be a metric pair where $X$ is a compact ANR metric space. For any integer $k\geq 1$, any abelian group $G$, and any $\omega\in \mathrm{H}_k(X;G)$, we define the generalized filling radius as follows:
		$$\filrad_k((X,E),G,\omega):=\inf\{r>0|\,\mathrm{H}_k(\iota_r^E;G)(\omega)=0\},$$
		where $\iota_r^E:X\hookrightarrow B_r(X,E)$ is the (corestriction of the) isometric embedding. In other words, we have the map
		$$\filrad_k((X,E),G,\cdot):\mathrm{H}_k(X;G)\longrightarrow \mathbb{R}_{\geq 0}.$$
	\end{definition}
	
	\begin{Remark}
		Following the discussion in Remark \ref{rem:relative-filling-radius} after Equation (\ref{eq:rel-fill-rad}), one can also prove that the smallest possible value of the generalized filling radius is attained when $E$ is an injective metric space. Hence, we denote $\filrad_k(X,G,\omega)$ instead of $\filrad_k((X,E),G,\omega)$ whenever $E$ is injective, for  simplicity.
	\end{Remark}
	
	Let $M$ be an $n$-dimensional metric manifold. Then, note that we have $\filrad_n(M,G,[M]) = \filrad(M)$ in the following two cases: (1) when $M$ is orientable and $G=\mathbb{Z}$, and (2) when $M$ is non-orientable and $G=\mathbb{Z}_2$. 
	
	A priori, one can define the generalized filling radius for any metric space $X$. However, we believe that the context of ANR metric spaces is the right level of generalization for our purposes because of the following  proposition analogous  to Proposition \ref{prop:filradpersistence}.
	
	\begin{proposition}\label{prop:genfilradpersistence}
		Let $X$ be a compact ANR metric space. Then, for any $k\geq 1$ and nonzero $\omega\in\mathrm{H}_k(X;\mathbb{F})$, we have $\filrad_k(X,\mathbb{F},\omega)>0$, and
		$$(0,2\,\filrad_k(X,\mathbb{F},\omega)]\in\dgm_k^\mathrm{VR}(X;\mathbb{F})$$
		where $\mathbb{F}$ is an arbitrary field.
	\end{proposition}
	\begin{proof}
		First, note that one cannot apply Hausmann's theorem since $X$ is not necessarily a Riemannian manifold. However, since $X$ is ANR and a closed subset of $L^\infty(X)$, there exists an open $U\subseteq L^\infty(X)$ such that $X\subset U$ and $U$ retracts onto $X$. Let $\rho:U\rightarrow X$ be the retraction. Now, since $U$ is open there exists $r>0$ such that $B_r(X,L^\infty(X))\subseteq U$. Observe that the restriction $\rho_r:=\rho|_{B_r(X,L^\infty(X))}:B_r(X,L^\infty(X)) \rightarrow X$ is still a retraction. It means that $\rho_r\circ\iota_r=\mathrm{id}_X$. Therefore,
		$$\mathrm{H}_k(\iota_r;\mathbb{F}):\mathrm{H}_k(X;\mathbb{F})\longrightarrow\mathrm{H}_k(B_r(X,L^\infty(X));\mathbb{F})$$
		is injective. This implies that $\filrad_k(X,\mathbb{F},\omega)>0$ and that there exists some interval in $\dgmR_k(X;\mathbb{F})$ corresponding to the nonzero homology class $\omega\in\mathrm{H}_k(X;\mathbb{F})$.

		The remaining part of proof is essentially the same as the proof of Proposition \ref{prop:filradpersistence}, so we omit it.
	\end{proof}
	
	\begin{example}
		For any nonzero $\omega\in\mathrm{H}_1(M;\mathbb{F})$ with an arbitrary field $\mathbb{F}$, because of the result in \cite[Theorem 8.10]{Virk18}, one has that $2\,\filrad_1(M,\mathbb{F},\omega)=\frac{\mathrm{length}(\gamma)}{3}$ where $\gamma$ is a shortest closed curve representing the homology class $\omega$.
	\end{example}

	\medskip
	\paragraph{A refinement for the case $k=1$.}
	
	We now prove that when $k=1$, the intervals given by Proposition \ref{prop:genfilradpersistence} are the \emph{only} bars in $\dgmR_1(X;\mathbb{F}).$
	
	\begin{lemma}\label{lemma:surjectivepi1}
		Let $X$ be a compact geodesic metric space, which is a subspace of an injective metric space $(E,d_E)$. Then, for any $r>0$, the canonical inclusion $\iota_r:X\longhookrightarrow B_r(X,E)$ induces a surjection at the level of fundamental groups. In particular, this implies $\iota_r$ induces a surjection at the level of first degree of homology, also.
	\end{lemma}
	\begin{proof}
		Let $\gamma:[0,1] \to B_r(X,E)$ be an arbitrary continuous path with endpoints $x,x'$ in $X$. It is enough to show that $\gamma$ is homotopy equivalent to a path in $X$ relative to its endpoints. By Lebesgue number lemma, one can choose $0=t_0 < t_1 < \dots <t_n=1$ such that there exists $x_i \in X$ satisfying $\gamma([t_{i-1},t_i]) \subseteq B_r(x_i,E)$ for each $i\in\{1,\dots,n\}$. Let $y_i:=\gamma(t_i)$ for $i\in\{0,\dots,n\}$. Since each $B_r(x_i,E)$ is contractible by Lemma \ref{lem:intersectballs}, one can choose continuous paths $\alpha_i,\beta_i$ contained in $B_r(x_i,E)$ such that $\alpha_i$ is from $y_{i-1}$ to $x_i$ and $\beta_i$ is from $x_i$ to $y_i$ respectively. As $B_r(x_i,E)$ is contractible, $\gamma|_{[t_{i-1},t_i]}$ is homotopy equivalent to $\alpha_i \cdot\beta_i$ relative to endpoints. Hence we have 
		$$\gamma \simeq (\alpha_1 \ast \beta_1) \ast \dots \ast (\alpha_n \ast \beta_n)$$
		relative to endpoints. Note that $\alpha_1$ and $\beta_n$ can be chosen as geodesics in $X$ as they connect $x,x_1$ and $x_n,x'$ in $B_r(x_1,E),B_r(x_n,E)$ respectively. Hence it is enough to show that
		$$(\beta_1 \ast \alpha_2) \ast \dots \ast (\beta_{n-1} \ast \alpha_n)$$
		is homotopy equivalent to a path in $X$ relative to endpoints. Let us show that $\beta_i \cdot \alpha_{i+1}$ is homotopy equivalent to a path in $X$ for each $i$. Let $p$ be a midpoint of $x_i,x_{i+1}$ in $X$. Note that $p$ and $y_i$ are contained in $B_r(x_i,E) \cap B_r(x_{i+1},E)$, which is contractible (again by Lemma \ref{lem:intersectballs}). Let $\theta$ be a path in that intersection from $y_i$ to $p$. Let $\gamma_{x_i,p}$ be a shortest geodesic in $X$ from $x_i$ to $p$ and $\gamma_{p,x_{i+1}}$ be a shortest geodesic in $X$ from $p$ to $x_{i+1}$. Note that $\gamma_{x_i,p} \cdot \bar{\theta}$ is contained in $B_r(x_{i})$ and has endpoints $x_i,y_i$ hence it is homotopy equivalent to $\beta_i$ relative to endpoints. Similarly $\theta \cdot \gamma_{p,x_{i+1}}$ is homotopy equivalent to $\alpha_{i+1}$ relative to endpoints. Hence
		\begin{align*}
			\beta_i \cdot \alpha_{i+1} &\simeq \gamma_{x_i,p} \cdot \bar{\theta} \cdot \theta \cdot \gamma_{p,x_{i+1}} \\
			&\simeq \gamma_{x_i,p} \cdot \gamma_{p,x_{i+1}}
		\end{align*}
		relative to endpoints. This completes the proof of the first claim.

		For the second claim, exploit \cite[Theorem 2A.1]{h01}.
	\end{proof}

	In \cite[Theorem 8.10]{Virk18}, Z. Virk provided a proof of the Corollary below which takes place at the simplicial level. The proof we give below exploits the hyperconvexity properties of $L^\infty(X)$ and also our isomophism theorem, Theorem \ref{theorem:isom}. Given our main results, we can give a more concise proof. See \cite[Section 3]{dey17} for related results.

	\begin{corollary}
		Let $X$ be a compact geodesic metric space. Then, for any $I\in\dgmR_1(X;\mathbb{F})$, there exists $\omega\in\Hom_1(X;\Z)$ such that $I=(0,2\,\filrad_1(X,\Z,\omega)]$.
	\end{corollary}
	\begin{proof}
		Apply Lemma \ref{lemma:surjectivepi1} and Theorem \ref{theorem:isom}.
	\end{proof}

	\paragraph{A conjecture.}

	After seeing the proof of Proposition \ref{prop:genfilradpersistence}, some readers might wonder whether one can prove a version of Hausmann's theorem \cite[Theorem 3.5]{h95} for compact ANR metric spaces. This leads to formulating the conjecture below.
	
	\begin{conjecture}
		Let $(X,d_X)$ be a compact ANR metric space. Then, there exists $r(X)>0$ such that $\vr_r(X)$ is homotopy equivalent to $X$ for any $r\in(0,r(X)]$.
	\end{conjecture}

	\subsection{Rigidity of spheres}\label{sec:rigidity}
	A problem of interest in the area of persistent homology is that of deciding how much information from a metric space is captured by its associated persistent homology invariants. One basic (admittedly imprecise) question that we posed in page \pageref{fig:d2-s1-tight-span} is: 
	
	\qexact
	
	\begin{figure}
		\centering
		\includegraphics[scale=1]{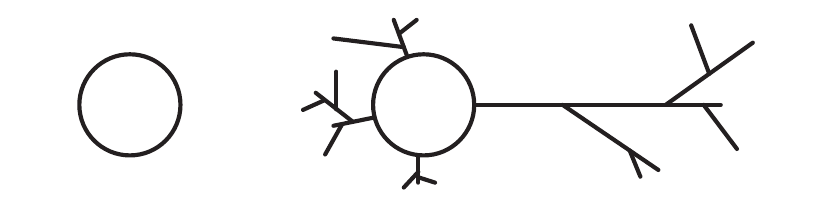}\caption{Two geodesic spaces with the same Vietoris-Rips persistence barcodes. Notice that these spaces are at a large Gromov-Hausdorff distance. \label{fig:s1-tree}}
	\end{figure}
	
	As proved in \cite{mo18} via the notion of \emph{core} of a metric graph or as a consequence  of \cite[Proposition 2.2]{h95}, the unit circle $\mathbb{S}^1$ and the join $X$ of $\mathbb{S}^1$ with disjoint trees of arbitrary length (regarded as a geodesic metric space) have the same Vietoris-Rips persistence barcodes (for all dimensions), see Figure \ref{fig:s1-tree}. However, by increasing the length of the trees attached  these two spaces are at arbitrarily large Gromov-Hausdorff distance, as shown  in Figure \ref{fig:s1-tree}. This means that, in full generality, Question \ref{q:exact} does not admit a reasonable answer if ``similarity" is measure in a strict metric sense via the Gromov-Hausdorff distance.
	
	A related type of questions one might pose are of the type:
	
	\begin{customquestion}{4}\label{q:approx}
		Let $\mathcal{C}$ be a given class of compact metric spaces. Does there exist $\epsilon_\mathcal{C}>0$ such that whenever $d_\mathrm{B}(\dgmR_\ast(X),\dgmR_\ast(Y))<\epsilon_\mathcal{C}$ for some $X,Y\in \mathcal{C}$, then $X$ and $Y$ are homotopy equivalent? 
	\end{customquestion}
	
	Answers to questions such as Questions \ref{q:exact} and \ref{q:approx}  above (together with Questions \ref{customq:estimator}, \ref{customq:estimator-tight}, and \ref{customq:estimator-reverse} in page \pageref{customq:estimator}) are not currently known in full generality. One might then consider ``localized" versions of the above questions: fix some special compact metric space $X_0$, and then assume $Y$ satisfies the respective conditions stipulated in the above question statements.
	
	In this regard, from work by Wilhelm \cite[Main Theorem 2]{wilhem} and Proposition \ref{prop:filradpersistence} we immediately obtain the following corollary for the case of Riemannian manifolds.
	
	\cororigid
	
	\begin{Remark}
		Note that the case of $n=1$ is simpler. Let $M$ be an arbitrary closed connected $1$-dimensional Riemannian manifold. Then, $M$ is isometric to $r\cdot\mathbb{S}^1$ for some $r>0$ and $I_{1,\mathbb{F}}^M=\left(0,\frac{2\pi r}{3}\right]$. Hence, $I_{1,\mathbb{F}}^M=I_1^{\mathbb{S}^1}$ obviously implies $M$ is isometric to $\mathbb{S}^1$.
	\end{Remark}
	
	\begin{Remark}
		Wilhelm's method of proof does not yield an explicit value for the parameter $\epsilon_n$ given in item 2 above. 
		Wilhelm's rigidity result was extended to Alexandrov spaces by Yokota \cite{yokota}, so Corollary \ref{Cor:Wilhelmrigidity} can be generalized to that context.
	\end{Remark}
	
	\begin{example}[A one parameter family of surfaces with the same Filling Radius as $\Sp^2$]\label{ex:one-parameter}
		If we ignore the sectional curvature condition in Corollary \ref{Cor:Wilhelmrigidity}, then for each $\varepsilon>0$ small enough one can construct a one-parameter family $ \{\Sp_h^2:h\in[0,\filrad(\Sp^2)-\varepsilon]\}$ of surfaces with the same filling radius as $\Sp^2$  such that $\Sp^2_0=\Sp^2$ but $\Sp_h^2$ is not isometric to $\Sp^2$ for any $h>0$.  This phenomenon is analogous to the one depicted in Figure \ref{fig:s1-tree}.
		
		\medskip
		Here is the construction (cf. Figure \ref{fig:oneparameter}):
		
		\medskip
		Let $u_1,u_2,u_3,u_4$ be the vertices of a regular tetrahedron inscribed in $\Sp^2$. Hence, $d_{\Sp^2}(u_i,u_j)=2\,\filrad(\Sp^2)$ for any $i \neq j$. Now, let $T$ be a very small spherical triangle contained inside the spherical triangle determined by the points $u_1,u_2,u_3$ as in Figure \ref{fig:oneparameter} (left). In other words, we choose  $\varepsilon:=\diam(T) \ll 2\,\filrad(\Sp^2)$.

		Now, for any $h\geq 0$, we define $\Sp^2_h$ in the following way:
		
		$$\Sp^2_h:=\Sp^2\backslash\mathrm{Int}(T)\times\{0\}\,\,\bigcup\,\,\partial T\times [0,h]\,\,\bigcup\,\, T\times \{h\}\subsetneq \Sp^2\times [0,h]$$
		with the metric $$d_{\Sp^2_h}((x,s),(y,t)):=d_{\Sp^2}(x,y)+\vert s-t \vert.$$ Then, $\Sp^2_h$ is a $2$-dimensional metric manifold. See Figure \ref{fig:oneparameter} (right) for the description of $\Sp^2_h$. Also, note that the following map is $1$-Lipschitz.
		\begin{align*}
			P_h:\Sp^2_h&\longrightarrow\Sp^2\\
			(x,s)&\longmapsto x
		\end{align*}
		
		\begin{claim}
			First, we claim that $\filrad(\Sp^2_h)\geq\filrad(\Sp^2)$ for any $h\geq 0$. 
		\end{claim}
		\begin{proof}
			Note that, since $P_h$ is $1$-Lipschitz, the following diagram commutes for any $r>0$.
			
			$$\begin{tikzcd}
				\Sp^2_h \arrow[r, hook] \arrow[d, "P_h"] & B_r(\Sp^2_h,L^\infty(\Sp^2_h))\arrow[d, "\tilde{P}_h"] \\
				\Sp^2 \arrow[r, hook] & B_r(\Sp^2,L^\infty(\Sp^2))
			\end{tikzcd}$$
			
			Moreover, since $(P_h)_\ast([\Sp^2_h])=[\Sp^2]$, the diagram above implies $\filrad(\Sp^2_h)\geq\filrad(\Sp^2)$ as we required.
		\end{proof}
		
		\begin{claim}
			Next, we claim that $\filrad(\Sp^2_h)\leq\filrad(\Sp^2)$ whenever $h+\varepsilon\leq 2\,\filrad(\Sp^2)$. 
		\end{claim}
		\begin{proof}
			For this we will prove that the spread of $\Sp^2_h$ is bounded by twice the filling radius of $\Sp^2$.
			
			Note that the set $\{(u_1,0),(u_2,0),(u_3,0),(u_4,0)\}\subset\Sp^2_h$ satisfies the following conditions,
			
			\begin{enumerate}
				\item $\diam(\{(u_1,0),(u_2,0),(u_3,0),(u_4,0)\})=2\,\filrad(\Sp^2)$, and
				
				\item $\min_{i=1,2,3,4}d_{\Sp^2_h}((x,s),(u_i,0))\leq 2\,\filrad(\Sp^2)$ for any $(x,s)\in\Sp^2_h$.
			\end{enumerate}
			
			Observe that the second condition holds because if $(x,s)\in\partial T\times [0,h]\cup T\times \{h\}$ (the triangular cylinder with its cap), $d_{\Sp^2_h}((x,s),(u_1,0))=d_{\Sp^2}(x,u_1)+s\leq h+\varepsilon\leq 2\,\filrad(\Sp^2)$.

			Hence, by Proposition \ref{prop:spreadfilrad}, $\filrad(\Sp^2_h)\leq\frac{1}{2}\spread(\Sp^2_h)\leq\filrad(\Sp^2)$.\vspace{\baselineskip}
		\end{proof}
		
		We then  conclude that $\filrad(\Sp^2)=\filrad(\Sp^2_h)$ whenever $h\in [0,2\,\filrad(\Sp^2)-\varepsilon]$.
		
	\end{example}
	
	\begin{Remark}
		Note that the above construction can be generalized to $\Sp^n$ for $n\geq 3$. Also, the small subset $T$ is  need not be a spherical triangle in general, though the argument becomes more involved in that case. For example, one can choose $T$ to be a small geodesic disk on $\Sp^2$.
	\end{Remark}
	
	\paragraph{Rigidity theorems with respect to the bottleneck distance.}
	
	Proposition \ref{prop:LiuWilhelm1} and \ref{prop:LiuWilhelm2}  below provide rigidity results with respect to the bottleneck distance (see Definition \ref{def:bottledis}).

	For the rest of this subsection we will assume that an arbitrary constant $c\geq 1$ is fixed. 
	
	\begin{proposition}\label{prop:LiuWilhelm1}
		Suppose $M$ is a closed connected $n$-dimensional Riemannian manifold with sectional curvature $\mathrm{K}_M\in[1,c]$ and injectivity radius $\mathrm{Inj}(M)\geq\frac{\pi}{2\sqrt{c}}$, then:
		\begin{enumerate}
			\item There exists $\varepsilon_n>0$ such that, if $$d_\mathrm{B}\big(\dgmR_n(M;\mathbb{F}),\dgmR_n(\mathbb{S}^n;\mathbb{F})\big)<\min\{\varepsilon_n,\frac{1}{\pi\sqrt{c}}\cdot\filrad(\mathbb{S}^n)\},$$
			then $M$ is diffeomorphic to $\Sp^n$.
			
			\item If $d_\mathrm{B}\big(\dgmR_n(M;\mathbb{F}),\dgmR_n(\mathbb{S}^n;\mathbb{F})\big)<\min\{2\filrad(\mathbb{S}^n)-\frac{\pi}{3},\frac{1}{\pi\sqrt{c}}\cdot\filrad(\mathbb{S}^n)\}$, then $M$ is a twisted $n$-sphere (and in particular, homotopy equivalent to the $n$-sphere).
		\end{enumerate}
	\end{proposition}
	
	If $M$ is even dimensional, then we can drop the assumption on the injectivity radius.
	
	\begin{proposition}\label{prop:LiuWilhelm2}
		Suppose $M$ is a closed connected $n$-dimensional Riemannian manifold with sectional curvature $\mathrm{K}_M\in[1,c]$ for even $n$, then:
		\begin{enumerate}
			\item There exists $\varepsilon_n>0$ such that, if $$d_\mathrm{B}\big(\dgmR_n(M;\mathbb{F}),\dgmR_n(\mathbb{S}^n;\mathbb{F})\big)<\min(\varepsilon_n,\frac{1}{4\sqrt{c}}\cdot\filrad(\mathbb{S}^n)),$$
			then $M$ is diffeomorphic to $\Sp^n$.
			
			\item If $d_\mathrm{B}\big(\dgmR_n(M;\mathbb{F}),\dgmR_n(\mathbb{S}^n;\mathbb{F})\big)<\min(2\filrad(\mathbb{S}^n)-\frac{\pi}{3},\frac{1}{4\sqrt{c}}\cdot\filrad(\mathbb{S}^n))$, then $M$ is a twisted $n$-sphere (and in particular, homotopy equivalent to the $n$-sphere).
		\end{enumerate}
	\end{proposition}
	
	\begin{lemma}\label{lemma:intdisfilradineq}
		Let $M$ be a closed connected $n$-dimensional Riemannian manifold. If $$d_\mathrm{B}\big(\dgmR_n(M;\mathbb{F}),\dgmR_n(\mathbb{S}^n;\mathbb{F})\big)<\varepsilon$$
		for some $\varepsilon\in(0,2\filrad(\mathbb{S}^n)]$, then either
		\begin{enumerate}
			\item $\filrad(M;\mathbb{F})<\varepsilon$, or
			\item $2\vert \filrad(M;\mathbb{F})-\filrad(\mathbb{S}^n)\vert<\varepsilon$.
		\end{enumerate}
	\end{lemma}
	\begin{proof}
		By Proposition \ref{prop:filradpersistence}, we know that $(0,2\filrad(M;\mathbb{F})]\in\dgmR_n(M;\mathbb{F})$. Suppose
		$$d_\mathrm{B}\big(\dgmR_n(M;\mathbb{F}),\dgmR_n(\mathbb{S}^n;\mathbb{F})\big)<\varepsilon$$
		for some $\varepsilon\in(0,2\filrad(\mathbb{S}^n)]$. Then, there is a partial matching (See Definition \ref{def:bottledis}) $R_\varepsilon$ between $\dgmR_n(M;\mathbb{F})$ and $\dgmR_n(\mathbb{S}^n;\mathbb{F})$ such that $\mathrm{cost}(R_\varepsilon)<\varepsilon$. Consider the following two cases:
		\begin{enumerate}
			\item Suppose the interval $(0,2\filrad(M;\mathbb{F})]$ is not matched to any interval in $\dgmR_n(\mathbb{S}^n;\mathbb{F})$. Then,
			$$\filrad(M;\mathbb{F})\leq \mathrm{cost}(R_\varepsilon)<\varepsilon.$$
			
			\item Suppose $(0,2\filrad(M;\mathbb{F})]$ is matched to some interval $(u,v]\in\dgmR_n(\mathbb{S}^n;\mathbb{F})$ in the partial matching $R_\varepsilon$. Then, we claim that $(u,v]=(0,2\filrad(\mathbb{S}^n)]$. Suppose not. Since we know that $\mathrm{VR}_{r}(\mathbb{S}^n)\simeq \mathbb{S}^n$ for any $r\in(0,2\filrad(\mathbb{S}^n)]$ by Theorem \ref{cor:Snfirsttype}, any interval in $\dgmR_n(\mathbb{S}^n;\mathbb{F})$ other than $(0,2\filrad(\Sp^n)]$ must be born after $2\filrad(\mathbb{S}^n)$. In particular, $u\geq 2\filrad(\mathbb{S}^n)$. This implies
			$$2\filrad(\mathbb{S}^n)\leq \vert u-0 \vert\leq \mathrm{cost}(R_\varepsilon)<\varepsilon\leq 2\filrad(\mathbb{S}^n),$$
			which is contradiction. Hence, $(0,2\filrad(M;\mathbb{F})]$ is matched to $(0,2\filrad(\mathbb{S}^n)]$ in the optimal matching. Therefore,
			$$2\vert \filrad(M;\mathbb{F})-\filrad(\mathbb{S}^n)\vert\leq \mathrm{cost}(R_\varepsilon)<\varepsilon.$$
		\end{enumerate}
	\end{proof}
	
	The proof strategy for Propositions \ref{prop:LiuWilhelm1} and \ref{prop:LiuWilhelm2} is to invoke Wilhelm's result \cite[Main Theorem 2]{wilhem} and  Lemma \ref{lemma:intdisfilradineq} above. However, if $\filrad(M)$ were small,  one would not be able to apply Wilhelm's theorem. To avoid that, we will invoke a result due to Liu \cite{liu}.
	
	\begin{proof}[Proof of Proposition \ref{prop:LiuWilhelm1}]
		Since $c\geq 1$, $\frac{1}{\pi\sqrt{c}}\cdot\filrad(\mathbb{S}^n)\leq 2\filrad(\mathbb{S}^n)$.
		\begin{enumerate}
			\item By Corollary \ref{Cor:Wilhelmrigidity}.(3), there is $\varepsilon_n>0$ such that $2\vert \filrad(M;\mathbb{F})-\filrad(\mathbb{S}^n)\vert<\varepsilon_n$ implies $M$ are diffeomorphic to $\mathbb{S}^n$.

			Suppose $d_\mathrm{B}\big(\dgmR_n(M;\mathbb{F}),\dgmR_n(\mathbb{S}^n;\mathbb{F})\big)<\min\{\varepsilon_n,\frac{1}{\pi\sqrt{c}}\cdot\filrad(\mathbb{S}^n)\}$. Then, we have $\filrad(M;\mathbb{F})<\min\{\varepsilon_n,\frac{1}{\pi\sqrt{c}}\cdot\filrad(\mathbb{S}^n)\}$, or $2\vert \filrad(M;\mathbb{F})-\filrad(\mathbb{S}^n)\vert<\min\{\varepsilon_n,\frac{1}{\pi\sqrt{c}}\cdot\filrad(\mathbb{S}^n)\}$ by Lemma \ref{lemma:intdisfilradineq}. However, the first case cannot happen since $\filrad(M;\mathbb{F})\geq\frac{1}{\pi\sqrt{c}}\cdot\filrad(\mathbb{S}^n)$ by \cite[Proofs of Theorem 1.1 and Proposition 1.6]{liu}. Therefore, $2\vert \filrad(M;\mathbb{F})-\filrad(\mathbb{S}^n)\vert<\min\{\varepsilon_n,\frac{1}{\pi\sqrt{c}}\cdot\filrad(\mathbb{S}^n)\}\leq\varepsilon_n$ so that $M$ and $\mathbb{S}^n$ are diffeomorphic.
			
			\item By basically the same argument, $d_\mathrm{B}\big(\dgmR_n(M;\mathbb{F}),\dgmR_n(\mathbb{S}^n;\mathbb{F})\big)<\min\{2\filrad(\mathbb{S}^n)-\frac{\pi}{3},\frac{1}{\pi\sqrt{c}}\cdot\filrad(\mathbb{S}^n)\}$ implies $2\vert \filrad(M;\mathbb{F})-\filrad(\mathbb{S}^n)\vert<2\filrad(\mathbb{S}^n)-\frac{\pi}{3}$. Therefore, we have $\mathrm{length}(I_{n,\mathbb{F}}^M)> \frac{\pi}{3}$ so that $M$ is a twisted $n$-sphere.
		\end{enumerate}
	\end{proof}

	\begin{proof}[Proof of Proposition \ref{prop:LiuWilhelm2}]
		The proof is basically the same as the proof of Proposition \ref{prop:LiuWilhelm1}. The only difference is we have to use \cite[Remark 1.8.(3)]{liu} instead of \cite[Proposition 1.6]{liu}.
	\end{proof}

	\subsection{Stability of the filling radius}
	
	In \cite{liu}, Liu studies the mapping properties of the filling radius. His results can be interpreted as providing certain guarantees for how the  filling radius changes under \emph{multiplicative} distortion of metrics.  Here we study the effect of additive distortion.
	
	\begin{customquestion}{5}\label{question:stab-fr}
		Under suitable restrictions,  there exists a constant $L>0$ such that for all closed connected metric manifolds $M$ and $N$ it holds that 
		\begin{equation}\label{eq:stab-fr}\big|\filrad(M)-\filrad(N)\big| \leq L\cdot \dgh(M,N).\end{equation} 
	\end{customquestion}
	
	This question is whether the filling radius could be stable as a map from the collection of all metric manifolds to the real line. The answer is negative, as the following example proves.
	
	\begin{example}[Counterexample for manifolds with different dimension] \label{ex:diff-dim}
		Fix $\epsilon>0$ and let $M = \mathbb{S}^1$ and $N_\epsilon = \mathbb{S}^1\times (\epsilon \cdot \mathbb{S}^1)$, a thin torus. Then, it is clear that $\dgh(M,N_\epsilon)\leq \epsilon$ whereas $\filrad(\mathbb{S}^1)=\frac{\pi}{3}$ and, by Remark \ref{rem:productfilrad}, $\filrad(N_\epsilon) = \frac{\pi}{3}\epsilon$.  This means that equation (\ref{eq:stab-fr}) cannot hold in general.
	\end{example}
	
	A subsequent possibility is considering only manifolds with the same dimension. The answer in this case is also negative:
	
	\begin{example}[Counterexample for manifolds with the same dimension]\label{ex:same-dim} Let $n\geq 2$ be any integer and $\epsilon,\delta>0$; we assume that $\delta\ll \epsilon$ so that a certain tubular neighborhood construction described below works. Consider $M = \mathbb{S}^n\subset \mathbb{R}^{n+1}$. Endow $\mathbb{S}^n$ with the usual round  Riemannian metric. Let $G_\epsilon$ be a (finite) metric graph embedded in $\mathbb{S}^n$ such that $\dgh(\mathbb{S}^n,G_\epsilon)<\epsilon$; such graphs always exist for compact geodesic spaces \cite[Proposition 7.5.5]{burago2001course}. Now, let $N_{\epsilon,\delta}$ be (a suitably smoothed out version of) the boundary of the $\delta$-tubular neighborhood of $G_\epsilon$ in $\mathbb{R}^{n+1}$. Now, $\dgh(M,N_{\epsilon,\delta})\leq C\cdot (\epsilon +\delta),$ for some constant $C>0$ whose exact value is not relevant. However, $\filrad(M) = \frac{1}{2}\arccos\left(\frac{-1}{n+1}\right)\geq \frac{\pi}{4}$ whereas $\filrad(N_{\epsilon,\delta}) \leq C_n\cdot \delta$ by  inequality (\ref{eq:ineqs-volfr}). This means that equation (\ref{eq:stab-fr}) cannot hold in general, even when the manifolds $M$ and $N$ have the same dimension. 
	\end{example}
	
	We are however able to establish the following:
	\begin{proposition}[Stability of the Filling Radius]\label{prop:stab-same-fund-class}
		Let $M$ be a closed connected $n$-dimensional manifold. Let $d_1$ and $d_2$ be two metrics on $M$ compatible with the manifold topology. Then, 
		$$\big|\filrad(M,d_1)-\filrad(M,d_2)\big|\leq \|d_1-d_2\|_{\infty}.$$
	\end{proposition}
	
	Actually, one can prove a more general result.
	
	\begin{proposition}[Stability of  Generalized Filling Radii]\label{prop:stab-same-fund-class-general}
		Let $M$ be a closed connected manifold. Let $d_1$ and $d_2$ be two metrics on $M$ compatible with the manifold topology. For any integer $k\geq 0$, any abelian group $G$, and any nonzero $\omega\in\mathrm{H}_k(M;G)$,
		$$\big|\filrad_k((M,d_1),G,\omega)-\filrad_k((M,d_2),G,\omega)\big|\leq \|d_1-d_2\|_{\infty}.$$
	\end{proposition}
	
	\begin{Remark}
		Proposition \ref{prop:stab-same-fund-class} is just a special case of Proposition \ref{prop:stab-same-fund-class-general} when $k=n$, $\omega=[M]$, and $G=\mathbb{Z}$ or $\mathbb{Z}_2$.
	\end{Remark}
	
	\begin{proof}[Proof of Proposition \ref{prop:stab-same-fund-class-general}]
		Let $i_1:(M,d_1)\longrightarrow L^\infty(M)$ (resp. $i_2:(M,d_2)\longrightarrow L^\infty(M)$) be the Kuratowski embedding of $M$ into $L^\infty(M)$ with respect to $d_1$ (resp. $d_2$). For arbitrary $r>0$, let $i_1^r:(M,d_1)\longrightarrow B_r(i_1(M),L^\infty(M))$ and $i_2^r:(M,d_2)\longrightarrow B_r(i_2(M),L^\infty(M))$ denote the corresponding isometric embeddings induced from $i_1$ and $i_2$. For arbitrary $r>0$, observe that
		$$B_r(i_1(M),L^\infty(M))\subseteq B_{r+\Vert d_1-d_2 \Vert_\infty}(i_2(M),L^\infty(M))$$
		because of the following argument: for arbitrary $f\in B_r(i_1(M),L^\infty(M))$, there exist $x\in M$ such that $\Vert f- d_1(x,\cdot) \Vert_\infty<r$. Hence,
		\begin{align*}
			\Vert f- d_2(x,\cdot) \Vert_\infty\leq\Vert f- d_1(x,\cdot) \Vert_\infty+\Vert d_1(x,\cdot)- d_2(x,\cdot) \Vert_\infty<r+\Vert d_1-d_2\Vert_\infty.
		\end{align*}
		In the similar way, one can prove that $B_r(i_2(M),L^\infty(M))\subseteq B_{r+\Vert d_1-d_2 \Vert_\infty}(i_1(M),L^\infty(M))$.

		Now, fix arbitrary  $r>\filrad_k((M,d_1),G,\omega)$ and let 
		$$j^r:B_r(i_1(M),L^\infty(M)) \longhookrightarrow B_{r+\Vert d_1-d_2 \Vert_\infty}(i_2(M),L^\infty(M))$$
		be the canonical inclusion map. The maps defined above fit into the following (in general non-commutative diagram):

		\[\begin{tikzcd}
			M \arrow{r}{i_1^r} \arrow[swap]{dr}{i_2^{r+\|d_1-d_2\|_\infty}} & B_r(i_1(M),L^\infty(M)) \arrow{d}{j^r} \\
			& B_{r+\|d_1-d_2\|_\infty}(i_2(M),L^\infty(M))
		\end{tikzcd}
		\]

		Next, we  prove that $j^r\circ i_1^r$ is homotopic to $i_2^{r+\Vert d_1-d_2 \Vert_\infty}$ via the linear interpolation
		\begin{align*}
			H:M\times[0,1]&\longrightarrow B_{r+\Vert d_1-d_2 \Vert_\infty}(i_2(M),L^\infty(M))\\
			(x,t)&\longmapsto (1-t)\,d_1(x,\cdot)+t\,d_2(x,\cdot).
		\end{align*}
		The only subtle point is whether this linear interpolation is always  contained in the thickening $B_{r+\Vert d_1-d_2 \Vert_\infty}(i_2(M),L^\infty(M))$ or not. To ascertain this, for arbitrary $x\in M$ and $t\in[0,1]$, compute the distance between  $H(x,t)$ and $d_2(x,\cdot)$ as:
		\begin{align*}
			\Vert (1-t)\,d_1(x,\cdot)+t\,d_2(x,\cdot)-d_2(x,\cdot) \Vert_\infty&=\vert 1-t \vert\cdot\Vert d_1(x,\cdot)-d_2(x\cdot) \Vert_\infty\\
			&\leq \Vert d_1-d_2 \Vert_\infty<r+\Vert d_1-d_2 \Vert_\infty.
		\end{align*}
		Hence, $H$ is a well-defined homotopy between $j^r\circ i_1^r$ and $i_2^{r+\Vert d_1-d_2 \Vert_\infty}$. Therefore, we have $$(j^r)_*\circ (i_1^r)_*=\big(i_2^{r+\Vert d_1-d_2 \Vert_\infty}\big)_*.$$
		From the assumption on $r$, we know that $(i_1^r)_*(\omega)=0$. By the above, this implies that
		$$\big(i_2^{r+\Vert d_1-d_2 \Vert_\infty}\big)_*(\omega)=0.$$
		Hence,  we have
		$$\filrad_k((M,d_2),G,\omega)\leq\filrad_k((M,d_1),G,\omega)+\Vert d_1-d_2 \Vert_\infty$$
		since $r>\filrad_k((M,d_1),G,\omega)$ is arbitrary. In the similar way, one can also show
		$$\filrad_k((M,d_1),G,\omega)\leq\filrad_k((M,d_2),G,\omega)+\Vert d_1-d_2 \Vert_\infty.$$
		This concludes the proof.
	\end{proof}

	\subsubsection{The strong filling radius}
	Examples \ref{ex:diff-dim} and \ref{ex:same-dim} suggest that the setting of Proposition \ref{prop:stab-same-fund-class} might be  a suitable one for studying stability of the filling radius. 
	
	In this section we consider a certain strong variant of the filling radius satisfying equation (\ref{eq:stab-fr}) which arises from the notion of persistent homology.  
	
	\begin{definition}[Strong filling radius]
		Given a closed connected $n$-dimensional metric manifold $M$ and a field $\mathbb{F}$, we define the \emph{strong filling radius} $\sfilrad(M;\mathbb{F})$ as  half the length of the largest interval in the $n$-th Vietoris-Rips persistence barcode of $M$:
		$$\sfilrad(M;\mathbb{F}):=\frac{1}{2}\max\big\{\mathrm{length}(I),\,I\in\dgmR_n(M;\mathbb{F})\big\}.$$
	\end{definition}
	
	The reader familiar with concepts from applied algebraic topology will have noticed that the definition of strong filling radius of an $n$-dimensional metric manifold  coincides with (one half of) the \emph{maximal persistence} of its associated Vietoris-Rips persistence module. In fact, for each nonnegative integer $k$ one can define the $k$-dimensional version of strong filling radius of any compact metric space $X$.
	
	\begin{definition}[Generalized strong filling radius]
		Given a compact metric space $X$, a field $\mathbb{F}$, and a nonnegative integer $k\geq 0$, we define the \emph{generalized strong filling radius} $\sfilrad_k(X;\mathbb{F})$ as  half the length of the largest interval in the $k$-th Vietoris-Rips persistence barcode of $X$:
		$$\sfilrad_k(X;\mathbb{F}):=\frac{1}{2}\max\big\{\mathrm{length}(I),\,I\in\dgmR_k(X;\mathbb{F})\big\}.$$
	\end{definition}

	\begin{Remark} \begin{itemize}
			\item When $X$ is isometric to a metric manifold $M$ with dimension $n$, we of course have $\sfilrad_n(X) = \sfilrad(M)$. 
			\item In general, $\sfilrad_k$ and $\filrad_k$ are obviously related in the sense that $\sfilrad_k(X;\mathbb{F})\geq \sup\{\filrad_k(X,\mathbb{F},\omega);\,\omega\in \mathrm{H}_k(X;\mathbb{F})\}$ for any nonnegative integer $k$.
		\end{itemize}
	\end{Remark}
	
	The following remark follows directly from Proposition \ref{prop:spread} and Proposition \ref{prop:filradpersistence}.
	\begin{Remark}\label{rem:ineqfilsfil}
		$\filrad(M;\mathbb{F}) \leq \sfilrad(M;\mathbb{F}) \leq \frac{1}{2}\spread(M)$ for any field $\mathbb{F}$ when $M$ is orientable, and $\mathbb{F}=\mathbb{Z}_2$ when $M$ is non-orientable.
	\end{Remark}
	
	\begin{definition}[$\mathbb{F}$-regularly filled manifold]
		Let $(M,d_M)$ be a closed connected metric manifold and $\mathbb{F}$ be a field. We say that $M$ is $\mathbb{F}$-\emph{regularly filled} if $\filrad(M;\mathbb{F})=\sfilrad(M;\mathbb{F})$.
	\end{definition}
	
	\begin{Remark}
		For each $n \geq 1$, the $n$-dimensional unit sphere with the intrinsic metric is $\mathbb{F}$-regularly filled for any field $\mathbb{F}$. 
		Indeed, by \cite[Proof of Theorem 2]{k83}, $\filrad(\mathbb{S}^n)=\frac{1}{2}\spread(\mathbb{S}^n)$. Hence, the result follows from Remark \ref{rem:ineqfilsfil}.
	\end{Remark}
	
	As a consequence of the remark above and Remark \ref{rem:spread-sn} we have
	
	\begin{corollary} For all integers $n\geq 1$,
		$\filrad(\mathbb{S}^n) = \sfilrad(\mathbb{S}^n;\mathbb{F}) = \frac{1}{2}\arccos\left(\frac{-1}{n+1}\right).$
	\end{corollary}
	
	There exist, however, non-regularly filled metric manifolds. We present two examples: the first one arises from our study of the K\"unneth formula in Section \ref{sec:vrapp} whereas the second one is a direct construction. Both examples make use of results from \cite{adams} about homotopy types of Vietoris-Rips complexes of $\mathbb{S}^1$.
	
	\begin{example}[A non-regularly filled metric manifold]
		Let $r>1$ and $X$ be the $\ell^\infty$-product $ \mathbb{S}^1 \times \mathbb{S}^1 \times (r \cdot \mathbb{S}^1)$. By Remark \ref{rem:productfilrad}, $\filrad(X)=\filrad(\mathbb{S}^1)=\frac{2 \pi}{3}.$ By Example \ref{example:torusbarcode}, $\dgmR_3(X;\mathbb{F})$ contains the interval $(\frac{2\pi r}{3}, \frac{4\pi r}{5}],$ which has length $\frac{2\pi r}{15}$. Hence, if $r>5$,  $X$ is not $\mathbb{F}$-regularly filled.
	\end{example}
	
	\begin{example}[A non-regularly filled Riemannian manifold]
		Take any  embedding of $\mathbb{S}^1$ into $\mathbb{R}^4$ and let $\epsilon>0$ be small. Consider the boundary $C_\epsilon$ of the $\epsilon$- tubular neighborhood around $\mathbb{S}^1$. This will be a $3$-dimensional submanifold of $\mathbb{R}^4$. As a submanifold it inherits the ambient inner product and $C_\epsilon$ can be regarded as a Riemannian manifold in itself. Then, as a metric space, with the geodesic distance, $C_\epsilon$ will be $\epsilon$-close to $\mathbb{S}^1$ (with geodesic distance) in Gromov-Hausdorff sense. Because we know that for   $r\in  (\frac{2\pi}{3}, \frac{4\pi}{5}]$, $\vr_r(\mathbb{S}^1)\simeq \mathbb{S}^3$, and because of the Gromov-Hausdorff stability of barcodes (cf. Theorem \ref{thm:stab-barcodes}), it must be that $\dgmR_3(C_\epsilon;\mathbb{F})$ contains an interval $I$ which itself contains $(\frac{2\pi}{3}+\epsilon, \frac{4\pi}{5}-\epsilon)$. This latter interval is non-empty whenever $\epsilon>0$ is small enough so that $\sfilrad(C_\epsilon;\mathbb{F})\approx \frac{2\pi}{15} -2\epsilon$. However, $\filrad(C_\epsilon)\approx \epsilon$.
	\end{example}
	
	By invoking the relationship between the Vietoris-Rips persistent homology and the strong filling radius, one can verify that the strong filling radii of two $n$-dimensional metric manifolds $M$ and $N$ are close if these two manifolds are similar in the Gromov-Hausdorff distance sense. 
	
	\begin{proposition}\label{prop:stability}
		Let $X$ and $Y$ be compact metric spaces . Then, for any integer $k\geq 0$,
		$$|\sfilrad_k(X;\mathbb{F})-\sfilrad_k(Y;\mathbb{F})| \leq 2 \, \dgh(X,Y).$$
	\end{proposition}
	\begin{proof}
		By Remark \ref{rem:stab-gh} one has
		\begin{align*}
			2\dgh(X,Y)&\geq d_\mathrm{I}\big(\dgmR_k(X;\mathbb{F}),\dgmR_k(Y;\mathbb{F})\big)\\
			&\geq\left| d_\mathrm{I}\big(\dgmR_k(X;\mathbb{F}),0_*\big)- d_\mathrm{I}\big(\dgmR_k(Y;\mathbb{F}),0_*\big)\right|,
		\end{align*}
		where the last inequality follows from the triangle inequality for the interleaving distance. The conclusion now follows from Example \ref{ex:int-zero}. 
	\end{proof}
	
	\begin{Remark}
		Albeit for the notation $\sfilrad_k$, the above stability result should be well known to readers familiar with applied algebraic topology concepts -- we state and prove it here however to provide some background for those readers who are not.
	\end{Remark}


	\appendix
	\section{Appendix}
	\subsection{Proof of Proposition \ref{prop:Katzgeoprop}}\label{app:proof-propo-katz}
	
	\restatepropKatzgeoprop*
	
	\begin{proof}
		\begin{enumerate}
			\item[(1)] The first claim  trivially follows from the definition of $\gamma_K$ and that of the $\ell^\infty$-norm.
			
			\item[(2)] For the second claim, observe that it is enough to show
			$$\Vert\gamma_K(f,g,s)-\gamma_K(f,g,t)\Vert_\infty\leq (t-s)\cdot \Vert f-g \Vert_\infty$$
			for any $f,g\in L^\infty(X)$ and $0\leq s\leq t\leq 1$.

			Fix arbitrary $x\in X$. Without loss of generality, one can assume that $f(x)\geq g(x)$. Then,
			$$\vert \gamma_K(f,g,s)(x)-\gamma_K(f,g,t)(x) \vert=\max\{f(x)-s\Vert f-g \Vert_\infty,g(x)\}-\max\{f(x)-t\Vert f-g \Vert_\infty,g(x)\}.$$
			Observe that, if $s\in \left[0,\frac{f(x)-g(x)}{\Vert f-g \Vert_\infty}\right]$,
			$$\max\{f(x)-s\Vert f-g \Vert_\infty,g(x)\}=f(x)-s\Vert f-g \Vert_\infty.$$
			Hence,
			\begin{align*}
				\vert \gamma_K(f,g,s)(x)-\gamma_K(f,g,t)(x) \vert&=(f(x)-s\Vert f-g \Vert_\infty)-\max\{f(x)-t\Vert f-g \Vert_\infty,g(x)\}\\
				&\leq (f(x)-s\Vert f-g \Vert_\infty)-(f(x)-t\Vert f-g \Vert_\infty)\\
				&=(t-s)\Vert f-g \Vert_\infty.
			\end{align*}
			
			Also, if $s\in \left[\frac{f(x)-g(x)}{\Vert f-g \Vert_\infty},1\right]$,
			
			$$\max\{f(x)-s\Vert f-g \Vert_\infty,g(x)\}=\max\{f(x)-t\Vert f-g \Vert_\infty,g(x)\}=g(x)$$
			
			so that $\vert \gamma_K(f,g,s)(x)-\gamma_K(f,g,t)(x) \vert=0\leq (t-s)\Vert f-g \Vert_\infty.$
			
			Since $x$ is arbitrary, we obtain
			$\Vert\gamma_K(f,g,s)-\gamma_K(f,g,t)\Vert_\infty\leq (t-s)\cdot \Vert f-g \Vert_\infty$.
			
			\item[(3)] Fix arbitrary $x\in X$. We will prove that
			$$\vert\gamma_K(f,g,t)(x)-\gamma_K(h,g,t)(x)\vert\leq 2\Vert f-h\Vert_\infty.$$
			Unfortunately, we have to do tedious case-by-case analysis.
			
			\begin{enumerate}
				\item If $f(x)\geq g(x)$ and $h(x)\geq g(x)$: Observe that, for $t\in\left[0,\frac{f(x)-g(x)}{\Vert f-g\Vert_\infty}\right]$,
				$$\gamma_K(f,g,t)(x)=f(x)-t\Vert f-g\Vert_\infty.$$
				Hence,
				\begin{align*}
					\gamma_K(f,g,t)(x)-\gamma_K(h,g,t)(x)&=(f(x)-t\Vert f-g\Vert_\infty)-\max\{h(x)-t\Vert h-g\Vert_\infty,g(x)\}\\
					&\leq (f(x)-t\Vert f-g\Vert_\infty)-(h(x)-t\Vert h-g\Vert_\infty)\\
					&=f(x)-h(x)-t(\Vert f-g\Vert_\infty-\Vert h-g\Vert_\infty)\\
					&\leq \vert f(x)-h(x) \vert+t\vert\Vert f-g\Vert_\infty-\Vert h-g\Vert_\infty\vert\\
					&\leq 2\Vert f-h\Vert_\infty.
				\end{align*}
				
				Now, for $t\in\left[\frac{f(x)-g(x)}{\Vert f-g\Vert_\infty},1\right]$, $\gamma_K(f,g,t)(x)=g(x)$. Hence,
				\begin{align*}
					\gamma_K(f,g,t)(x)-\gamma_K(h,g,t)(x)&=g(x)-\max\{h(x)-t\Vert h-g\Vert_\infty,g(x)\}\\
					&\leq g(x)-g(x)=0\leq  2\Vert f-h\Vert_\infty.
				\end{align*}
				In the similar way, one can also obtain
				$$\gamma_K(h,g,t)(x)-\gamma_K(f,g,t)(x)\leq 2\Vert f-h\Vert_\infty$$
				for any $t\in[0,1]$. Hence,
				$$\vert\gamma_K(f,g,t)(x)-\gamma_K(h,g,t)(x)\vert\leq 2\Vert f-h\Vert_\infty,$$
				as we wanted.
				
				\item If $f(x)\leq g(x)$ and $h(x)\leq g(x)$: This case is similar to the previous one so we omit it.
				
				\item If $f(x)\geq g(x)$ and $h(x)\leq g(x)$: Observe that
				$$\gamma_K(f,g,t)(x),\gamma_K(h,g,t)(x)\in[h(x),f(x)].$$
				Therefore,
				\begin{align*}
					\vert\gamma_K(f,g,t)(x)-\gamma_K(h,g,t)(x)\vert&\leq f(x)-h(x)\\
					&\leq\Vert f-h\Vert_\infty\leq 2\Vert f-h\Vert_\infty.
				\end{align*}
				
				\item If $f(x)\leq g(x)$ and $h(x)\geq g(x)$: Similar to the previous case.
			\end{enumerate}
			Since $x$ is arbitrary, we finally have
			$$\Vert\gamma_K(f,g,t)-\gamma_K(h,g,t)\Vert_\infty\leq 2\Vert f-h\Vert_\infty.$$
			
			\item[(4)] Fix arbitrary $x\in X$. We will prove that
			$$\vert\gamma_K(f,g,t)(x)-\gamma_K(f,h,t)(x)\vert\leq \Vert g-h\Vert_\infty.$$
			Let's do case-by-case analysis.
			
			\begin{enumerate}
				\item If $f(x)\geq g(x)$ and $f(x)\geq h(x)$: Observe that, for $t\in\left[0,\frac{f(x)-g(x)}{\Vert f-g\Vert_\infty}\right]$,
				$$\gamma_K(f,g,t)(x)=f(x)-t\Vert f-g\Vert_\infty.$$
				Hence,
				\begin{align*}
					\gamma_K(f,g,t)(x)-\gamma_K(f,h,t)(x)&=(f(x)-t\Vert f-g\Vert_\infty)-\max\{f(x)-t\Vert f-h\Vert_\infty,h(x)\}\\
					&\leq (f(x)-t\Vert f-g\Vert_\infty)-(f(x)-t\Vert f-h\Vert_\infty)\\
					&=t(\Vert f-h\Vert_\infty-\Vert f-g\Vert_\infty)\\
					&\leq \Vert g-h\Vert_\infty.
				\end{align*}
				
				Now, for $t\in\left[\frac{f(x)-g(x)}{\Vert f-g\Vert_\infty},1\right]$, $\gamma_K(f,g,t)(x)=g(x)$. Hence,
				\begin{align*}
					\gamma_K(f,g,t)(x)-\gamma_K(f,h,t)(x)&=g(x)-\max\{f(x)-t\Vert f-h\Vert_\infty,h(x)\}\\
					&\leq g(x)-h(x)\leq\Vert g-h\Vert_\infty.
				\end{align*}
				
				In the similar way, one can also obtain
				$$\gamma_K(f,h,t)(x)-\gamma_K(f,g,t)(x)\leq \Vert g-h\Vert_\infty$$
				for any $t\in[0,1]$. Hence,
				$$\vert\gamma_K(f,g,t)(x)-\gamma_K(f,h,t)(x)\vert\leq \Vert g-h\Vert_\infty.$$
				as we wanted.
				
				\item If $f(x)\leq g(x)$ and $f(x)\leq h(x)$: Similar to the previous case.
				
				\item If $f(x)\geq g(x)$ and $f(x)\leq h(x)$: Observe that
				$$\gamma_K(f,g,t)(x),\gamma_K(f,h,t)(x)\in[g(x),h(x)].$$
				Therefore,
				\begin{align*}
					\vert\gamma_K(f,g,t)(x)-\gamma_K(f,h,t)(x)\vert&\leq h(x)-g(x)\\
					&\leq\Vert g-h\Vert_\infty.
				\end{align*}
				
				\item If $f(x)\leq g(x)$ and $f(x)\geq h(x)$: Similar to the previous case.
			\end{enumerate}
			Since $x$ is arbitrary, we finally have
			$$\Vert\gamma_K(f,g,t)-\gamma_K(f,h,t)\Vert_\infty\leq \Vert g-h \Vert_\infty.$$
			
			\item[(5)] Fix arbitrary $x\in X$. Suppose $f(x)\geq g(x)$. Then,
			$$\phi(x)=\max\{f(x)-s\Vert f-g\Vert_\infty,g(x)\}$$
			and
			$$\psi(x)=\max\{f(x)-t\Vert f-g\Vert_\infty,g(x)\}.$$
			By  property (1) of this proposition, we know $\Vert \phi-\psi\Vert_\infty=(t-s)\Vert f-g \Vert_\infty$. Moreover, since $\phi(x)\geq\psi(x)$, we have
			$$\gamma_K(\phi,\psi,\lambda)(x)=\max\{\phi(x)-\lambda\Vert\phi-\psi\Vert_\infty,\psi(x)\}.$$
			Observe that,
			\begin{align*}
				\phi(x)-\lambda\Vert\phi-\psi\Vert_\infty&=\max\{f(x)-s\Vert f-g\Vert_\infty,g(x)\}-\lambda(t-s)\Vert f-g \Vert_\infty\\
				&=\max\{f(x)-((1-\lambda)s+\lambda t)\Vert f-g\Vert_\infty,g(x)-\lambda(t-s)\Vert f-g \Vert_\infty\}.
			\end{align*}
			Since $f(x)-((1-\lambda)s+\lambda t)\Vert f-g\Vert_\infty\geq f(x)-t\Vert f-g\Vert_\infty$ and $g(x)\geq g(x)-\lambda(t-s)\Vert f-g \Vert_\infty$, we finally have
			\begin{align*}
				\gamma_K(\phi,\psi,\lambda)(x)&=\max\{f(x)-((1-\lambda)s+\lambda t)\Vert f-g\Vert_\infty,g(x)\}\\
				&=\gamma_K(f,g,(1-\lambda)s+\lambda t)(x).
			\end{align*}
			One can do the similar proof for the case when $f(x)\leq g(x)$. Hence, we have
			$$\gamma_K(\phi,\psi,\lambda)=\gamma_K(f,g,(1-\lambda)s+\lambda t).$$
			
			\item[(6)] Consider the special case $s=0$ and $t=1$. Fix arbitrary $x\in X$. Observe that $\gamma_K(f,g,r)(x)$ is between $f(x)$ and $g(x)$. Therefore,
			$$\vert \gamma_K(f,g,r)(x)-h(x) \vert\leq\max\{\vert f(x)-h(x)\vert,\vert g(x)-h(x)\vert\}\leq\max\{\Vert f-h \Vert_\infty,\Vert g-h\Vert_\infty\}.$$
			Since $x$ is arbitrary, we have
			$$\Vert \gamma_K(f,g,r)-h \Vert_\infty\leq\max\{\Vert f-h \Vert_\infty,\Vert g-h\Vert_\infty\}=\max\{\Vert \gamma_K(f,g,0)-h \Vert_\infty,\Vert \gamma_K(f,g,1)-h\Vert_\infty\}.$$
			For general $s$ and $t$, combine this result with   property (5).
		\end{enumerate}
	\end{proof}
	
	\subsection{Some properties of $\gamma_K$}\label{sec:app:katz-bicombing}
	\begin{example}
		In this example, we will see that some of the nice properties considered in Lemma \ref{lemma:geobicomb} do not hold for the Katz geodesic bicombing (see Definition \ref{def:Katzgeobicomb}). Consider $X$ to be a two point space. Then $L^\infty(X)$ can be regarded as $\mathbb{R}^2$ with the $\ell^\infty$-norm.
		\begin{enumerate}
			\item \emph{Katz's geodesic bicombing is not conical in general:} We will find $f,f',g,g'\in L^\infty(X)$ and $t\in [0,1]$ such that
			$$\Vert \gamma_K(f,g,t)-\gamma_K(f',g',t)\Vert_\infty > (1-t)\Vert f-f'\Vert_\infty+t\Vert g-g'\Vert_\infty.$$
			Let $f=f'=(0,0)$, $g=(c,d)$ for some $0<c<d$, and $g=(c',d')$ for some $0<c'<d'$. Then,
			$$\gamma_K(0,g,t)=\begin{cases}(t\,d,t\,d)&\text{if }t\in[0,\frac{c}{d}]\\(c,t\,d)&\text{if }t\in[\frac{c}{d},1]\end{cases}$$
			and we have the similar expression for $\gamma_K(0,g',t)$. Hence, for any $t\in \left[\max\{\frac{c}{d},\frac{c'}{d'}\},1\right)$,
			$$\gamma_K(0,g,t)=(c,t\,d)$$
			and
			$$\gamma_K(0,g',t)=(c',t\,d').$$
			Therefore, if we choose $\vert c-c' \vert> \vert d-d' \vert$ (for example, $(c,d)=(4,5)$ and $(c',d')=(1,5)$), then we have
			$$\Vert \gamma_K(0,g,t)-\gamma_K(0,g',t) \Vert_\infty=\vert c-c' \vert$$
			and
			$$t\Vert g-g' \Vert_\infty=t\vert c-c'\vert.$$
			Hence, we have
			$$\Vert \gamma_K(0,g,t)-\gamma_K(0,g',t) \Vert_\infty>t\Vert g-g' \Vert_\infty.$$
			So, Katz geodesic bicombing is not conical. In particular, this implies it is not convex.
			
			\item \emph{Katz geodesic bicombing is not reversible in general:} We will find $f,g\in L^\infty(X)$ and $t\in [0,1]$ such that
			$$\gamma_K(f,g,t)\neq\gamma_K(g,f,1-t).$$
			Let $f=(0,0)$ and $g=(c,d)$ for some $0<c<d$ as before. Then,
			$$\gamma_K(0,g,t)=\begin{cases}(td,td)&\text{if }t\in[0,\frac{c}{d}]\\(c,td)&\text{if }t\in[\frac{c}{d},1]\end{cases}$$
			and
			$$\gamma_K(g,0,t)=\begin{cases}(c-td,(1-t)d)&\text{if }t\in[0,\frac{c}{d}]\\(0,(1-t)d)&\text{if }t\in[\frac{c}{d},1]\end{cases}$$
			Now, if we choose $t\in\left(0,\min\{\frac{c}{d},1-\frac{c}{d}\}\right]$, we have
			$\gamma_K(0,g,t)=(td,td)$ and $\gamma_K(g,0,1-t)=(0,td)$. Hence,
			$$\gamma_K(0,g,t)\neq\gamma_K(g,0,1-t).$$
			
		\end{enumerate}
	\end{example}
	
	\subsection{Proof of the generalized functorial nerve lemma}\label{sec:app:functnerve}
	The goal of this appendix is proving the following \emph{Generalized Functorial Nerve Lemma}.
	
	\functnerve*
	
	Our proof of  Theorem \ref{thm:fuctnerve} invokes many elements of \cite[Section 4.G]{h01} which provides a proof of the classical nerve lemma.
	
	\begin{definition}
		Let $X$ be a topological space and let $\mathcal{U}=\{U_\alpha\}_{\alpha\in\Lambda}$ be an open covering of $X$ ($\Lambda$ is an arbitrary index set). For any $\sigma=\{\alpha_0,\dots,\alpha_n\}\in\mathrm{N}\,\mathcal{U}$, the nonempty intersection $U_{\alpha_0}\cap\dots \cap\,U_{\alpha_n}$ is denoted by $U_\sigma$. Note that, when $\sigma=\{\alpha_0,\dots,\alpha_n\}\in\mathrm{N}\,\mathcal{U}$ and $\sigma'$ is an $n'$-face of $\sigma$, there are the canonical inclusions
		$$i_{\sigma\sigma'}:U_\sigma\longhookrightarrow U_{\sigma'}$$
		and
		$$j_{\sigma\sigma'}:\Delta_{n'}\longhookrightarrow \Delta_n.$$

		Then, the \emph{complex of spaces} corresponding to $\mathcal{U}$ consists of the set of all $U_\sigma$ and the set of all canonical inclusions $i_{\sigma\sigma'}$ over all possible $\sigma'\subseteq\sigma\in\mathrm{N}\,\mathcal{U}$.

		The \emph{realization} of this complex of spaces, denoted by $\Delta X_\mathcal{U}$, is defined in the following way:
		$$\Delta X_\mathcal{U}:=\bigsqcup_{\sigma=\{\alpha_0,\dots,\alpha_n\}\in\mathrm{N}\,\mathcal{U}} U_\sigma\times\Delta_n\slash\sim$$
		where $(x,p)\sim (x',p')$ whenever $i_{\sigma\sigma'}(x)=x'$ and $j_{\sigma\sigma'}(p')=p$.
	\end{definition}
	
	We need the following slight improvements of Propositions 4G.1 and 4G.2 of \cite{h01}. These improved claims are actually implicit in their respective proofs -- see \cite[pp. 458-459]{h01}.
	
	\begin{proposition}[Proposition 4G.1 of \cite{h01}] \label{prop:4g1}
		Let $X$ be a topological space and $\mathcal{U}=\{U_\alpha\}_{\alpha\in\Lambda}$ be a good open cover of $X$ (every nonempty finite intersection is contractible). Then,
		\begin{align*}
			f:\Delta X_\mathcal{U}&\longrightarrow \mathrm{N}\,\mathcal{U}\\
			(x,p)&\longmapsto p\,\,\text{if}\,\,(x,p)\in U_\sigma\times\Delta_n
		\end{align*}
		is a homotopy equivalence between $\Delta X_\mathcal{U}$ and $\mathrm{N}\,\mathcal{U}$.
	\end{proposition}
	\begin{proof}
		First of all, since $U_\sigma$ is contractible whenever $\sigma\in\mathrm{N}\,\mathcal{U}$, note that there is a homotopy equivalence $\phi_\sigma:U_\sigma\rightarrow \{*\}$ for any $\sigma\in\mathrm{N}\,\mathcal{U}$.

		The homotopy equivalence between $\Delta X_\mathcal{U}$ and $\mathrm{N}\,\mathcal{U}$ is just a special case of \cite[Proposition 4G.1]{h01}. The choice of $f$ is implicit in the fact that both of $\Delta X_\mathcal{U}$ and $\mathrm{N}\,\mathcal{U}$ are deformation retracts of $\Delta MX_\mathcal{U}$ where $\Delta MX_\mathcal{U}$ is the realization of the complex of spaces consisting of the mapping cylinders $M\phi_\sigma$ for any $\sigma\in\mathrm{N}\,\mathcal{U}$ and the canonical inclusions between them.
	\end{proof}

	\begin{proposition}[Proposition 4G.2 of \cite{h01}] \label{prop:4g2}
		Let $X$ be a paracompact space, $\mathcal{U}=\{U_\alpha\}_{\alpha\in\Lambda}$ be an open cover of $X$, and $\{\psi_\alpha\}_{\alpha\in\Lambda}$ be a partition of unity subordinate to the cover $\mathcal{U}$ (it must exist since $X$ is paracompact).

		Then,
		\begin{align*}
			g:X&\longrightarrow\Delta X_\mathcal{U}\\
			x&\longmapsto (x,(\psi_\alpha(x))_{\alpha\in\Lambda})
		\end{align*}
		is a homotopy equivalence between $X$ and $\Delta X_\mathcal{U}$.
	\end{proposition}
	\begin{proof}
		The proof is the same as \cite[Proposition 4G.2]{h01}.
	\end{proof}

	\begin{lemma}\label{lemma:realizationnervecommute}
		Let $X$ and $Y$ be two topological spaces, $\rho:X\longrightarrow Y$ be a continuous map, $\mathcal{U}=\{U_\alpha\}_{\alpha\in A}$ and $\mathcal{V}=\{V_\beta\}_{\beta\in B}$ be good open covers (every nonempty finite intersection is contractible) of $X$ and $Y$ respectively, based on arbitrary index sets $A$ and $B$, and $\pi:A\longrightarrow B$ be a map such that
		$$\rho(U_\alpha)\subseteq V_{\pi(\alpha)}$$
		for any $\alpha\in A$.

		Let $\mathrm{N}\,\mathcal{U}$ and $\mathrm{N}\,\mathcal{V}$ be the nerves of $\mathcal{U}$ and $\mathcal{V}$, respectively. Observe that $\pi$ induces the canonical simplicial map $\bar{\pi}:\mathrm{N}\,\mathcal{U}\longrightarrow\mathrm{N}\,\mathcal{V}$ since $U_{\alpha_0}\cap\dots \cap\,U_{\alpha_n}\neq\emptyset$ implies $V_{\pi(\alpha_0)}\cap\dots \cap\,V_{\pi(\alpha_n)}\neq\emptyset$, and $\rho$ induces the canonical map $\bar{\rho}:\Delta X_\mathcal{U}\longrightarrow\Delta Y_\mathcal{V}$ mapping $(x,p)$ to $(\rho(x),\bar{\pi}(p))$.

		Then, there exist homotopy equivalences $f:\Delta X_\mathcal{U}\longrightarrow \mathrm{N}\,\mathcal{U}$ and $f':\Delta Y_\mathcal{V}\longrightarrow \mathrm{N}\,\mathcal{V}$ which commute with $\bar{\rho}$ and $\bar{\pi}$:
		$$\begin{tikzcd}
			\Delta X_\mathcal{U} \arrow[r, "f"] \arrow[d, "\bar{\rho}"] & \mathrm{N}\,\mathcal{U}\arrow[d, "\bar{\pi}"] \\
			\Delta Y_{\mathcal{V}} \arrow[r, "f'"] & \mathrm{N}\,\mathcal{V}
		\end{tikzcd}$$
	\end{lemma}
	\begin{proof}
		By Proposition \ref{prop:4g1},
		\begin{align*}
			f:\Delta X_\mathcal{U}&\longrightarrow \mathrm{N}\,\mathcal{U}\\
			(x,p)&\longmapsto p\,\text{if}\,(x,p)\in U_\sigma\times\Delta_n
		\end{align*}
		is a homotopy equivalence between $\Delta X_\mathcal{U}$ and $\mathrm{N}\,\mathcal{U}$. Also,
		\begin{align*}
			f':\Delta Y_\mathcal{V}&\longrightarrow \mathrm{N}\,\mathcal{V}\\
			(y,q)&\longmapsto q\,\text{if}\,(y,q)\in V_\sigma\times\Delta_n
		\end{align*}
		is a homotopy equivalence between $\Delta Y_\mathcal{V}$ and $\mathrm{N}\,\mathcal{V}$.

		To check the commutativity of the diagram, fix arbitrary $(x,p)\in U_\sigma\times\Delta_n\subseteq\Delta X_\mathcal{U}$. Then,
		\begin{align*}
			\bar{\pi}\circ f(x,p)=\bar{\pi}(p)=f'(\rho(x),\bar{\pi}(p))=f'\circ \bar{\rho}(x,p).
		\end{align*}
		Hence, $\bar{\pi}\circ f=f'\circ \bar{\rho}$ as we wanted.
	\end{proof}

	\begin{lemma}\label{lemma:Xrealizationcommute}
		Let $X$ and $Y$ be two paracompact spaces, $\rho:X\longrightarrow Y$ be a continuous map, $\mathcal{U}=\{U_\alpha\}_{\alpha\in A}$ and $\mathcal{V}=\{V_\beta\}_{\beta\in B}$ be open covers of $X$ and $Y$ respectively, based on arbitrary index sets $A$ and $B$, and $\pi:A\longrightarrow B$ be a map such that
		$$\rho(U_\alpha)\subseteq V_{\pi(\alpha)}$$
		for any $\alpha\in A$.

		Let $\mathrm{N}\,\mathcal{U}$ and $\mathrm{N}\,\mathcal{V}$ be the nerves of $\mathcal{U}$ and $\mathcal{V}$, respectively. Observe that $\pi$ induces the canonical simplicial map $\bar{\pi}:\mathrm{N}\,\mathcal{U}\longrightarrow\mathrm{N}\,\mathcal{V}$ since $U_{\alpha_0}\cap\dots \cap\,U_{\alpha_n}\neq\emptyset$ implies $V_{\pi(\alpha_0)}\cap\dots \cap\,V_{\pi(\alpha_n)}\neq\emptyset$, and $\rho$ induces the canonical map $\bar{\rho}:\Delta X_\mathcal{U}\longrightarrow\Delta Y_\mathcal{V}$ mapping $(x,p)$ to $(\rho(x),\bar{\pi}(p))$.

		Then, there exist homotopy equivalences $g:X\longrightarrow\Delta X_\mathcal{U}$ and $g':Y\longrightarrow\Delta Y_\mathcal{V}$ which commute with $\rho$ and $\bar{\rho}$ up to homotopy:
		$$\begin{tikzcd}
			X \arrow[r, "g"] \arrow[d, "\rho"] & \Delta X_\mathcal{U}\arrow[d, "\bar{\rho}"] \\
			Y \arrow[r, "g'"] & \Delta Y_\mathcal{V}
		\end{tikzcd}$$
	\end{lemma}
	\begin{proof}
		By Proposition \ref{prop:4g2}
		\begin{align*}
			g:X&\longrightarrow\Delta X_\mathcal{U}\\
			x&\longmapsto (x,(\psi_\alpha(x))_{\alpha\in\Lambda})
		\end{align*}
		is a homotopy equivalence between $X$ and $\Delta X_\mathcal{U}$, where $\{\psi_\alpha\}_{\alpha\in A}$ is a partition of unity subordinate to the cover $\mathcal{U}$. And,
		\begin{align*}
			g':Y&\longrightarrow\Delta Y_\mathcal{V}\\
			y&\longmapsto (y,(\psi_\beta'(y))_{\beta\in B})
		\end{align*}
		is a homotopy equivalence between $Y$ and $\Delta Y_\mathcal{V}$ where $\{\psi_\beta'\}_{\beta\in B}$ is a partition of unity subordinate to the cover $\mathcal{V}$.

		Finally, we will show that $\bar{\rho}\circ g\simeq g'\circ \rho$. Observe that, for arbitrary $x\in X$,
		
		$$\bar{\rho}\circ g(x)=\bar{\rho}(x,(\psi_\alpha(x))_{\alpha\in\Lambda})=\bigg(\rho(x),\bar{\pi}\big((\psi_\alpha(x))_{\alpha\in A}\big)\bigg)$$
		and
		$$g'\circ \rho(x)=g'(\rho(x))=\bigg(\rho(x),\big(\psi_\beta'(\rho(x))\big)_{\beta\in B}\bigg).$$
		Hence, one can just construct a homotopy between $\bar{\rho}\circ g$ and $g'\circ \rho$ in the following way:
		\begin{align*}
			h:X\times [0,1]&\longrightarrow\Delta Y_\mathcal{V}\\
			(x,t)&\longmapsto \bigg(\rho(x),(1-t)\,\bar{\pi}\big((\psi_\alpha(x))_{\alpha\in A}\big)+t\,\big(\psi_\beta'(\rho(x))\big)_{\beta\in B}\bigg)
		\end{align*}
		Here, note that the linear interpolation between $\bar{\pi}\big((\psi_\alpha(x))_{\alpha\in A}\big)$ and $\big(\psi_\beta'(\rho(x))\big)_{\beta\in B}$ is well-defined since, because of the properties of partition of unity and the assumption that $\rho(U_\alpha)\subseteq V_{\pi(\alpha)}$,
		$$\rho(x)\in \bigcap_{\alpha:\psi_\alpha(x)>0}V_{\pi(\alpha)}\cap\bigcap_{\beta:\psi_\beta'(\rho(x))>0}V_\beta$$
		so that
		$$\{\pi(\alpha)\in B:\psi_\alpha(x)>0\}\cup\{\beta\in B:\psi_\beta'(\rho(x))>0\}$$
		forms a simplex in $\mathrm{N}\,\mathcal{V}$.
	\end{proof}
	
	Finally, one can prove the functorial nerve lemma.
	
	\begin{proof}[Proof of Theorem \ref{thm:fuctnerve}]
		Combine Lemma \ref{lemma:realizationnervecommute} and Lemma \ref{lemma:Xrealizationcommute}.
	\end{proof}

	\subsection{Proof of $\vr_r(\Sp^n)\simeq\Sp^n$ for $r\in \left(0,\arccos\left(-\frac{1}{n+1}\right)\right]$}\label{sec:app:SnHausmann}
	
	In this appendix we will prove Theorem \ref{cor:Snfirsttype}.
	
	\thmSnfirsttype*
	
	Since the case of $\Sp^1$ is already proved in \cite{adams}, it is enough to prove the above theorem for $\Sp^n$ with $n\geq 2$. Moreover, unlike the other parts of the paper, in this section we discriminate between the simplicial complex $\vr_r(\Sp^n)$ and its realization $\vert\vr_r(\Sp^n)\vert$.

	To prove Theorem \ref{cor:Snfirsttype}, we will basically emulate the proof strategy which Hausmann  in \cite{h95}. However, a crucual modification will be necessary which requires  the following version of Jung's theorem.
	
	\begin{definition}\label{def:gch}
		Given a nonempty subset $A\subset \Sp^n$, its \emph{geodesic convex hull $\mathrm{conv}_{\Sp^n}(A)$}  is defined to be the set consisting of the union of all minimizing geodesics between pairs of points in $A$. It is clear that when $A$ is contained in an open hemisphere,  $\mathrm{conv}_{\Sp^n}(A) = \{\Pi_{\Sp^n}(c)|\,c\in \mathrm{conv}(A)\}$ where $\Pi_{\Sp^n}(p) := \frac{p}{\|p\|}$ for $p\neq 0$ and $\Pi_{\Sp^n}(p):=0$ otherwise.
	\end{definition}
	
	\begin{theorem}[A version of Jung's theorem for spheres]\label{thm:Jung}
		For any $n\geq 1$, if $A\subset\Sp^n$ satisfies $D:=\diam(A)<\arccos\left(-\frac{1}{n+1}\right)$, then there must be $u\in\mathrm{Conv}_{\Sp^n}(A)$ such that $A\subseteq \overline{B}_{\psi(D)}(u,\Sp^n)$, where
		\begin{align*}
			\psi:\left[0,\arccos\left(-\frac{1}{n+1}\right)\right]&\longrightarrow\R_{\geq 0}\\
			D&\longmapsto\arccos\left(\sqrt{\frac{1+(n+1)\cos D}{n+2}}\right).
		\end{align*}
	\end{theorem}
	
	The version of Jung's theorem stated above is different from the one considered by Katz \cite[Lemma 2]{k83} in the following two senses: (1) We provide a precise formula for the radius $\psi(D)$ of the closed ball covering $A$, depending on $D=\diam(A)$. In particular, our version is stronger when $D$ is small. (2) On the contrary, if $D$ is large (close to $\arccos\left(-\frac{1}{n+1}\right)$), then the radius $\psi(D)$ can be as large as $\frac{\pi}{2}$. But $\frac{\pi}{2}$ is strictly greater number than $\pi-\arccos\left(-\frac{1}{n+1}\right)$ which is the radius guaranteed by Katz's version. So, for the case when $D$ is large, Katz's version is stronger.

	The proof of our version is somewhat similar to the classical proof in \cite{federer2014geometric}.
	
	\begin{Remark}
		Note that the map $\phi$ satisfies the following properties: 
		\begin{enumerate}
			\item $\psi(D)\leq\frac{\pi}{2}$ for any $D\in\left[0,\arccos\left(-\frac{1}{n+1}\right)\right]$.
			
			\item $\psi$ is an increasing function.
			
			\item $\lim\limits_{D\rightarrow 0+}\psi(D)=0$.
		\end{enumerate}
	\end{Remark}
	
	\begin{proof}[Proof of Theorem \ref{thm:Jung}]
		Without loss of generality, one can assume $A$ is compact. Recall that one can view $\Sp^n$ as a subset of $\R^{n+1}$ in the following way:
		$$\Sp^n=\{(x_1,\dots,x_{n+1})\in\R^{n+1}:x_1^2+\cdots x_{n+1}^2=1\}.$$
		Also, for any $x,y\in\Sp^n$, the Euclidean norm $\Vert x-y\Vert$ and the geodesic distance $d_{\Sp^n}(x,y)$ satisfy the following relationship:
		$$\Vert x-y \Vert=\sqrt{2-2\cos(d_{\Sp^n}(x,y))}.$$
		Now, if we apply \cite[Lemma 2.10.40. ]{federer2014geometric} with $P:=A\times\{1\}$, there are $p\in\R^{n+1}$ and $c\geq 0$ such that,
		\begin{enumerate}
			\item For all $a\in A$, $\Vert a-p \Vert\leq c$.
			\item $p$ belongs to the convex hull of $\{a\in A: \Vert a-p\Vert=c\}$.
		\end{enumerate}
		Therefore, there are non-negative numbers $\lambda_1,\dots,\lambda_{n+2}$ and $a_1,\dots,a_{n+2}\in\{a\in A: \Vert a-x\Vert=c\}$ such that
		\begin{enumerate}
			\item $p=\sum_{i=1}^{n+2}\lambda_i a_i$.
			\item $1=\sum_{i=1}^{n+2}\lambda_i$.
		\end{enumerate}
		Hence, one can easily check $\Vert p \Vert\leq 1$. Also, since
		$$\Vert a_i-a_j\Vert=\sqrt{2-2\cos(d_{\Sp^n}(x,y))}\leq\sqrt{2-2\cos D}<\sqrt{2+\frac{2}{n+1}},$$
		$p\neq 0 $ by \cite[Lemma 1]{dubins1981equidiscontinuity}. Furthermore, for each $j\in\{1,\dots,n+2\}$,
		
		\begin{align*}
			2c^2&=\sum_{i=1}^{n+2}\lambda_i(2c^2-2\langle (a_i-p),(a_j-p)\rangle)=\sum_{i=1}^{n+2}\lambda_i\Vert (a_i-p)-(a_j-p) \Vert^2\\
			&=\sum_{i=1}^{n+2}\lambda_i\Vert a_i-a_j \Vert^2\leq\sum_{i\neq j}\lambda_i(2-2\cos(D))=(1-\lambda_j)(2-2\cos D).
		\end{align*}
		So, by summation with respect to $j$, we have $2(n+2)c^2\leq(n+1)(2-2\cos D)$. Therefore,
		$$c\leq\sqrt{\frac{(n+1)(1-\cos D)}{n+2}}<1.$$
		Finally, let $u:=\frac{p}{\Vert p \Vert}$. Then, $u\in\mathrm{Conv}_{\Sp^n}(A)$ since $p\in\mathrm{Conv}(A)$. Also, one can check that 
		$$\Vert a-u \Vert\leq\sqrt{2-2\sqrt{1-c^2}}\leq\sqrt{2-2\sqrt{\frac{1+(n+1)\cos D}{n+2}}}$$
		for all $a\in A$. This implies $d_{\Sp^n}(a,u)\leq\arccos\left(\sqrt{\frac{1+(n+1)\cos D}{n+2}}\right)=\psi(D)$ for any $a\in A$, so the proof is complete.
	\end{proof}
	
	\subsubsection{The proof of Theorem \ref{cor:Snfirsttype}}
	
	Choose a total ordering on the points of $\Sp^n$. From now on, whenvever we describe a finite subset of $\Sp^n$ by $\{x_0,\dots,x_q\}$, we suppose that $x_0<x_1<\cdots<x_q$. Let $r\in\left(0,\arccos\left(-\frac{1}{n+1}\right)\right]$. We shall associate to each $q$-simplex $\sigma:=\{x_0,\dots,x_q\}\in \vr_r(\Sp^n)$ a singular $q$-simplex $T_\sigma:\Delta_q\longrightarrow\Sp^n$. Recall that the standard Euclidean $q$-simplex $\Delta_q$ is defined in the following way:
	$$\Delta_q:=\left\{\sum_{i=0}^qt_ie_i:t_i\in [0,1]\text{ and }\sum_{i=0}^q t_i=1\right\}.$$
	This map $T_\sigma$ is defined inductively as follows: set $T(e_0)=x_0$. Suppose that $T_\sigma(z)$ is defined for $y=\sum_{i=0}^{p-1}s_ie_i$. Let $z:=\sum_{i=0}^p t_ie_i$. If $t_p=1$, we pose $T_\sigma(z)=x_p$. Otherwise, let
	$$x:=T_\sigma\left(\frac{1}{1-t_p}\sum_{i=0}^{p-1}t_i e_i\right).$$
	We define $T_\sigma(z)$ as the point on the unique shortest geodesic joining $x$ to $x_p$ with $d_{\Sp^n}(x,T_\sigma(z))=t_p\cdot d_{\Sp^n}(x,x_p)$ (the unique shortest geodesic exists since $\mathrm{Conv}_{\Sp^n}(\{x_0,\dots,x_q\})$ must be contained in some open ball of radius smaller than $\frac{\pi}{2}$ by Theorem \ref{thm:Jung}). To sum up, $T_\sigma$ is defined inductively on $\Delta_p$ for $p\leq q$ as the \emph{geodesic join} of $T_\sigma(\Delta_{p-1})$ with $x_p$.

	If $\sigma'$ is a a face of $\sigma$ of dimension $p$, we form the euclidean sub $p$-simplex $\Delta'$ of $\Delta_q$ formed by the points $\sum_{i=0}^q t_i e_i\in\Delta_q$ with $t_i=0$ if $x_i\notin\sigma'$. One can check by induction on $\dim\sigma'$ that
	\begin{equation}\label{eq:Hausmannmap}
		T_{\sigma'}=T_\sigma\vert_{\Delta'}.
	\end{equation}
	
	By (\ref{eq:Hausmannmap}), the correspondence $\sigma\mapsto T_\sigma$ gives rise to a continuous map
	$$T:\vert\vr_r(\Sp^n)\vert\longrightarrow\Sp^n.$$

	Here is a quick overview of how we will prove Theorem \ref{cor:Snfirsttype}. Through Lemma \ref{lemma:contractivesphere} (which enables the application of Hausmann's `crushings' on sufficiently small subsets of spheres), Lemma \ref{lemma:simplehomiso}, and Lemma \ref{lemma:Hausmannhomiso} we will prove that $T$ induces an isomorphism at homology level. Also, by Lemma \ref{lemma:HausmannFG}, we will prove that $T$ also induces an isomorphism at the level of fundamental groups. Finally, the proof of Theorem \ref{cor:Snfirsttype} will follow by invoking the homology Whitehead theorem.

	\begin{lemma}\label{lemma:contractivesphere}
		Let $x\in\Sp^n$, $y,z\in B_{\frac{\pi}{2}}(x,\Sp^n)$, and $\gamma_y:[0,1]\rightarrow\Sp^n$ (resp. $\gamma_z:[0,1]\rightarrow\Sp^n$) be the unique shortest geodesics from $x$ to $y$ (resp. from $x$ to $z$). Then,
		$$d_{\Sp^n}(\gamma_y(s),\gamma_z(s))\leq d_{\Sp^n}(\gamma_y(t),\gamma_z(t))$$
		for any $0\leq s\leq t\leq 1$.
	\end{lemma}
	\begin{proof}
		Let $d_{\Sp^n}(x,y)=a$ and $d_{\Sp^n}(x,z)=b$. Without loss of generality, one can assume $a\geq b$. By the spherical law of cosines, one can compute
		
		$$\cos\big(d_{\Sp^n}(\gamma_y(t),\gamma_z(t))\big)=\cos(ta)\cos(tb)+\sin(ta)\sin(tb)\cos\theta$$
		for any $t\in [0,1]$, where $\theta$ is the angle between $\gamma_y$ and $\gamma_z$ at $x$.

		Now, consider the following map
		\begin{align*}
			f:[0,1]&\longrightarrow\R_{\geq 0}\\
			t&\longmapsto\cos(ta)\cos(tb)+\sin(ta)\sin(tb)\cos\theta.
		\end{align*}
		To complete the proof, it is enough to show this $f$ is non-increasing. Observe that,
		\begin{align*}
			f'(t)&=-a\sin(ta)\cos(tb)-b\cos(ta)\sin(tb)+a\cos(ta)\sin(tb)\cos\theta+b\sin(ta)\cos(tb)\cos\theta\\
			&\leq-a\sin(ta)\cos(tb)-b\cos(ta)\sin(tb)+a\cos(ta)\sin(tb)+b\sin(ta)\cos(tb)\\
			&-(a-b)\sin(ta)\cos(tb)+(a-b)\cos(ta)\sin(tb)\\
			&=-(a-b)\sin(t(a-b))\leq 0.
		\end{align*}
		Hence, $f$ is non-increasing, so this concludes the proof.
	\end{proof}
	
	The following Lemma is an analogue of \cite[(3.3) Proposition]{h95}.
	
	\begin{lemma}\label{lemma:simplehomiso}
		Let $0<r'\leq r\leq\arccos\left(-\frac{1}{n+1}\right)$. Then the canonical inclusion $\vr_{r'}(\Sp^n)\subset\vr_r(\Sp^n)$ induces an isomorphism on homology.
	\end{lemma}
	\begin{proof}
		Let $\sigma=\{x_0,\dots,x_q\}$ be a simplex of $\vr_r(\Sp^n)$ and let $I_\sigma$ be the image of $T_\sigma$. If $\sigma'$ is a face of $\sigma$ then $I_{\sigma'}\subseteq I_\sigma$, and thus $\vr_\delta(I_{\sigma'})$ is a subcomplex of $\vr_\delta(I_\sigma)$ for all $\delta>0$. On the other hand, $\vr_\delta(I_\sigma)$ is acyclic for all $\delta>0$. Indeed, by Theorem \ref{thm:Jung}, $\exists\,u\in I_\sigma$ such that $I_\sigma\subset B_{\frac{\pi}{2}}(u,\Sp^n)$. So, one can consider the obvious crushing from $I_\sigma$ to $\{x\}$ via the shortest geoedesics. So, $\vr_\delta(I_\sigma)$ must be contractible by Lemma \ref{lemma:contractivesphere} and \cite[(2.3) Corollary]{h95}. These considerations show that for $0<\delta'\leq\delta\leq\arccos\left(-\frac{1}{n+1}\right)$, the correspondence
		$$\sigma\longmapsto \vr_{\delta'}(I_\sigma)$$
		is an acyclic carrier $\Phi_{\delta,\delta'}$ from $\vr_\delta(\Sp^n)$ to $\vr_{\delta'}(\Sp^n)$ (see \cite[\S 13]{munkres2018elements}).

		We now use the theorem of acyclic carrier (\cite[Theorem 13.3]{munkres2018elements}). This implies that there exists an augmentation preserving chain map $\nu:C_\ast(\vr_r(\Sp^n))\longrightarrow C_\ast(\vr_{r'}(\Sp^n))$ which is carried by $\Phi_{r,r'}$. Let $\mu$ denote the canonical inclusion from $\vr_{r'}(\Sp^n)$ into $\vr_r(\Sp^n)$. Then, $\phi_{r',r'}$ is an acyclic carrier for both $\nu\circ\mu_\ast$ and the indentity of $C_\ast(\vr_{r'}(\Sp^n))$. By theorem of acyclic carrier again, these two maps are chain homotopic and thus $\nu\circ\mu_\ast$ induces the identity on $\Hom_\ast(\vr_{r'}(\Sp^n))$. The same argument shows that $\mu_\ast\circ\nu$ induces the identity on $\Hom_\ast(\vr_r(\Sp^n))$ (using the acyclic carrier $\Phi_{r,r}$).
	\end{proof}
	
	We will now compare the simplicial homology of $\vr_r(\Sp^n)$ with the singular homology of $M$. Formula (\ref{eq:Hausmannmap}) shows that the correspondence $\sigma\longmapsto T_\sigma$ gives rise to a chain map
	$$T_\sharp^r:C_\ast(\vr_r(\Sp^n))\longrightarrow SC_\ast(\Sp^n)$$
	where $SC_\ast(\Sp^n)$ denotes the singular chain complex of $\Sp^n$.

	The following Lemma is an analogue of \cite[(3.4) Proposition]{h95}.
	
	\begin{lemma}\label{lemma:Hausmannhomiso}
		If $0<r\leq\arccos\left(-\frac{1}{n+1}\right)$ then the chain map $T_\sharp$ induces an isomorphism on homology.
	\end{lemma}
	\begin{proof}
		We shall need a few accessory chain complexes. For $\delta>0$, denote by $SC_\ast(\Sp^n;\delta)$ the sub-chain complexes of $SC_\ast(\Sp^n)$ based on singular simplexes $\tau$ such that there exists $u\in\Sp^n$ with the image of $\tau$ is contained in the open ball $B_\delta(u,\Sp^n)$. Recall that the inclusion $SC_\ast(\Sp^n;\delta)\longhookrightarrow SC_\ast(\Sp^n)$ induces an isomorphism on homology (\cite[Theorem 31.5]{h95}).

		We shall also use the ordered chain complex $C_\ast'(\vr_r(\Sp^n))$: the group $C_q'(\vr_r(\Sp^n))$ is free abelian group on $(q+1)$-tuples $(x_0,\dots,x_q)$ such that $\{x_0\}\cup\cdots\cup\{x_q\}$ is a simplex of $\vr_r(\Sp^n)$. One can view that $C_\ast(\vr_r(\Sp^n))$ as a sub-chain complex of $C_\ast'(\vr_r(\Sp^n))$ by associating a $q$-simplex $\{x_0,\dots,x_q\}$ of $\vr_r(\Sp^n)$ (with our convention that $x_0<x_1<\dots<x_q$ for the well-ordering on $\Sp^n$) the $(q+1)$-tuple $(x_0\dots,x_q)$. It is also classical that this inclusion is homology isomorphism (\cite[Theorem 3.6]{h95}). Observe that the construction $\sigma\mapsto T_\sigma$ does not require that the vertices of $\sigma$ are all distinct. One can then define $T_\sigma$ for a basis element of $C_\ast'(\vr_r(\Sp^n))$ and thus extend to a chain map $T_\sharp^r:C_\ast'(\vr_r(\Sp^n))\longrightarrow SC_\ast(\Sp^n;\psi(r))$. Now, choose $r'<r$ such that $\psi(r')\leq\frac{r}{2}$. One then has the following commutative diagram:
		$$\begin{tikzcd}
			C_\ast'(\vr_{r'}(\Sp^n)) \arrow[r, "T_\sharp^{r'}"] \arrow[d] & SC_\ast(\Sp^n;\psi(r'))\arrow[d] \\
			C_\ast'(\vr_r(\Sp^n)) \arrow[r, "T_\sharp^r"] & SC_\ast(\Sp^n;\psi(r))
		\end{tikzcd}$$
		Let $\tau:\Delta_q\longrightarrow\Sp^n$ be a singular simplex whose image is contained in some open ball of radius $\psi(r')$. The $(q+1)$-tuple $(\tau(e_0),\dots,\tau(e_q))$ is element of $C_q'(\vr_r(\Sp^n))$. This correspondence gives rise to a chain map
		$$R:SC_\ast(\Sp^n;\psi(r'))\longrightarrow C_\ast'(\vr_r(\Sp^n))$$
		The composition $R\circ T_\sharp^{r'}$ is equal to the canonical inclusion $C_\ast'(\vr_{r'}(\Sp^n))\subset C_\ast'(\vr_r(\Sp^n))$ which induces a homotopy isomorphism by Lemma \ref{lemma:simplehomiso}. Let us now understand the composition $T_\sharp^r\circ R:SC_\ast(\Sp^n;\psi(r'))\rightarrow SC_\ast(\Sp^n;\psi(r))$. Let $\tau:\Delta_q\longrightarrow\Sp^n$ be a singular simplex such that $\tau(\Delta_q)\subset B_{\psi(r')}(y,\Sp^n)$ for some $y\in\Sp^n$. Therefore, $\tau':=T_\sharp^r\circ R(\tau)$ also satisfies $\tau'(\Delta_q)\subset B_{\psi(r')}(y,\Sp^n)$ since $\psi(r')<\frac{\pi}{2}$. Hence, $\tau$ and $\tau'$ are canonically homotopic (following, for each $s\in\Delta_q$, the shortest geodesic joining $\tau(s)$ to $\tau'(s)$). As in the proof of the homotopy axiom for singular homology (\cite[\S 30]{h95}), these provide a chain homotopy between $T_\sharp^r\circ R$ and the inclusion $SC_\ast(\Sp^n;\psi(r'))\subset SC_\ast(\Sp^n;\psi(r))$. As said before, this inclusion is known to induce a homology isomorphism. Therefore, $T_\sharp^r\circ R$ induces an isomorphism on homology.

		We have shown that both $R\circ T_\sharp^{r'}$ and $T_\sharp^r\circ R$ induce homology isomorphisms. Therefore, $T_\sharp^r$ induces a morphism both injective and surjective, hence homology isomorphism.
	\end{proof}
	
	\begin{lemma}\label{lemma:HausmannFG}
		If $0<r\leq\arccos\left(-\frac{1}{n+1}\right)$, the map 
		$$T:\vert \vr_r(\Sp^n)\vert\longrightarrow\Sp^n$$
		induces an isomorphism on the fundamental groups.
	\end{lemma}
	\begin{proof}
		Let $\gamma:[0,1]\longrightarrow\Sp^n$ represent an element of $\pi_1(\Sp^n)$. Choose large enough positive integer $N$such that $\frac{1}{N}$ is smaller than the Lebesgue number for the covering $\{\gamma^{-1}(B_{\frac{r}{2}}(x,\Sp^n))\}_{x\in\Sp^n}$. Then, $d_{\Sp^n}\left(\gamma\left(\frac{k}{N}\right),\gamma\left(\frac{k+1}{N}\right)\right)<r$ for any $k=0,\dots,N-1$. Hence the path $\gamma\vert_{\left[\left(\frac{k}{N}\right),\left(\frac{k+1}{N}\right)\right]}$ is then canonically homotopic to a parametrization of the shortest geodesic joining $\gamma\left(\frac{k}{N}\right)$ to $\gamma\left(\frac{k+1}{N}\right)$. This shows that $\gamma$ is homotopic to a composition $\gamma'$ of geodesics in open balls of radius $\frac{r}{2}$. Such a path $\gamma'$ represents the image of $T$ of an element of $\pi_1(\vert\vr_r(\Sp^n)\vert)$, the latter being identified with the edge-path group of the simplicial complex $\vr_r(\Sp^n)$ (\cite[p.134-139]{spanier1989algebraic}). This proves that $\pi_1 T:\pi_1(\vert\vr_r(\Sp^n)\vert)\longrightarrow\pi_1(\Sp^n)$ is surjective.

		Now, to prove the injectivity, suppose $\pi_1 T([\alpha])=0$ where $\alpha:[0,1]\longrightarrow \vert\vr_r(\Sp^n)\vert$ is a continuous map satisfying $\alpha(0)=\alpha(1)$. Moreover, again by \cite[p.134-139]{spanier1989algebraic}, one can assume $\alpha$ is induced by an edge-path of $\vr_r(\Sp^n)$. In other words, there are a positive integer $N$, and $x_0,\dots,x_{N-1},x_N=x_0\in\Sp^n$ such that $d_{\Sp^n}(x_i,x_{i+1})<r$ and $\alpha\left(\frac{i}{N}\right)=x_i$ for $i=0,\dots,N-1$ (here, we view $x_i$ as $0$-simplex). Next, by the assumption, $[T\circ\alpha]=\pi_1 T([\alpha])=0$. This implies that, there is a homotopy map $H:[0,1]\times [0,1]\longrightarrow\Sp^n$ such that $H(t,1)=T\circ\alpha(t)$ and $H(t,0)=H(0,s)=H(1,s)=x_0$ for any $t,s\in [0,1]$. Next, choose large enough positive integer $N'$ so that if we triangulate $[0,1]\times [0,1]$ with vertices $\left(\frac{k}{N'},\frac{l}{N'}\right)$ for $k,l=0,\dots,N'$, each triangle is contained in one of $\{H^{-1}(B_{\frac{r}{2}}(x,\Sp^n))\}_{x\in\Sp^n}$. It means that, $d_{\Sp^n}\left(H\left(\frac{k}{N'},\frac{l}{N'}\right),H\left(\frac{k'}{N'},\frac{l'}{N'}\right)\right)<r$ whenever $\left(\left(\frac{k}{N'},\frac{l}{N'}\right),\left(\frac{k'}{N'},\frac{l'}{N'}\right)\right)$ is an edge of the triangulation . Because of this observation, one can prove that the edge path $H(0,1),H\left(\frac{1}{N'},1\right),\dots,H\left(\frac{N'-1}{N'},1\right),H(1,1)$ is equivalent to $x_0$. Also, it is easy to check that two edge path $H(0,1),H\left(\frac{1}{N'},1\right),\dots,H\left(\frac{N'-1}{N'},1\right),H(1,1)$ and $x_0,x_1,\dots,x_{N-1},x_N$ are equivalent. This means that $[\alpha]=0$. So, $\pi_1T$ is injective.
	\end{proof}
	
	We are now in position to state the proof of Theorem \ref{cor:Snfirsttype}.
	
	\begin{proof}[Proof of Theorem \ref{cor:Snfirsttype}]
		As mentioned in the beginning of this section, one can assume $n\geq 2$. Hence, $\Sp^n$ is simply connected. Therefore, by Lemma \ref{lemma:HausmannFG}, $\vert \vr_r(\Sp^n) \vert$ is also simply connected. Also, by Lemma \ref{lemma:Hausmannhomiso} and the isomorphism between simplicial and singular homology \cite[\S 34]{munkres2018elements}, $T$ induces an isomorphism on homology. Therefore, $T$ is homotopy equivalence by \cite[Corollary 4.33]{h01}.
	\end{proof}

	
	\subsection{Simplicial proof of Theorem \ref{thm:intervalform}}\label{sec:app:otherpfintvltype}
	
	In this section, we will provide an alternative proof of Theorem \ref{thm:intervalform} which takes place at the level of simplicial complexes.
	
	\begin{lemma}\label{lemma:simplicialleft-right}
		Suppose that a compact metric space $(X,d_X)$, a field $\mathbb{F}$, and a nonnegative integer $k$ are given. Then, for every $I\in\dgmR_k(X;\mathbb{F})$,
		\begin{description}
			\item[(i)] if $u\in[0,\infty)$ is the left endpoint of $I$, then $u\notin I$ (i.e. $I$ is left-open).
			
			\item[(ii)] if $v\in[0,\infty)$ is the right endpoint of $I$, then $v\in I$ (i.e. $I$ is right-closed).
		\end{description}
	\end{lemma}
	\begin{proof}[Proof of \textbf{\emph{(i)}}]
		The fact that $I\in\dgmR_k(X;\mathbb{F})$ implies that, for each $r\in I$, there exists a simplicial $k$-cycle $c_r$ on $\vr_r(X)$ with coefficients in $\mathbb{F}$ satisfying the following:
		
		\begin{enumerate}
			\item $[c_r]\in \mathrm{H}_k(\vr_r(X);\mathbb{F})$ is nonzero for any $r\in I$.
			\item $(i_{r,s})_\ast([c_r])=[c_s]$ for any $r\leq s$ in $I$.
		\end{enumerate}
		
		Now, suppose that $u$ is a closed left endpoint of $I$ (so, $u\in I$).  In particular, by the above there exists a simplicial $k$-cycle $c_u$ on $\vr_u(X)$ with coefficients in $\mathbb{F}$ with the above two properties. Then, $c_u=\sum_{i=1}^l\alpha_i\sigma_i$ where $\alpha_i\in \mathbb{F}$ and $\sigma_i$ is a subset of $X$ with cardinality $k+1$ and $\diam(\sigma_i)<u$ for any $i=1,\dots,l$. Observe that one can choose small $\varepsilon>0$ such that $\diam(\sigma_i)<u-\varepsilon$ for any $i=1,\dots,l$. Therefore, if we define $c_{u-\varepsilon}:=\sum_{i=1}^l\alpha_i\sigma_i$, it is also a simplicial $k$-cycle on $\vr_{u-\varepsilon}(X)$. Also, it satisfies $(i_{u-\varepsilon,u})_\sharp(c_{u-\varepsilon})=c_u$ obviously since $i_{u-\varepsilon,u}$ is the canonical inclusion.
		
		Moreover,  $c_{u-\varepsilon}$ cannot be null-homologous. Otherwise, there would exist a simplicial $(k+1)$-chain $d_{u-\varepsilon}$ of $\vr_{u-\varepsilon}(X)$ with coefficients in $\mathbb{F}$ such that $\partial_{k+1}^{(u-\varepsilon)} d_{u-\varepsilon}=c_{u-\varepsilon}$. However, this would imply that $$\partial_{k+1}^{(u)}\circ(i_{u-\varepsilon,u})_\sharp(d_{u-\varepsilon})=(i_{u-\varepsilon,u})_\sharp\circ\partial_{k+1}^{(u-\varepsilon)}(d_{u-\varepsilon})=(i_{u-\varepsilon,u})_\sharp(c_{u-\varepsilon})=c_u$$
		because of the naturality of the boundary operators $\partial_{k+1}^{(u-\varepsilon)}$ and $\partial_{k+1}^{(u)}$. This would in turn contradict  the property $[c_u]\neq 0$.

		So, we must have $[c_{u-\varepsilon}]\neq 0$. But, the existence of such $c_{u-\varepsilon}$ contradicts  the fact that $u$ is the left endpoint of $I$. Therefore, one  concludes that $u$ cannot be a  closed left endpoint, so it must be an open endpoint.
	\end{proof}
	
	\begin{proof}[Proof of \textbf{\emph{(ii)}}]
		Now, suppose that $v$ is an open right endpoint of $I$ (so that $v\notin I$ and therefore $c_v$ is not defined by the above two conditions). Choose small enough $\varepsilon>0$ so that $v-\varepsilon\in I$, and let
		$$c_v:=(i_{v-\varepsilon,v})_\sharp(c_{v-\varepsilon}).$$
		Then, $c_v$ must be null-homologous.

		This means that there exists a simplicial $(k+1)$-dimensional chain $d_v$ of $\vr_v(X)$ with coefficients in $\mathbb{F}$ such that $\partial_{k+1}^{(v)} d_v=c_v$. Then, $d_v=\sum_{i=1}^l\alpha_i\tau_i$ where $\alpha_i\in \mathbb{F}$ and $\tau_i$ is a subset of $X$ with cardinality $k+2$ and $\diam(\tau_i)<v$ for any $i=1,\dots,l$. Observe that one can choose $\varepsilon'\in (0,\varepsilon]$ such that $\diam(\tau_i)<v-\varepsilon'$ for any $i=1,\dots,l$. Therefore, if we define $d_{v-\varepsilon'}:=\sum_{i=1}^l\alpha_i\tau_i$, it is also a simplicial $(k+1)$-dimensional chain on $\vr_{v-\varepsilon'}(X)$. Also, it satisfies $(i_{v-\varepsilon',v})_\sharp(d_{v-\varepsilon'})=d_v$. Then, again by the naturality of boundary operators,
		
		$$(i_{v-\varepsilon',v})_\sharp\circ\partial_{k+1}^{(v-\varepsilon')}(d_{v-\varepsilon'})=\partial_{k+1}^{(v)}\circ(i_{v-\varepsilon',v})_\sharp(d_{v-\varepsilon'})=\partial_{k+1}^{(v)} d_v=c_v.$$
		
		Since $(i_{v-\varepsilon',v})_\sharp$ is injective and $(i_{v-\varepsilon',v})_\sharp\circ (i_{v-\varepsilon,v-\varepsilon'})_\sharp (c_{v-\varepsilon})=(i_{v-\varepsilon,v})_\sharp(c_{v-\varepsilon})=c_v$, one can conclude that $\partial_{k+1}^{(v-\varepsilon')}(d_{v-\varepsilon'})=(i_{v-\varepsilon,v-\varepsilon'})_\sharp (c_{v-\varepsilon})$. This indicates that 
		$$0=[(i_{v-\varepsilon,v-\varepsilon'})_\sharp (c_{v-\varepsilon})]=(i_{v-\varepsilon,v-\varepsilon'})_\ast ([c_{v-\varepsilon}])=[c_{v-\varepsilon'}],$$
		but it contradicts the fact that $[c_{v-\varepsilon'}]\neq 0$. Therefore, $v$ must be a closed endpoint.
	\end{proof}
	
	Now, Theorem \ref{thm:intervalform} is achieved as a direct result of Lemma \ref{lemma:simplicialleft-right}.
\end{document}